\documentclass[10pt]{amsart}
\usepackage{times,amsmath,amsbsy,amssymb,amscd,mathrsfs}
\usepackage{slashbox}\usepackage{graphicx,subfigure,epstopdf,wrapfig,chemarrow}

\usepackage{algorithm2e} 
\usepackage{multicol,multirow}
\usepackage{mathtools}
\usepackage[usenames,dvipsnames,svgnames,table]{xcolor}
\usepackage[all]{xy}
\usepackage{wrapfig}
\usepackage{tcolorbox}

\usepackage{tikz,tikz-cd}
\usepackage[utf8]{inputenc}
\usepackage{pgfplots} 
\usepackage{pgfgantt}
\usepackage{pdflscape}
\pgfplotsset{compat=newest} 
\pgfplotsset{plot coordinates/math parser=false}
\newlength\fwidth

\usepackage[numbered]{mcode}

\definecolor{myBlue}{rgb}{0.0,0.0,0.55}
\usepackage[pdftex,colorlinks=true,citecolor=myBlue,linkcolor=myBlue]{hyperref}

\usepackage[hyperpageref]{backref}

\usepackage{comment,enumerate,multicol,xspace}

  \newcounter{mnote}
  \let\oldmarginpar\marginpar
    \renewcommand\marginpar[1]{\-\oldmarginpar[\raggedleft\footnotesize #1]%
    {\raggedright\footnotesize #1}}

\newtheorem{theorem}{Theorem}[section]
\newtheorem{lemma}[theorem]{Lemma}
\newtheorem{corollary}[theorem]{Corollary}
\newtheorem{proposition}[theorem]{Proposition}

\newtheorem{example}[theorem]{Example}

\newtheorem{remark}[theorem]{Remark}

\newcommand{\dx}{\,{\rm d}x}
\newcommand{\dd}{\,{\rm d}}

\newcommand{\bs}{\boldsymbol}

\DeclareMathOperator*{\spa}{span}

\newcommand{\curl}{{\rm curl\,}}
\renewcommand{\div}{\operatorname{div}}
\newcommand{\grad}{{\rm grad\,}}

\DeclareMathOperator{\dist}{dist}
\newcommand{\tr}{\operatorname{tr}}

\newcommand{\step}[1]{\noindent\raisebox{1.5pt}[10pt][0pt]{\tiny\framebox{$#1$}}\xspace}

\newcommand{\vertiii}[1]{{\left\vert\kern-0.25ex\left\vert\kern-0.25ex\left\vert #1 
    \right\vert\kern-0.25ex\right\vert\kern-0.25ex\right\vert}}

\usepackage{slashbox}
\usepackage{booktabs}
\usepackage{stmaryrd}
\usepackage{scalefnt}
\usepackage{graphicx}
\allowdisplaybreaks[1]
\newcommand{\Oplus}{\ensuremath{\vcenter{\hbox{\scalebox{1.5}{$\oplus$}}}}}

\begin{document}
\title{Finite Element de Rham and Stokes Complexes in Three Dimensions}

 \author{Long Chen}%
 \address{Department of Mathematics, University of California at Irvine, Irvine, CA 92697, USA}%
 \email{chenlong@math.uci.edu}%
 \author{Xuehai Huang}%
 \address{School of Mathematics, Shanghai University of Finance and Economics, Shanghai 200433, China}%
 \email{huang.xuehai@sufe.edu.cn}%

 \thanks{The first author was supported by NSF DMS-1913080 and DMS-2012465.}
 \thanks{The second author is the corresponding author. The second author was supported by the National Natural Science Foundation of China Project 12171300, the Natural Science Foundation of Shanghai 21ZR1480500, and the Fundamental Research Funds for the Central
Universities 2019110066.}

\makeatletter
\@namedef{subjclassname@2020}{\textup{2020} Mathematics Subject Classification}
\makeatother


\begin{abstract}
Finite element de Rham complexes and finite element Stokes complexes with various smoothness in three dimensions are systematically constructed.
First smooth scalar finite elements in three dimensions are derived through a non-overlapping decomposition of the simplicial lattice.
Based on the smooth scalar finite elements, both $H(\div)$-conforming finite elements and $H(\curl)$-conforming finite elements with various smoothness are devised, which induce the finite element de Rham complexes with various smoothness and the associated commutative diagrams.
The div stability is established for the $H(\div)$-conforming finite elements, and the exactness of these finite element complexes.
\end{abstract}

\maketitle


\section{Introduction}
Hilbert complexes play a fundamental role in the theoretical analysis and the design of stable numerical methods for partial differential equations~\cite{ArnoldFalkWinther2006,Arnold;Falk;Winther:2010Finite,Arnold:2018Finite,ChenHuang2018}. In~\cite{Chen;Huang:2022femcomplex2d}, we have constructed two-dimensional finite element complexes with various smoothness, including finite element de Rham complexes, finite element elasticity complexes, and finite element divdiv complexes etc.
In this work we shall deal with a more challenging problem, that is to construct finite element de Rham complexes and finite element Stokes complexes with various smoothness in three dimensions. 

Introduce the following Sobolev spaces on a domain $\Omega\subseteq \mathbb R^3$
\begin{equation*}
\begin{aligned}
H^1(\Omega) &= \{\phi \in L^2(\Omega) : \grad \phi \in \bs L^2(\Omega;\mathbb R^3) \},\\
\bs H(\curl,\Omega) &= \{\bs u \in  \bs L^2(\Omega;\mathbb R^3): \curl \bs u \in \bs L^2(\Omega;\mathbb R^3) \},\\
\bs H(\grad\curl,\Omega) &= \{\bs u \in  \bs L^2(\Omega;\mathbb R^3): \curl \bs u \in \bs H^1(\Omega;\mathbb R^3) \},\\
\bs H^1(\curl,\Omega) &= \{\bs u \in  \bs H^1(\Omega;\mathbb R^3): \curl \bs u \in \bs H^1(\Omega;\mathbb R^3) \},\\
\bs H(\div,\Omega) &= \{\bs u \in\bs  L^2(\Omega;\mathbb R^3): \div \bs u\in L^2(\Omega) \}.
\end{aligned}
\end{equation*}
%
The de Rham complex reads as
\begin{equation}\label{eq:deRham}
\mathbb R\hookrightarrow{} H^1(\Omega)\xrightarrow{\grad}\bs H(\curl,\Omega)\xrightarrow{\curl}\bs H(\div;\Omega)\xrightarrow{\div}L^2(\Omega) \xrightarrow{}0. 
\end{equation}
The Stokes complexes read as
\begin{equation}\label{eq:Stokes1}
\mathbb R\hookrightarrow{} H^1(\Omega)\xrightarrow{\grad}\bs H(\grad\curl,\Omega)\xrightarrow{\curl}\bs H^1(\Omega; \mathbb R^3)\xrightarrow{\div}L^2(\Omega) \xrightarrow{}0, 
\end{equation}
\begin{equation}\label{eq:Stokes2}
\mathbb R\hookrightarrow{} H^2(\Omega)\xrightarrow{\grad}\bs H^1(\curl,\Omega)\xrightarrow{\curl}\bs H^1(\Omega; \mathbb R^3)\xrightarrow{\div}L^2(\Omega) \xrightarrow{}0. 
\end{equation}
For simplicity, we assume $\Omega$ is homeomorphic to a ball and thus the de Rham complex \eqref{eq:deRham} and Stokes complexes \eqref{eq:Stokes1}-\eqref{eq:Stokes2} are exact. That is
$$
\mathbb R = \ker(\grad), \ker(\curl) = {\rm img}(\grad), \ker(\div) = {\rm img}(\curl),
$$
and $\div$ is surjective.
	 
We first construct smooth finite elements $\mathbb V^{\grad}$ and $\mathbb V^{L^2}$ for the scalar function space $H^1(\Omega)$ and $L^2(\Omega)$, respectively. Following our recent work~\cite{Chen;Huang:2022femcomplex2d} introduce the simplicial lattice $\mathbb T^{3}_k = \big \{ \alpha = (\alpha_0, \ldots, \alpha_3)\in\mathbb N^{0:3} \mid  \alpha_0  + \ldots + \alpha_3 = k \big \}$. The Bernstein basis for polynomial is $\{\lambda^{\alpha}, \alpha \in \mathbb T^{3}_k\}$. Given integer vector $\boldsymbol r= (r^{\texttt{v}}, r^e, r^f)^{\intercal}$ with $r^f\geq-1$, $r^e\geq \max\{2r^f,-1\}$ and $r^{\texttt{v}}\geq \max\{2r^e,-1\}$, and integer $k\geq \max\{2r^{\texttt{v}}+1,0\}$, we derive a geometric decomposition of the polynomial space $\mathbb P_{k}(T)$ based on a partition of the simplicial lattice $\mathbb T^{3}_k$.
With such geometric decomposition, we construct $C^{r^f}$-continuous finite element spaces $\mathbb V_k(\bs r)$ with $C^{r^{\texttt{v}}}$-smoothness at vertices, $C^{r^e}$-smoothness on edges, and $C^{r^f}$-smoothness on faces. Here $C^{-1}$-smoothness means discontinuity. Therefore if $r^f = -1$, $\mathbb V_k(\bs r)\subset L^2(\Omega)$ and for $r^f \geq 0$, $\mathbb V_k(\bs r)\subset H^1(\Omega)$.

We then construct $H(\div)$-conforming finite element space $\mathbb V^{\div}_{k}(\boldsymbol{r}_2) = \mathbb V_k^3(\bs r_2) \cap \bs H(\div,\Omega)$ and establish the div stability between finite elements spaces
\begin{equation*}
\div\mathbb V^{\div}_{k}(\boldsymbol{r}_2)=\mathbb V^{L^2}_{k-1}(\bs r_2\ominus 1),   
\end{equation*}
where $r\ominus n := \max\{r-n, -1\}$. 
The key is to prove the div stability of bubble spaces
\begin{equation*}
\div\mathbb B_{k}^{\div}(T;\boldsymbol{r}_2)=\mathbb B_{k-1}(T; \bs r_2\ominus 1)/\mathbb R,
\end{equation*}
see \eqref{eq:bubbleT} and \eqref{eq:bubbledecomprfn1} for the precise definition of these bubble spaces and \eqref{eq:boundr2fordivbubble} for the condition on the parameter $\bs r_2$. 
%
We further construct $H(\div)$-conforming finite element spaces $\mathbb V^{\div}_{k}(\bs r_2, \bs r_3)$ with an inequality constraint on the smoothness parameters $\bs r_3\geq\bs r_2\ominus 1$, and prove the div stability
$$
\div \mathbb V^{\div}_{k}(\bs r_2, \bs r_3) = \mathbb V^{L^2}_{k-1}(\bs r_3). 
$$
When $r_2^f \geq 0$, $\mathbb V^{\div}_{k}(\bs r_2, \bs r_3)\subset \bs H^1(\Omega)$ and $(\mathbb V^{\div}_{k}(\bs r_2, \bs r_3) , \mathbb V^{L^2}_{k-1}(\bs r_3))$ is a stable finite element velocity-pressure pair for the Stokes equation. 

By the aid of the degrees of freedom (DoFs) of spaces $\mathbb V^{\grad}_{k+2}(\bs r_0)$ and $\mathbb V^{\div}_{k}(\bs r_2, \bs r_3)$, we construct $H(\curl)$-conforming finite element spaces $\mathbb V_{k+1}^{\curl}(\bs r_1, \bs r_2) = \{\bs v\in \mathbb V_{k+1}^{\curl}(\bs r_1) \mid \curl \bs v\in \mathbb V_k^{\div}(\bs r_2) \cap \ker(\div)\}$. When identifying DoFs, we keep DoFs for $\curl \bs v\in \mathbb V^{\div}_k(\bs r_2)$, combine DoFs for $\mathbb V^3_{k+1}(\bs r_1)$, and eliminate linear dependent ones.

Given $\bs r_0 \geq 0, \bs r_1 = \bs r_0 -1, \bs r_2\geq\bs r_1\ominus1, \bs r_3\geq \bs r_2\ominus1$, under the assumptions $k\geq\max\{2 r_1^{\texttt{v}} + 1,2 r_2^{\texttt{v}} + 1,2 r_3^{\texttt{v}} + 2,1\} 
$, $(\bs r_2, \bs r_3)$ being a div stable pair, and
$$
r_1^{\texttt{v}}\geq2r_1^e+1,\; r_1^{e}\geq2r_1^f+1,\; r_2^{\texttt{v}}\geq2r_2^e,\; r_2^{e}\geq2r_2^f, \; r_3^{\texttt{v}}\geq2r_3^e,\; r_3^{e}\geq2r_3^f,
$$
we acquire the finite element de Rham complexes with various smoothness in three dimensions
\begin{equation}\label{intro:femderhamcomplex3dgeneral}
	\mathbb R\xrightarrow{\subset}\mathbb V^{\grad}_{k+2}(\boldsymbol{r}_0)\xrightarrow{\grad}\mathbb V^{\curl}_{k+1}(\boldsymbol{r}_1,\boldsymbol{r}_2)\xrightarrow{\curl}\mathbb V^{\div}_{k}(\boldsymbol{r}_2,\boldsymbol{r}_3)\xrightarrow{\div}\mathbb V^{L^2}_{k-1}(\boldsymbol{r}_3)\to0,
\end{equation}
and 
the following commutative diagram
$$
\begin{array}{c}
\xymatrix{
\mathbb R \ar[r]^-{\subset} & \mathcal C^{\infty}(\Omega) \ar[d]^{I_h^{\grad}} \ar[r]^-{\grad}                & \mathcal C^{\infty}(\Omega;\mathbb R^3) \ar[d]^{I_h^{\curl}}   \ar[r]^-{\curl} & \mathcal C^{\infty}(\Omega;\mathbb R^3) \ar[d]^{I_h^{\div}}   \ar[r]^-{\div} & \ar[d]^{I_h^{L^2}} \mathcal C^{\infty}(\Omega) \ar[r]^{} & 0 \\
 \mathbb R \ar[r]^-{\subset} & \mathbb V_{k+2}^{\grad}(\bs r_0) \ar[r]^{\grad}
                &  \mathbb V_{k+1}^{\curl}(\bs r_1, \bs r_2)   \ar[r]^{\curl} & \mathbb V_{k}^{\div}(\bs r_2, \bs r_3)  \ar[r]^{\div} &  \mathbb V_{k-1}^{L^2}(\bs r_3) \ar[r]^{}& 0,}
\end{array}
$$
where $I_h^{\grad}, I_h^{\curl}, I_h^{\div}$, and $I_h^{L^2}$ are the canonical interpolation operators using the DoFs. 

When $r_2^f\geq 0$, space $\mathbb V_{k}^{\div}(\bs r_2, \bs r_3)\subset \bs H^1(\Omega)$ and $ \mathbb V_{k+1}^{\curl}(\bs r_1, \bs r_2) \subset \bs H(\grad \curl,\Omega)$. Therefore \eqref{intro:femderhamcomplex3dgeneral} becomes a finite element Stokes complex. Existing works on finite element Stokes complexes~\cite{Neilan2015} and finite element de Rham complexes~\cite{Christiansen;Hu;Hu:2018finite} are examples of \eqref{intro:femderhamcomplex3dgeneral}.  
We refer to~\cite{ChristiansenHu2018,FuGuzmanNeilan2020,HuZhangZhang2022} for some discrete Stokes complexes based on split meshes, whose shape functions are piece-wise polynomials.

The developed tools (simplicial lattice and barycentric calculus) and the approach to construct $\mathbb V^{\curl}_{k+1}(\boldsymbol{r}_1,\boldsymbol{r}_2)$ and $\mathbb V^{\div}_{k}(\boldsymbol{r}_2,\boldsymbol{r}_3)$ will also shed light on the unified construction of finite element differential complexes other than de Rham and Stokes complexes.

The rest of this paper is organized as follows. The simplical lattice and barycentric calculus are introduced in Section~\ref{sec:simplicallattice}. 
In Section~\ref{sec:geodecomp3d}, the geometric decomposition of $C^m$-conforming finite elements in three dimensions and $H(\div)$-conforming finite elements are studied. In Section~\ref{sec:divstability}, the div stability is proved, and smooth $H(\div)$-conforming finite elements are constructed.
Finite element de Rham complexes with various smoothness are devised in Section~\ref{sec:femderhamcomplex}. Smooth scalar finite elements in arbitrary dimension are constructed in Appendix~\ref{sec:geodecompnd}.

\section{Simplical Lattice and Barycentric Calculus}\label{sec:simplicallattice}
Let $T \in \mathbb{R}^{n}$ be an $n$-dimensional simplex with vertices $\texttt{v}_{0}, \texttt{v}_{1}, \ldots, \texttt{v}_{n}$ in general position. That is
$$
T = \left \{ \sum_{i=0}^n \lambda_i \texttt{v}_i \mid 0\leq \lambda_i \leq 1, \sum_{i=0}^n\lambda_i = 1 \right \},
$$
where $\lambda = (\lambda_0, \lambda_1, \ldots, \lambda_n) $ is called the barycentric coordinate, and volume of $T$ is non-zero. We will write $T = {\rm Convex}(\texttt{v}_0, \ldots, \texttt{v}_n)$, where ${\rm Convex}$ stands for the convex combination. Some content of this section can be found in the book~\cite{Lai;Schumaker:2007Spline} but with different notation. Here notation on sub-simplexes are adapted from~\cite{ArnoldFalkWinther2009}.

\subsection{The simplicial lattice}
For two non-negative integers $l\leq m$, we will use the multi-index notation $\alpha \in \mathbb{N}^{l:m}$, meaning $\alpha=\left(\alpha_{l}, \cdots, \alpha_{m}\right)$ with integer $\alpha_{i} \geqslant 0$. The length of a multi-index is $|\alpha|:=\sum_{i=l}^m \alpha_{i}$ for $\alpha\in \mathbb{N}^{l:m}$. We can also treat $\alpha$ as a row vector with integer valued coordinates. We use the convention that: a vector $\alpha \geq c$ means $\alpha_i \geq c$ for all components $i=0,1, \ldots, n$.
We define $\lambda^{\alpha}:=\lambda_{0}^{\alpha_{0}} \cdots \lambda_{n}^{\alpha_{n}}$ for $\alpha\in \mathbb{N}^{0: n}$.

A simplicial lattice of degree $k$ and dimension $n$ is a multi-index set of $n+1$ components and with fixed length $k$, i.e.,
$$
\mathbb T^{n}_k = \left \{ \alpha = (\alpha_0, \alpha_1, \ldots, \alpha_n)\in\mathbb N^{0:n} \mid  \alpha_0 + \alpha_1 + \ldots + \alpha_n = k \right \}.
$$
An element $\alpha\in \mathbb T^{n}_k$ is called a node of the lattice. 
It holds that
$$| \mathbb T^{n}_k | = {n + k \choose k} = \dim \mathbb P_k(T),$$
where $\mathbb P_k(T)$  denotes the set of real valued polynomials defined on $T$ of degree less than or equal to $k$.
Indeed the Bernstein basis of $\mathbb P_k(T)$ is
$$
\{ \lambda^{\alpha}: = \lambda_0^{\alpha_0}\lambda_1^{\alpha_1}\ldots \lambda_n^{\alpha_n} \mid \alpha \in \mathbb T^{n}_k\}.
$$
For a subset $S\subseteq \mathbb T^{n}_k$, we define
$$
\mathbb P_k(S) = \spa \{ \lambda^{\alpha}, \alpha \in S\subseteq \mathbb T^{n}_k \}.
$$
With such one-to-one mapping between the lattice node $\alpha$ and the Bernstein polynomial $\lambda^{\alpha}$, we can study properties of polynomials through the simplicial lattice.
%


\subsection{Geometric embedding of a simplicial lattice}
We can embed the simplicial lattice into a geometric simplex by using $\alpha/k$ as the barycentric coordinate of node $\alpha$. Given $\alpha\in \mathbb T^{n}_k$, the barycentric coordinate of $\alpha$ is given by
$$
\lambda(\alpha) = (\alpha_0, \alpha_1, \ldots, \alpha_n )/k.
$$
Let $T$ be a simplex with vertices $\{\texttt{v}_0, \texttt{v}_1, \ldots, \texttt{v}_n\}$. The geometric embedding is
$$
x: \mathbb T^{n}_k \to T, \quad x(\alpha) = \sum_{i=0}^n \lambda_i(\alpha) \texttt{v}_i. 
$$
We will always assume such a geometric embedding of the simplicial lattice exists and write as $\mathbb T^{n}_k(T)$. 


A simplicial lattice $\mathbb T^{n}_k$ is, by definition, an algebraic set. Through the geometric embedding $\mathbb T^{n}_k(T)$, we can use operators for the geometric simplex $T$ to study this  algebraic set. For example, for a subset $S\subseteq T$, we use $\mathbb T^{n}_k(S) = \{ \alpha \in \mathbb T^{n}_k, x(\alpha)\in S\}$ to denote the portion of lattice nodes whose geometric embedding is inside $S$. The superscript ${}^n$ will be replaced by the dimension of $S$ when $S$ is a lower dimensional sub-simplex. 

\subsection{Sub-simplicial lattices}
Following~\cite{ArnoldFalkWinther2009}, we let $\Delta(T)$ denote all the subsimplices of $T$, while $\Delta_{\ell}(T)$ denotes the set of subsimplices of dimension $\ell$, for $0\leq \ell \leq n$. The cardinality of $\Delta_{\ell}(T)$ is $\displaystyle{n+1\choose \ell+1}$. Elements of $\Delta_0(T) = \{\texttt{v}_0, \ldots, \texttt{v}_n\}$ are $n+1$ vertices of $T$ and $\Delta_n(T) = T$.

For a sub-simplex $f\in \Delta_{\ell}(T)$, we will overload the notation $f$ for both the geometric simplex and the algebraic set of indices. Namely $f = \{f(0), \ldots, f(\ell)\}\subseteq \{0, 1, \ldots, n\}$ and 
$$
f ={\rm Convex}(\texttt{v}_{f(0)}, \ldots, \texttt{v}_{f(\ell)}) \in \Delta_{\ell}(T)
$$
is the $\ell$-dimensional simplex spanned by the vertices $\texttt{v}_{f(0)}, \ldots, \texttt{v}_{f( \ell)}$.

If $f \in \Delta_{\ell}(T)$, then $f^{*} \in \Delta_{n- \ell-1}(T)$ denotes the sub-simplex of $T$ opposite to $f$. When treating $f$ as a subset of $\{0, 1, \ldots, n\}$, $f^*\subseteq \{0,1, \ldots, n\}$ so that $f\cup f^* = \{0, 1, \ldots, n\}$, i.e., $f^*$ is the complement of set $f$. Geometrically,
$$
f^* ={\rm Convex}(\texttt{v}_{f^*(1)}, \ldots, \texttt{v}_{f^*(n-\ell)}) \in \Delta_{n- \ell-1}(T)
$$
is the $(n- \ell-1)$-dimensional simplex spanned by vertices not contained in $f$. 

%
%

Given a sub-simplex $f\in \Delta_{\ell}(T)$, through the geometric embedding $f \hookrightarrow T$, we define the prolongation/extension operator $E: \mathbb T^{\ell}_k \to \mathbb T^{n}_k$ as follows:
$$
E(\alpha)_{f(i)} = \alpha_{i}, i=0,\ldots, \ell, \quad \text{ and } E(\alpha)_j = 0, j\not\in f.
$$
For example, assume $f = \{ 1, 3, 4\}$, then for 
$\alpha = ( \alpha_{0},\alpha_{1},\alpha_{2})\in \mathbb T^{\ell}_k$, the extension $ E(\alpha)= (0, \alpha_{0}, 0, \alpha_{1}, \alpha_{2}, \ldots, 0).$ The geometric embedding $x(E(\alpha))\in f$ which justifies the notation $\mathbb T^{\ell}_k(f)$.
With a slight abuse of notation, for a node $\alpha_f\in \mathbb T^{\ell}_k(f)$, we still use the same notation $\alpha_f\in \mathbb T^{n}_k(T)$ to denote such extension. Then we have the following direct decomposition
\begin{equation}\label{eq:decalpha}
\alpha = E(\alpha_f) + E(\alpha_{f^*}) = \alpha_f + \alpha_{f^*}, \text{ and } |\alpha | = |\alpha_f | + | \alpha_{f^*}|.
\end{equation}
Based on \eqref{eq:decalpha}, we can write a Bernstein polynomial as
\begin{equation*}
\lambda^{\alpha} = \lambda_{f}^{\alpha_f}\lambda_{f^*}^{\alpha_{f^*}},
\end{equation*}
where $\lambda_{f}=\lambda_{f(0)} \ldots \lambda_{f(\ell)}\in \mathbb P_{\ell +1}(f)$ is the bubble function on $f$.

In summary, by treating $f$ as a set of indices, we can apply the operators $\cup, \cap, {}^*, \backslash$ on sets. While treating $f$ as a geometric simplex, $\partial f, \stackrel{\circ}{f}$ etc can be applied.
%
%
%


\subsection{Distance to a sub-simplex}
Given $f\in \Delta_{\ell}(T)$, we define the distance of a node $\alpha\in \mathbb T_k^n$ to $f$ as
\begin{equation*}
\dist(\alpha, f) :=| \alpha_{f^*} | = \sum_{i\in f^*} \alpha_i.
\end{equation*}
Next we present a geometric interpretation of $\dist(\alpha, f)$.
We set the vertex $\texttt{v}_{f(0)}$ as the origin and embed the lattice to the scaled reference simplex $k\hat T = {\rm span} \{\bs 0, k\bs e_1, \ldots, k\bs e_n\}$. Then $|\alpha_{f^*}| = s$ corresponds to the linear equation
$$
x_{f^*(1)} + x_{f^*(2)} + \ldots + x_{f^*(n - \ell)} = s,
$$
which defines a hyper-plane in $\mathbb R^n$, denoted by $L(f,s)$, with a normal vector $\bs 1_{f^*}$.
The simplex $f$ can be thought of as convex combination of vectors $\{\bs e_{f(0)f(i)}\}_{i=1}^{\ell}$. Obviously $\bs 1_{f^*}\cdot \bs e_{f(0)f(i)} = 0$ as the zero pattern is complementary to each other. So $f$ is parallel to the hyper-plane $L(f,s)$. The distance $\dist(\alpha, f)$ for $\alpha\in L(f,s)$ is the intercept of the hyper-plane $L(f,s)$; see Fig. \ref{fig:dist} for an illustration.
%
 In particular $f\in L(f,0)$ and $\lambda_{i}|_f = 0$ for $i\in f^*$. Indeed $f = \{x\in T\mid \lambda_i(x) = 0, i\in f^*\}$.
 

We can extend the definition to the distance between sub-simplexes. For $e\in \Delta_{\ell}(T),f \in \Delta(T)$, define $$\dist(e,f) = \min_{\alpha \in \mathbb T_k^{\ell}(e)} \dist (\alpha, f). $$
Then it is easy to verify that: for $e\in \Delta(f^*)$, $\dist (e,f) = k$ and for $e\in \Delta(f)$, i.e., $e\cap f \neq \varnothing,$ then $\dist (e,f) = 0$. 
%
%
%

\begin{figure}[htbp]
\subfigure[Distance to an edge.]{
\begin{minipage}[t]{0.5\linewidth}
\centering
\includegraphics*[height=4.45cm]{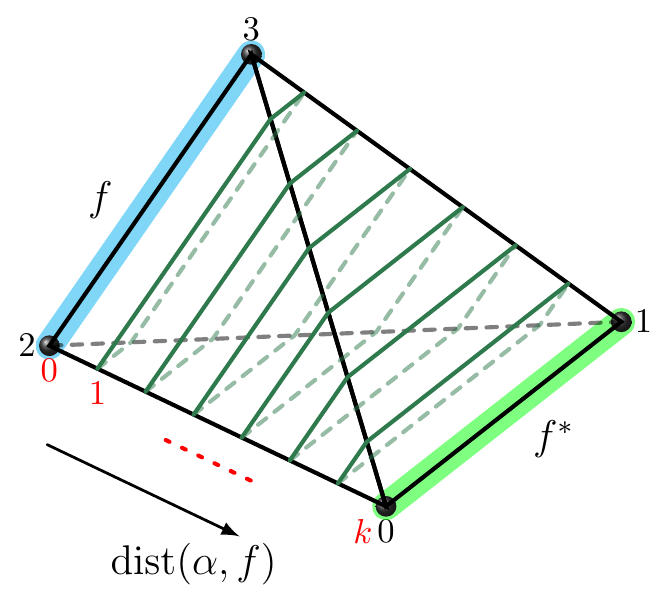}
\end{minipage}}
\subfigure[Distance to a face.]
{\begin{minipage}[t]{0.5\linewidth}
\centering
\includegraphics*[height=3.8cm]{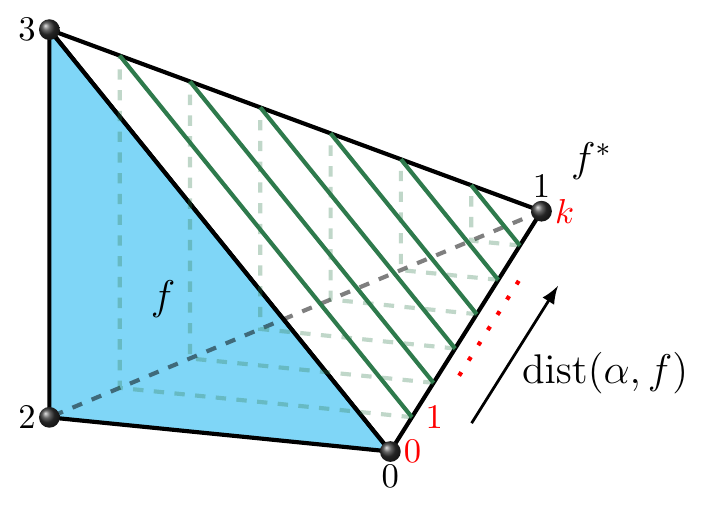}
\end{minipage}}
\caption{Distance to a sub-simplex.}
\label{fig:dist}
\end{figure}


We define the lattice tube of $f$ with distance $r$ as
$$
D(f, r) = \{ \alpha \in \mathbb T^{n}_k, \dist(\alpha,f) \leq r\},
$$
which contains lattice nodes at most $r$ distance away from $f$. We overload the notation
$$
L(f,s) = \{ \alpha \in \mathbb T^{n}_k, \dist(\alpha,f) = s\},
$$
which is defined as a plane early but here is a subset of lattice nodes on this plane. Then by definition, 
$$
D(f, r) = \cup_{s=0}^r L(f,s), \quad L(f,s) = L(f^*, k - s). 
$$
By definition $D(f, -1) = \varnothing$, $D(f,0) = L(f,0) = f$, and $L(f,k) = f^*$.
We have the following characterization of lattice nodes in $D(f, r)$.
\begin{lemma} \label{lm:dist}
For lattice node $\alpha \in \mathbb T^{n}_k$,
\begin{align*}
\alpha \in D(f, r)  &\iff  |\alpha_{f^*}| \leq r \iff | \alpha_{f} | \geq k - r,\\
\alpha \notin D(f, r) & \iff  |\alpha_{f^*}| > r \iff | \alpha_{f} | \leq k - r - 1.
\end{align*}
\end{lemma}
\begin{proof}
By definition of $\dist(\alpha, f)$ and the fact $| \alpha_{f} |  + | \alpha_{f^*} | = k$.
\end{proof}

For each vertex $\texttt{v}_i\in \Delta_0(T)$,
$$
D(\texttt{v}_i, r) =  \{ \alpha \in \mathbb T^{n}_k, |\alpha_{i^*}| \leq r \},
$$
which is isomorphic to a simplicial lattice $\mathbb T_{r}^n$ of degree $r$. For a face $f\in \Delta_{n-1}(T)$, $D(f,r)$ is a trapezoid of height $r$ with base $f$. In general for $f\in \Delta_{\ell}(T)$, the hyper plane  $L(f, r)$ will cut the simplex $T$ into two parts, and $D(f,r)$ is the part containing $f$.


\subsection{Derivative and distance}
Recall that in~\cite{ArnoldFalkWinther2009} a smooth function $u$ is said to vanish to order $r$ on $f$ if $D^{\beta} u|_f = 0$ for all $\beta \in \mathbb N^{1:n}, |\beta | < r$. The following result shows that the vanishing order $r$ of a Bernstein polynomial $\lambda^{\alpha}$ on $f$ is the distance $\dist(\alpha, f)$.
\begin{lemma}\label{lm:derivative}
Let $f\in \Delta_{\ell}(T)$ be a sub-simplex of $T$. For $\alpha\in \mathbb T^{n}_k, \beta \in \mathbb N^{1:n}$, then 
$$
D^{\beta} \lambda^{\alpha}|_{f} = 0, \quad  \text{ if }\dist(\alpha, f) > |\beta|, 
$$
\end{lemma}
\begin{proof}
For $\alpha\in \mathbb T^{n}_k$, we write $\lambda^{\alpha} = \lambda_{f}^{\alpha_f}\lambda_{f^*}^{\alpha_{f^*}}$. When $|\alpha_{f^*}| > |\beta|$, the derivative $D^{\beta} \lambda^{\alpha}$ will contain a factor $\lambda_{f^*}^{\gamma}$ with $\gamma\in \mathbb N^{1:n-\ell},$ and $|\gamma| = |\alpha_{f^*}| - |\beta| > 0$. Therefore $D^{\beta} \lambda^{\alpha}|_{f} = 0$ as $\lambda_{i}|_f = 0$ for $i\in f^*$.
\end{proof}

\subsection{Barycentric calculus}

The normalized basis $\lambda^{\alpha}/\alpha!$ has the constant integral
\begin{equation}\label{eq:normalizedlambda}
\int_T \frac{1}{\alpha!}\lambda^{\alpha} \dx = \frac{n!}{(k+n)!}{|T|}, \quad \forall~\alpha\in \mathbb T_k^n.
\end{equation}

Denote by $\bs t_{i,j}$ the edge vector from $\texttt{v}_i$ to $\texttt{v}_j$. By computing the constant directional derivative $\boldsymbol t_{i,j}\cdot\nabla \lambda_{\ell}$ by values on the two vertices, we have
\begin{equation}\label{eq:tijlambdal}
\boldsymbol t_{i,j}\cdot\nabla \lambda_{\ell}=\delta_{j\ell}-\delta_{i\ell}=\begin{cases}
1, & \textrm{ if } \ell=j,\\
-1, & \textrm{ if } \ell=i,\\
0, & \textrm{ if } \ell\neq i,j.
\end{cases}  
\end{equation}
Define $\epsilon_i\in \mathbb N^{0:n}$ as $\epsilon_i = (0,\ldots, 1, \ldots, 0)$ and $\epsilon_{ij} = \epsilon_i - \epsilon_j = (0,\ldots, 1, \ldots, -1,\ldots, 0)$.

\begin{lemma}
For $\alpha \in \mathbb T^n_k$ and $0\leq i\neq j\leq n$, 
 \begin{align}
\label{eq:dtji} \nabla (\lambda^{\alpha + \epsilon_i})\cdot \bs t_{j,i}& = (\alpha_i + 1) \lambda^{\alpha} - \alpha_j \lambda^{\alpha + \epsilon_{ij}}.
\end{align}
\end{lemma}
\begin{proof}
 By direct calculation and formula \eqref{eq:tijlambdal}.
\end{proof}

For two nodes $\alpha,\beta \in \mathbb T^n_k$, define the distance
\begin{equation*}
\dist(\alpha,\beta)= \frac{1}{2} \|\alpha - \beta\|_{\ell_1}.
\end{equation*}
Two nodes $\alpha, \beta\in \mathbb T^{n}_k$ are adjacent if $\dist(\alpha,\beta)= 1$. By assigning edges to all adjacent nodes, the simplicial lattice $ \mathbb T^n_k$ becomes an undirected graph and denoted by $\mathcal G( \mathbb T^n_k)$. The distance of two nodes in the graph is the length of a minimal path connecting them, where the length of a path is defined as the number of edges in the path. Obviously the graph $\mathcal G( \mathbb T^n_k)$ is connected.

\begin{lemma}
For $\alpha,\beta \in \mathbb T^n_k$, it holds
\begin{equation*}
\dist(\alpha,\beta)= 1 \iff \beta = \alpha + \epsilon_{ij}, \quad \text{for some } i,j \in [0, n], i\neq j. 
\end{equation*}
\end{lemma}
\begin{proof}
Notice for two non-negative integer, if $\alpha_i \neq \beta_i$, then $|\alpha_i - \beta_i |\geq 1$. As $\alpha,\beta \in \mathbb T^n_k$, we have $\sum_{i=0}^n (\alpha_i - \beta_i) = 0$. The condition $\dist(\alpha,\beta)= 1$ means $\sum_{i=0}^n |\alpha_i - \beta_i| = 2$. So the only possibility is: $\alpha_i-\beta_i = -1$ and $\alpha_j-\beta_j = 1$ for some $0\leq i,j\leq n$ and $i\neq j$. 
\end{proof}

We give an explicit form for functions in $\mathbb P_{k}(T)\cap L_0^2(T) = \{ p\in \mathbb P_{k}(T): \int_T p \dx = 0 \}$.
\begin{lemma}\label{lm:L20}
For $k\geq 0$, it holds
$$
\mathbb P_{k}(T)\cap L_0^2(T)=\mathrm{span}\{\lambda^{\alpha}/\alpha!-\lambda^{\beta}/\beta!: \alpha, \beta\in \mathbb T^n_k \textrm{ and } \dist(\alpha,\beta)=1 \}.
$$
\end{lemma}
\begin{proof}
By the integral formula \eqref{eq:normalizedlambda}, $\lambda^{\alpha}/\alpha!-\lambda^{\beta}/\beta!\in \mathbb P_{k}(T)\cap L_0^2(T)$ for $\alpha, \beta\in \mathbb T^n_k$. 

As the graph $\mathcal G( \mathbb T^n_k)$ is connected, we can find a spanning tree $\mathcal T$ with number of edge is equal to $ | \mathbb T^n_k | - 1 = \dim (\mathbb P_{k}(T)\cap L_0^2(T))$. 
So 
$
\{\lambda^{\alpha}/\alpha!-\lambda^{\beta}/\beta!: [\alpha, \beta] \text{ is an edge in } \mathcal T\}
$
is a basis of $\mathbb P_{k}(T)\cap L_0^2(T)$. Then the result follows as the edge of $\mathcal T$ is a subset of $\mathcal G( \mathbb T^n_k)$ and $ [\alpha, \beta]$ is an edge iff $\dist(\alpha,\beta)= 1$. 
\end{proof}

\begin{figure}[htbp]
\begin{center}
\includegraphics[width=6cm]{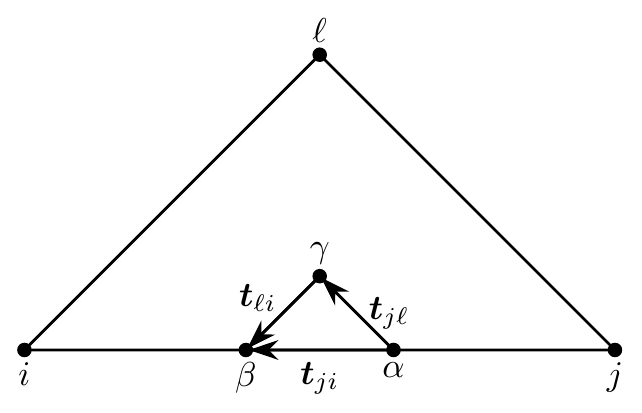}
\caption{Velocity fields satisfying $\div \bs u =  \frac{1}{\alpha!}\lambda^{\alpha} - \frac{1}{\beta!} \lambda^{\beta}$ with $\beta = \alpha + \epsilon_{ij}$. One direction is $\bs t_{ji}$ and another is a detour through $\gamma$.}
\label{default}
\end{center}
\end{figure}

Given a basis function $p = \frac{1}{\alpha!}\lambda^{\alpha} - \frac{1}{\beta!} \lambda^{\beta}$, we can find two $\bs u$ satisfying $\div \bs u = p$. 
\begin{lemma}\label{lm:alphabeta}
 Let $\alpha, \beta\in \mathbb T^{n}_k$ and $\beta = \alpha + \epsilon_{ij}$. Then
\begin{equation}\label{eq:divalphabeta1}
\frac{1}{\beta!\alpha_j}\div (\lambda^{\alpha+\epsilon_i}\bs t_{j,i}) = \frac{1}{\alpha!}\lambda^{\alpha} - \frac{1}{\beta!} \lambda^{\beta}.
\end{equation}
\end{lemma}
\begin{proof}
Direct calculation using formula \eqref{eq:dtji}.  
\end{proof}

\begin{lemma}\label{lm:alphabetagamma}
Let $\alpha, \beta\in \mathbb T^{n}_k$ and $\beta = \alpha + \epsilon_{ij}$. Choose an $\ell\in [0:n], \ell \neq i,j$, s.t. $\gamma = \alpha + \epsilon_{\ell j} = \beta + \epsilon_{\ell i}\in \mathbb T^{n}_k$. Then
\begin{equation}\label{eq:divalphabeta}
\frac{1}{\gamma_!\alpha_j}\div(\lambda^{\alpha + \epsilon_{\ell}}\boldsymbol{t}_{j,\ell})+\frac{1}{\gamma_!\beta_i}\div(\lambda^{\beta+ \epsilon_{\ell}}\boldsymbol{t}_{\ell,i}) = \frac{1}{\alpha!}\lambda^{\alpha} - \frac{1}{\beta!} \lambda^{\beta}.
\end{equation}
\end{lemma}
\begin{proof}
By \eqref{eq:divalphabeta1}, we have
$$
\frac{1}{\gamma_!\alpha_j}\div(\lambda^{\alpha + \epsilon_{\ell}}\boldsymbol{t}_{j,\ell})=\frac{1}{\alpha!}\lambda^{\alpha} - \frac{1}{\gamma!} \lambda^{\gamma},
\quad
\frac{1}{\gamma_!\beta_i}\div(\lambda^{\beta+ \epsilon_{\ell}}\boldsymbol{t}_{\ell,i})=\frac{1}{\gamma!}\lambda^{\gamma} - \frac{1}{\beta!} \lambda^{\beta}.
$$
Then \eqref{eq:divalphabeta} follows.
\end{proof}

\section{Smooth Finite Elements in Three Dimensions}\label{sec:geodecomp3d}
In this section, an integer vector $\boldsymbol r= (r^{\texttt{v}}, r^e, r^f)^{\intercal}$ is called a valid smoothness vector if $r^f\geq-1$, $r^e\geq \max\{2r^f,-1\}$ and $r^{\texttt{v}}\geq \max\{2r^e,-1\}$. It is also denoted as $\bs r = (r^{0}, r^1, r^2)^{\intercal}$, where the super-script $\ell = 0,1,2$ represents the dimension of the sub-simplex.

\subsection{A decomposition of the simplicial lattice}

\begin{lemma}\label{lm:disjoint}
For $\ell = 1,2$, if $r^{\ell-1}\geq 2r^{\ell}\geq0$, the sub-sets 
$$\{ D(f, r^{\ell}) \backslash \left [ \cup_{e\in \Delta_{\ell - 1}(f)} D(e, r^{\ell - 1})\right ], f\in \Delta_{\ell} (T)\}$$ are disjoint.
\end{lemma}
\begin{proof}
Consider two different sub-simplices $ f, \tilde f \in \Delta_{\ell} (T)$. The dimension of their intersection is at most $\ell - 1$. Therefore $f\cap \tilde f\subseteq e$ for some $e\in \Delta_{\ell -1}(f)$. Then $e^*\subseteq (f\cap \tilde f )^* = f^*\cup \tilde f^*$. For $\alpha \in D(f, r^{\ell})\cap D(\tilde f, r^{\ell})$, we have $|\alpha_{e^*}| \leq |\alpha_{f^*}| + |\alpha_{\tilde f^*}|\leq 2r^{\ell}\leq r^{\ell - 1}$. Therefore we have shown the intersection region $D(f, r^{\ell})\cap D(\tilde f, r^{\ell})\subseteq \cup_{e\in \Delta_{\ell - 1}(f)} D(e,r^{\ell-1})$ and the result follows.
%
\end{proof}

Next we remove $D(e, r^{i})$ from $D(f, r^{\ell})$ for all $e\in \Delta_{i}(T)$ and $i=0,1,\ldots, \ell-1$. 
\begin{lemma} \label{lm:Deltaf=DeltaT}
Given integer $m\geq 0$, let non-negative integer array $\bs r=(r^{0},r^1, r^2)$ satisfy
$$
r^{2}=m,\;\; r^{\ell}\geq 2r^{\ell+1} \; \textrm{ for } \ell=0,1.
$$
Let $k\geq 2r^{0}+1 \geq 8m + 1$. For $\ell = 1,2,$
\begin{equation}\label{eq:Deltaf=DeltaT}
D(f, r^{\ell}) \backslash \left [ \bigcup_{i=0}^{\ell-1}\bigcup_{e\in \Delta_{i}(f)}D(e, r^{i}) \right ] = D(f, r^{\ell}) \backslash \left [ \bigcup_{i=0}^{\ell-1}\bigcup_{e\in \Delta_{i}(T)}D(e, r^{i}) \right ] .
\end{equation}
\end{lemma}
\begin{proof}
In \eqref{eq:Deltaf=DeltaT}, the relation $\supseteq$ is obvious as $\Delta_i(f)\subseteq \Delta_i(T)$. To prove $\subseteq$, it suffices to show for $\alpha \in D(f, r^{\ell}) \backslash \left [ \bigcup_{i=0}^{\ell-1}\bigcup_{e\in \Delta_{i}(f)}D(e, r^{i}) \right ]$, it is also not in $D(e, r^i)$ for $e\in \Delta_i(T)$ and $e\not\in\Delta_i(f)$. 

By definition, 
$$
|\alpha_{f^*}|\leq r^{\ell},\; |\alpha_{e}|\leq k - r_{i}-1 \; \textrm{ for all }e\in\Delta_i(f), i=0,\ldots,\ell-1.
$$
For each $e\in\Delta_i(T)$ but $e\not\in\Delta_i(f)$, the dimension of the intersection $e \cap f$ is at most $i-1$. It follows from $r^{j}\geq 2r^{j+1}$ and $k\geq 2r^{0}+1$ that: when $i>0$,
$$
|\alpha_{e}|=|\alpha_{e\cap f}|+|\alpha_{e\cap f^*}|\leq k - r^{i-1}-1+r^{\ell}\leq k - r^{i}-1,
$$
and when $i=0$,
$$
|\alpha_{e}|=|\alpha_{e\cap f^*}|\leq r^{\ell}\leq k - r^{i}-1.
$$
So $|\alpha_{e^*}| > r^i$. We conclude that $\alpha \not\in D(e, r^i)$ for all $e\in\Delta_i(T)$ and \eqref{eq:Deltaf=DeltaT} follows. 
\end{proof}

We are in the position to present an important partition of the simplicial lattice.
\begin{theorem}\label{th:decT}
Given integer $m\geq 0$, let non-negative integer array $\bs r=(r^{0},r^1, r^2)^{\intercal}$ satisfy
$$
r^{2}=m,\;\; r^{\ell}\geq 2r^{\ell+1} \; \textrm{ for } \ell=0,1.
$$
Let $k\geq 2r^{0}+1 \geq 8m + 1$. Then we have the following direct decomposition of the simplicial lattice  on a tetrahedron $T$:
\begin{align}\label{eq:smoothdecnd}
  \mathbb T^{3}_k(T) = \Oplus_{\ell = 0}^{3}\Oplus_{f\in \Delta_{\ell}(T)} S_{\ell}(f, \bs r),
\end{align}
where
\begin{align*}
S_0(\texttt{v}, \bs r) &=  D(\texttt{v}, r^0), \\
S_{\ell}(f, \bs r) &= D(f, r^{\ell}) \backslash \left [ \bigcup_{i=0}^{\ell-1}\bigcup_{e\in \Delta_{i}(f)}D(e, r^{i}) \right ], \; \ell = 1,2, \\
S_3(T, \bs r) & = \mathbb T^{3}_k(T) \backslash  \left [  \bigcup_{i=0}^{2}\bigcup_{f\in \Delta_{i}(T)}D(f, r^{i}) \right ].
\end{align*}
\end{theorem}
\begin{proof}
First we show that the sets  $\{S_{\ell}(f, \bs r), f\in\Delta_{\ell}(T), \ell=0,\ldots,3\}$ are disjoint.
Take two vertices $\texttt{v}_1, \texttt{v}_2\in \Delta_0(T)$. For $\alpha\in D(\texttt{v}_1, r^{0})$, we have $\alpha_{\texttt{v}_1} \geq k - r^0$. As $\texttt{v}_1\in \texttt{v}_2^*$ and $k\geq 2r^0+1$, $|\alpha_{\texttt{v}_2^*}|\geq \alpha_{\texttt{v}_1}\geq k- r^{0}\geq r^{0}+1$, i.e., $\alpha \notin D(\texttt{v}_2, r^{0})$. Hence $\{S_{0}(\texttt{v}), \texttt{v}\in\Delta_{0}(T)\}$ are disjoint and $\Oplus_{\texttt{v}\in\Delta_{0}(T)} S_{0}(\texttt{v})$ is a disjoint union.
By Lemma \ref{lm:disjoint} and \eqref{eq:Deltaf=DeltaT}, we know $\{S_{\ell}(f, \bs r), f\in\Delta_{\ell}(T), \ell=0,\ldots,3\}$ are disjoint.

Next we inductively prove 
\begin{equation}\label{eq:USequalUD}
\Oplus_{i = 0}^{\ell}\Oplus_{f\in \Delta_{i}(T)} S_{i}(f, \bs r)=\bigcup_{i=0}^{\ell}\bigcup_{f\in \Delta_{i}(T)}D(f, r^{i}) \quad \textrm{ for }\; \ell=0,1, 2.
\end{equation}
Obviously \eqref{eq:USequalUD} holds for $\ell=0$. Assume \eqref{eq:USequalUD} holds for $\ell<j$. Then
\begin{align*}
&\quad\Oplus_{i = 0}^{j}\Oplus_{f\in \Delta_{i}(T)} S_{i}(f, \bs r)= \Oplus_{f\in \Delta_{j}(T)} S_{j}(f, \bs r)\;\oplus \;\bigcup_{i=0}^{j-1}\bigcup_{e\in \Delta_{i}(T)}D(e, r^{i}) \\
&= \Oplus_{f\in \Delta_{j}(T)}\left(D(f, r^{j}) \backslash \left [ \bigcup_{i=0}^{j-1}\bigcup_{e\in \Delta_{i}(T)}D(e, r^{i}) \right ]\right)\;\oplus \;\bigcup_{i=0}^{j-1}\bigcup_{e\in \Delta_{i}(T)}D(e, r^{i}) \\
&=\bigcup_{i=0}^{j}\bigcup_{f\in \Delta_{i}(T)}D(f, r^{i}).
\end{align*}
By induction, \eqref{eq:USequalUD} holds for $\ell=0,1,2$. Then \eqref{eq:smoothdecnd} is true from the definition of $S_3(T, \bs r)$ and~\eqref{eq:USequalUD}. 
\end{proof}

\begin{remark}\rm
When integer vector $\bs r=(r^{0},r^1, r^2)^{\intercal}$  satisfies
$$r^2=-1,\quad  r^1\geq -1, \quad r^0 \geq \max\{2r^1,-1\},$$ and nonnegative integer $k\geq 2r^0+1$, the direct decomposition \eqref{eq:smoothdecnd} still holds with 
$$
S_0(\texttt{v}, \bs r)=  D(\texttt{v}, r^0), \quad S_{1}(e, \bs r) = D(e, r^{1}) \backslash \bigcup_{\texttt{v}\in \Delta_{0}(e)}D(\texttt{v}, r^{0}),
$$
$$
S_{2}(f, \bs r)=\varnothing, \quad S_3(T, \bs r) = \mathbb T^{3}_k(T) \backslash  \left [  \bigcup_{i=0}^{1}\bigcup_{f\in \Delta_{i}(T)}D(f, r^{i}) \right ].
$$
\end{remark}

From the implementation point of view, the index set $S_{\ell}(f, \bs r)$ can be found by a logic array and set the entries as true when the distance constraint holds. 

\begin{figure}[htbp]
\label{fig:facebubble}
\subfigure[A decomposition of the simplical lattice.]{
\begin{minipage}[t]{0.5\linewidth}
\centering
\includegraphics*[width=4.2cm]{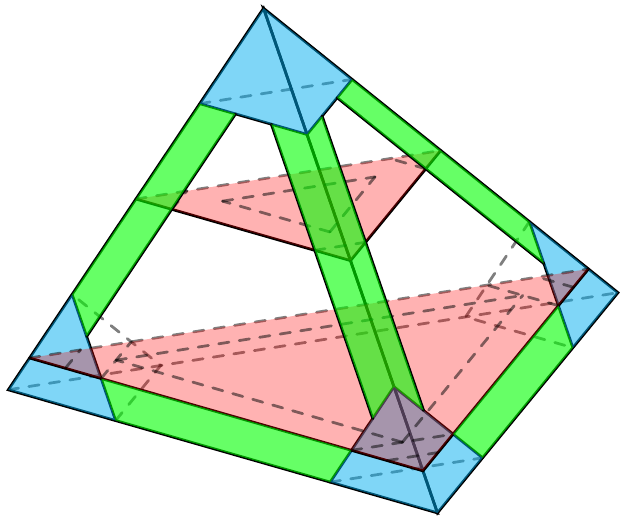}
\end{minipage}}
\subfigure[Plane $L(f,r^f)$]
{\begin{minipage}[t]{0.5\linewidth}
\centering
\includegraphics*[width=3.75cm]{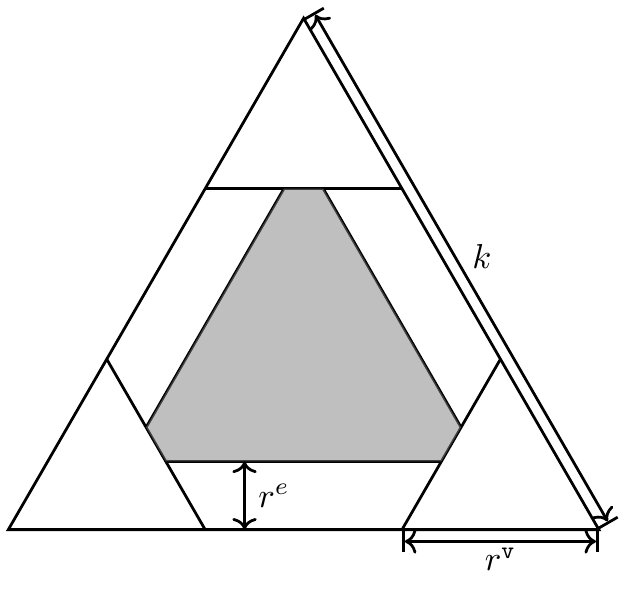}
\end{minipage}}
\subfigure[Plane $L(f,j)$ for $r^e - r^f\leq j\leq r^{\texttt{v}} - r^e$.]{
\begin{minipage}[t]{0.5\linewidth}
\centering
\includegraphics*[width=3.5cm]{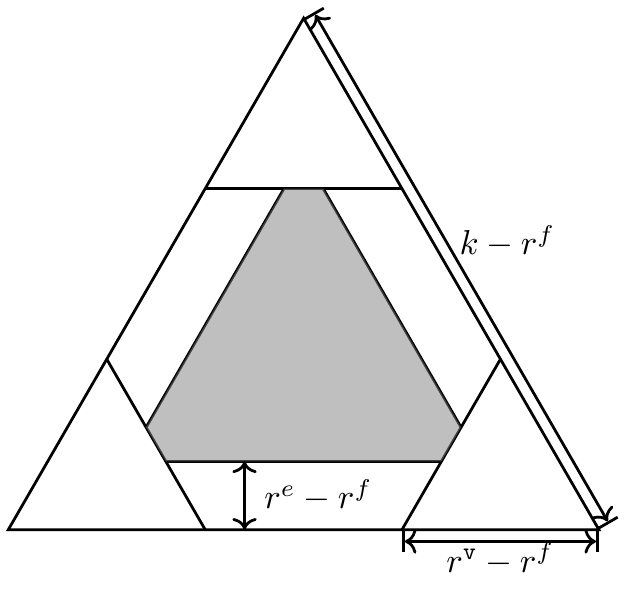}
\end{minipage}}
\subfigure[Plane $L(f,j)$ for $r^{\texttt{v}}  - r^e\leq j \leq k-r^{\texttt{v}} - 1$.]
{\begin{minipage}[t]{0.5\linewidth}
\centering
\includegraphics*[width=3cm]{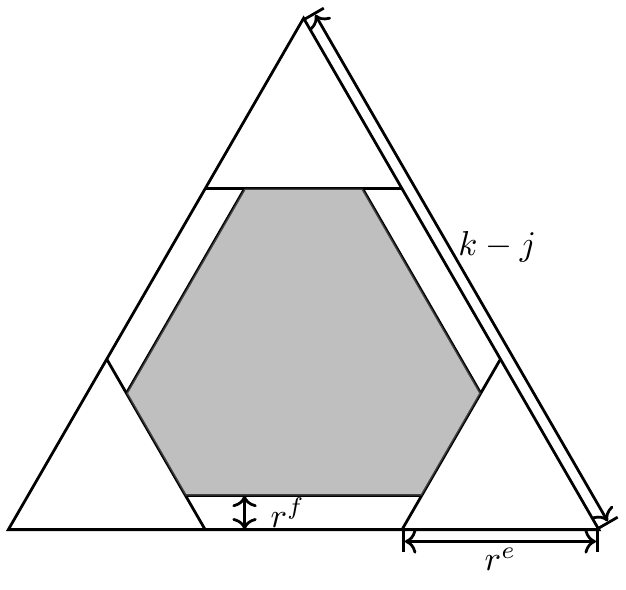}
\end{minipage}}
\caption{Different planes $L(f,j)$ and $S_2(f,  \begin{pmatrix}
r^{\texttt{v}} \\
r^e
\end{pmatrix}-j)$ (in gray). The distance to a vertex is decreasing from $r^{\texttt{v}}$ to $r^{e}$ and the distance to an edge is from $r^e$ to $r^f$.}
\end{figure}

Introduce the polynomial bubble space on $T$
\begin{equation}\label{eq:bubbleT}
\mathbb B_k(T; \bs r) = \mathbb P_k(S_3(T, \bs r)) = \spa \{ \lambda^{\alpha},\; \alpha \in S_3(T, \bs r) \}, 
\end{equation}
the face bubble space on $f\in\Delta_2(T)$
\begin{equation*}
\mathbb B_k(f;\begin{pmatrix}
r^{\texttt{v}} \\
r^e
\end{pmatrix}) = {\rm span} \left\{ \lambda_f^{\alpha}, \;\alpha \in \mathbb T_k^2(f)\backslash  \bigcup_{i=0}^{1}\bigcup_{e\in \Delta_{i}(f)}D(e, r^{i}) \right\},
\end{equation*}
and the edge bubble space
\begin{equation*}
\mathbb B_k(e; r^{\texttt{v}}) = {\rm span} \left\{ \lambda_e^{\alpha} = b_e^{r^\texttt{v}+1} \lambda_e^{\alpha - r^\texttt{v}-1}, \;\alpha \in \mathbb T_k^1(e)\backslash  \bigcup_{\texttt{v}\in \Delta_0(e)}D( \texttt{v}, r^{\texttt{v}}) \right\}.
\end{equation*}
In general, on the lattice plane $L(f, j)$, the face bubble space becomes $\mathbb B_{k-j}(f; \begin{pmatrix}
r^{\texttt{v}} \\
r^e
\end{pmatrix}-j)$ and
$$
S_3(T, \bs r) = \bigcup_{j=r^f+1}^{k - r^{\texttt{v}} -1} \left (S_2(f,  \begin{pmatrix}
r^{\texttt{v}} \\
r^e
\end{pmatrix}-j)\cap L(f, j) \right ). 
$$
We provide several slides of $L(f, j)$ and corresponding $S_{\ell}$ in Figures \ref{fig:facebubble} (b)(c)(d). 

\begin{corollary}
We have the following geometric decomposition of $\mathbb P_{k}(T)$
\begin{align}\label{eq:PrSdec3d}
\mathbb P_{k}(T) = &\Oplus_{\ell = 0}^{3} \Oplus_{f\in \Delta_{\ell}(T)} \mathbb P_k(S_{\ell}(f, \bs r)).
\end{align}

\end{corollary}

\subsection{Smooth finite elements in three dimensions}
In this subsection, we shall construct finite element spaces with smoothness parameter $\boldsymbol r= (r^{\texttt{v}}, r^e, r^f)^{\intercal}$.

For each edge $e$, we choose two normal vectors $n_1^e, n_2^e$ and abbreviate as $n_1,n_2$. For each face $f$, we choose a normal vector $n_f$ and abbreviate as $n$ when $f$ is clear in the context. When in a conforming mesh $\mathcal T_h$,  $n_1^e, n_2^e$ or $n_f$ depends on $e$ or $f$ not the element containing it.  

\begin{theorem}\label{thm:Cr3dfemunisolvence}
 Let $\boldsymbol r= (r^{\texttt{v}}, r^e, r^f)^{\intercal}$  with $r^{f}=m \geq -1$, $r^{e}\geq\max\{2r^{f},-1\}$, $r^{\texttt{v}}\geq\max\{2r^{e},-1\}$, and nonnegative integer $k\geq 2r^{\texttt{v}}+1$. The shape function space $\mathbb P_{k}(T)$ is determined by the DoFs
\begin{align}
\label{eq:C13d0}
\nabla^j u (\texttt{v}), & \quad \texttt{v}\in \Delta_0(T), j=0,1,\ldots,r^{\texttt{v}}, \\
\label{eq:C13d1}
\int_e \frac{\partial^{j} u}{\partial n_1^{i}\partial n_2^{j-i}} \, q \dd s, & \quad e\in \Delta_1(T), q \in \mathbb P_{k - 2(r^{\texttt{v}}+1) + j}(e), 0\leq i\leq j\leq r^{e}, \\
\label{eq:C13d2}
\int_f \frac{\partial^{j} u}{\partial n_f^{j}} \, q \dd S, & \quad  f\in \Delta_2(T), q \in \mathbb B_{k-j}(f;\begin{pmatrix}
r^{\texttt{v}} \\
r^e
\end{pmatrix}-j), 0\leq j\leq r^{f},\\
\label{eq:C13d3}
\int_T u \, q \dx, & \quad q \in \mathbb B_k(T;\bs r).
\end{align}
With mesh $\mathcal T_h$, define the global $C^m$-continuous finite element space
\begin{align*}
\mathbb V_k(\mathcal T_h; \boldsymbol{r}) &= \{u\in C^m(\Omega): u|_T\in\mathbb P_k(T)\textrm{ for all } T\in\mathcal T_h, \\
&\qquad\textrm{ and all the DoFs~\eqref{eq:C13d0}-\eqref{eq:C13d3} are single-valued}\}.
\end{align*}
\end{theorem}
\begin{proof}
Thanks to the geometric decomposition \eqref{eq:PrSdec3d}, the number of DoFs \eqref{eq:C13d0}-\eqref{eq:C13d3} is same as $\dim\mathbb P_{k}(T)$. Take $u\in\mathbb P_{k}(T)$ and assume all the DoFs \eqref{eq:C13d0}-\eqref{eq:C13d3} vanish. We will prove $u=0$.

The vanishing DoF \eqref{eq:C13d0} implies $(\nabla^ju)(\texttt{v})=0$ for $\texttt{v}\in\Delta_0(T)$ and $0\leq j\leq r^{\texttt{v}}$, which combined with the vanishing DoF \eqref{eq:C13d1} yields $(\nabla^ju)|_e=0$ for $e\in\Delta_1(T)$ and $0\leq j\leq r^{e}$. Then $\left.\frac{\partial^{j} u}{\partial n_f^{j}}\right|_f\in\mathbb B_{k-j}(f;\begin{pmatrix}
r^{\texttt{v}} \\
r^e
\end{pmatrix}-j)$ for $f\in \Delta_2(T)$ and $0\leq j\leq r^{f}$. Now the vanishing DoF \eqref{eq:C13d2} indicates $(\nabla^ju)|_{f}=0$ for $f\in\Delta_2(T)$ and $0\leq j\leq r^{f}$. As a result $u\in\mathbb B_k(T;\bs r)$. Therefore $u=0$ follows from the vanishing DoF \eqref{eq:C13d3}.

Finally $\mathbb V_k(\mathcal T_h; \boldsymbol{r}) \subset C^m(\Omega)$ since we derive $(\nabla^ju)|_{f}=0$ for $f\in\Delta_2(T)$ and $0\leq j\leq m$ by using only DoFs \eqref{eq:C13d0}-\eqref{eq:C13d2} on $f$.
\end{proof}

\begin{remark}\rm 
When $\dim\mathbb B_k(T;\bs r)\geq1$, DoF \eqref{eq:C13d3} can be changed to 
$$
\int_T u\,q\dx,\quad q \in \mathbb P_0(T)\oplus \mathbb B_k(T;\bs r)/\mathbb R,
$$ where $\mathbb B_k(T;\bs r)/\mathbb R:=\mathbb B_k(T;\bs r)\cap L_0^2(T)$.
Similarly, when $\dim\mathbb B_{k}(f;\begin{pmatrix}
r^{\texttt{v}} \\
r^e
\end{pmatrix})\geq1$, the face DoF $\int_f u\, q\dd S, q \in \mathbb B_{k}(f;\begin{pmatrix}
r^{\texttt{v}} \\
r^e
\end{pmatrix})$ can be changed to $\int_f u q\dd S, q \in \mathbb P_0(f)\oplus \mathbb B_{k}(f;\begin{pmatrix}
r^{\texttt{v}} \\
r^e
\end{pmatrix})/\mathbb R$. 
\end{remark}

Next we count the dimension of the finite element space. For integers $0\leq i\leq j\leq n$, recall the combinatorial formula
\begin{equation}\label{eq:combinaformula}
{j\choose i}+{j+1\choose i}+\cdots+{n\choose i}={n+1\choose i+1}-{j\choose i+1},
\end{equation}
which holds from ${n\choose i}+{n\choose i+1}={n+1\choose i+1}$. Here we understand ${i\choose i+1}$ as $0$.
For an $n$-dimensional tetrahedron, the number ${n+1\choose i+1}$ of $i$-dimensional faces equals to the sum of the number ${n\choose i}$ of $i$-dimensional faces including $\texttt{v}_0$ as a vertex and the number ${n\choose i+1}$ of $i$-dimensional faces excluding $\texttt{v}_0$.

\begin{proposition}\label{lem:Cr3dfemdimension}
The dimension of $\mathbb V_k(\mathcal T_h; \boldsymbol{r})$ is
\begin{align*}
\dim \mathbb V_k(\mathcal T_h; \boldsymbol{r}) &= \sum_{i=0}^3C_i(k,\bs r)|\Delta_i(\mathcal T_h)|.
\end{align*}
where
\begin{align*}
C_0(k,\bs r)&={r^{\texttt{v}}+3 \choose 3},\\
C_1(k,\bs r)&=(k+r^e-2r^{\texttt{v}}-1){r^e+2\choose 2}-{r^e+2\choose 3},\\
C_2(k,\bs r)&={k+3\choose3}-3{r^{\texttt{v}}+3\choose3}-3{k-2r^{\texttt{v}}-1 \choose 3}-{k+2-r^f\choose3} \\
&\quad +3{r^{\texttt{v}}+2-r^f\choose3}-3(r^f+1){k-2r^{\texttt{v}}+r^e\choose 2} + 3{k-2r^{\texttt{v}} +r^f\choose 3}, \\
C_3(k,\bs r) &={k+3\choose3} -4 C_0(k,\bs r)-6C_1(k,\bs r)-4C_2(k,\bs r).
\end{align*}
\end{proposition}
\begin{proof}
By \eqref{eq:combinaformula} with $i=2$, the number of DoFs on edge $e\in\Delta_1(\mathcal{T}_h)$ is
\begin{align*}
\sum_{j=0}^{r^e}(j+1)(k-2r^{\texttt{v}}-1+j)&=(k-2r^{\texttt{v}}-1){r^e+2\choose 2} + 2\sum_{j=0}^{r^e}{j+1\choose 2} \\
&=(k+r^e-2r^{\texttt{v}}-1){r^e+2\choose 2} - {r^e+2\choose 3}.
\end{align*}
Applying \eqref{eq:combinaformula} with $i=2$ again, the number of DoFs on face $f\in\Delta_2(\mathcal{T}_h)$ is
\begin{align*}
&\quad \sum_{j=0}^{r^f}\dim\mathbb B_{k-j}(f;\begin{pmatrix}
r^{\texttt{v}} \\
r^e
\end{pmatrix}-j)\\
&=\sum_{j=0}^{r^f}\left[{k+2-j\choose2}-3{r^{\texttt{v}}+2-j\choose2} - 3{k-2r^{\texttt{v}}+r^e\choose 2}+3{k-2r^{\texttt{v}}-1 +j\choose 2} \right] \\
&={k+3\choose3}-{k+2-r^f\choose3}-3{r^{\texttt{v}}+3\choose3}+3{r^{\texttt{v}}+2-r^f\choose3} \\
&\quad -3(r^f+1){k-2r^{\texttt{v}}+r^e\choose 2} + 3{k-2r^{\texttt{v}} +r^f\choose 3}-3{k-2r^{\texttt{v}}-1 \choose 3},
\end{align*}
which ends the proof.
\end{proof}

\subsection{$H(\div)$-conforming finite elements}
For a linear space $V$, denote by $V^3 := V\otimes \mathbb R^3$. Let $\mathbb V^{\div}_{k}(\mathcal T_h; \boldsymbol{r}):= \mathbb V_k^3(\mathcal T_h; \boldsymbol{r})\cap\boldsymbol{H}(\div,\Omega)$, where $\mathbb V_k(\mathcal T_h; \boldsymbol{r})$ is defined in Theorem~\ref{thm:Cr3dfemunisolvence}. Define the polynomial bubble space 
$$\mathbb B_k^{\div}(T;\bs r) := \ker({\rm tr}^{\div})\cap \mathbb B_k^3(T; \bs r),$$
where ${\rm tr}^{\div} \boldsymbol v = \boldsymbol n\cdot \boldsymbol v|_{\partial T}$. When $r^f\geq 0$, $\mathbb V^{\div}_{k}(\mathcal T_h; \boldsymbol{r}) =\mathbb V_k^3(\mathcal T_h; \boldsymbol{r})\subset \bs H^1(\Omega;\mathbb R^3)$ and $\mathbb B_k^{\div}(T;\bs r) = \mathbb B_k^3(T; \bs r)$. When $r^f = -1$, $\mathbb V_k^3(\mathcal T_h; \boldsymbol{r})$ is discontinuous and to be in $\boldsymbol{H}(\div,\Omega)$, the normal direction should be continuous. 

A precise characterization of $\mathbb B_k^{\div}(T;\bs r)$ with $r^f=-1$ is given below, where for an integer $m\geq -1$, $m_+ := \max \{m, 0\}$, and the Iverson bracket $[ m=-1]=\begin{cases}
1 & \textrm{ if } m =-1,\\
0 &  \textrm{ if } m \neq-1.
\end{cases}$
For each face $f$, choose two linearly independent tangent vectors $\{\bs t_f^1, \bs t_f^2\}$ and for each edge $e$, choose a tangent vector $\bs t_e$. Define
\begin{align*}
\mathbb B_{k}^{\div}(f; \bs r_+) := {}&
\mathbb B_{k}(f; 
\begin{pmatrix}
 r_+^{\texttt{v}}\\
 r_+^e
\end{pmatrix}
) \otimes  \spa \{\bs t_f^1, \bs t_f^2\},
\\
\mathbb B_{k}^{\div}(e; \bs r_+) := {}& \mathbb B_k(e;  r_+^{\texttt{v}})\otimes \spa \{\bs t_e\}.
\end{align*}

\begin{lemma}
Consider $\bs r = (r^{\texttt{v}},r^e,-1)^{\intercal}$ with $r^{\texttt{v}}\geq \max\{2 r^e,-1\}$ and $r^e\geq -1$. We have
\begin{align}\label{eq:bubbledecomprfn1}
\mathbb B_k^{\div}(T; 
\begin{pmatrix}
r^\texttt{v} \\
r^e\\
 -1
\end{pmatrix}
) = &{}\,
\mathbb B_k^3(T; \bs r_+) \Oplus_{f\in \Delta_{2}(T)} 
\mathbb B_{k}^{\div}(f; \bs r_+)\\
&\Oplus_{e\in \Delta_1(T)} [r^e=-1]
\mathbb B_{k}^{\div}(e; \bs r_+).
\notag
\end{align}
\end{lemma}
\begin{proof}
For $\bs r = (-1,-1,-1)^{\intercal}$, we have proved the desired decomposition in ~\cite{Chen;Huang:2021Geometric}
\begin{align*}
\mathbb B_k^{\div}(T; 
\bs -1)  = {}&\mathbb B_k^3(T; \bs 0)  
\Oplus_{f\in \Delta_{2}(T)} 
\mathbb B_{k}^{\div}(f; \bs 0)
\Oplus_{e\in \Delta_1(T)} \mathbb B_{k}^{\div}(e; 0).
\end{align*}

We then consider the case $r^{\texttt{v}}\geq 0$. By definition,
\begin{align*}
\mathbb B_k^{3}(T; 
\begin{pmatrix}
r^\texttt{v} \\
r^e\\
 -1
\end{pmatrix}
) = &
\mathbb B_k^3(T; 
\begin{pmatrix}
r^\texttt{v} \\
r_+^e\\
 0
\end{pmatrix}
) \Oplus_{f\in \Delta_{2}(T)} 
\mathbb B_{k}^3(f; 
\begin{pmatrix}
 r^{\texttt{v}}\\
 r^e_+
\end{pmatrix}
) \Oplus_{e\in \Delta_1(T)} [r^e=-1] \mathbb B_k^3(e;r^{\texttt{v}}).
\end{align*}
We write
$$
\mathbb B_{k}^3(f; 
\begin{pmatrix}
 r^{\texttt{v}}\\
 r_+^e
\end{pmatrix}
) = \mathbb B_{k}(f; 
\begin{pmatrix}
 r^{\texttt{v}}\\
 r_+^e
\end{pmatrix}
) \otimes \left (\spa \{\bs t_f^1, \bs t_f^2\} + \spa \{ \bs n_f\}\right ). 
$$
The intersection with $\ker(\div)$ will keep the tangential components only. Similarly only $\mathbb B_k(e;  r^{\texttt{v}})\otimes \spa \{\bs t_e\} $ is left in the $t-n$ decomposition of $\mathbb B_k^3(e;  r^{\texttt{v}})$. 
\end{proof}

Notice that we have the relation
$$
\mathbb B_k^3(T; 
\begin{pmatrix}
r_+^\texttt{v} \\
r_+^e\\
0
\end{pmatrix}
)
\subset
\mathbb B_k^{\div}(T; 
\begin{pmatrix}
r^\texttt{v} \\
r^e\\
 -1
\end{pmatrix}
)
\subset 
\mathbb B_k^3(T; 
\begin{pmatrix}
r^\texttt{v} \\
r^e\\
 -1
\end{pmatrix}
)
$$
and $$\mathbb B_k^{\div}(T; 
\begin{pmatrix}
r^\texttt{v} \\
r^e\\
 -1
\end{pmatrix}
)
\subseteq \mathbb B_k^{\div}(T; 
\begin{pmatrix}
-1 \\
-1 \\
 -1
\end{pmatrix}
) = (\ker(\tr^{\div})\cap \mathbb P_k(T;\mathbb R^3)).$$

\begin{theorem}\label{thm:Hdiv3dfemunisolvence}
 Let $\boldsymbol r= (r^{\texttt{v}}, r^e, r^f)^{\intercal}$  with $r^{f}= -1$, $r^{e}\geq -1$, $r^{\texttt{v}}\geq\max\{2r^{e},-1\}$, and nonnegative integer $k\geq 2r_+^{\texttt{v}}+1$. 
The shape function space $\mathbb P_k(T;\mathbb R^3)$ is determined by the DoFs
\begin{align}
\label{eq:divdof0}
\nabla^j \bs v (\texttt{v}), & \quad \texttt{v}\in \Delta_0(T), j=0,1,\ldots,r^{\texttt{v}}, \\
\label{eq:divdof1}
\int_e \frac{\partial^{j} \bs v}{\partial n_1^{i}\partial n_2^{j-i}} \cdot \bs q \dd s, & \quad e\in \Delta_1(T), \bs q \in \mathbb P^3_{k - 2(r^{\texttt{v}}+1) + j}(e), 0\leq i\leq j\leq r^{e}, \\
\label{eq:divdof2}
\int_f \boldsymbol  v\cdot\boldsymbol n \, q \dd S, & \quad q\in  \mathbb B_{k}(f; 
\begin{pmatrix}
 r^{\texttt{v}}\\
 r^e
\end{pmatrix}
), f\in \Delta_2(T), \\
\label{eq:divdof3}
\int_K \boldsymbol v \cdot \boldsymbol q \dx, &\quad \boldsymbol q\in \mathbb B_k^{\div}(T; 
\boldsymbol r
).
\end{align}
With mesh $\mathcal T_h$, define the global $H(\div)$-conforming finite element space
\begin{align*}
\mathbb V_k^{\div}(\mathcal T_h; \boldsymbol{r}) &= \{\boldsymbol{v}\in \boldsymbol{H}(\div,\Omega): \boldsymbol{v}|_T\in\mathbb P_k(T;\mathbb R^3)\textrm{ for all } T\in\mathcal T_h, \\
&\qquad\quad\;\;\;\textrm{ and all the DoFs~\eqref{eq:divdof0}-\eqref{eq:divdof3} are single-valued}\}.
\end{align*}
\end{theorem}
\begin{proof}
The unisolvence of DoFs~\eqref{eq:divdof0}-\eqref{eq:divdof3} for $\mathbb P_k(T;\mathbb R^3)$ follows from Theorem~\ref{thm:Cr3dfemunisolvence} and decomposition \eqref{eq:bubbledecomprfn1}. More precisely, let us first consider the case $r^{\texttt{v}}\geq 0, r^e\geq 0, r^f = -1$. Then using the DoFs for $\bs r_+ = (r^{\texttt{v}}, r^e, 0)^{\intercal}$, we know $\mathbb P_k(T;\mathbb R^3)$ is determined by \eqref{eq:divdof0}-\eqref{eq:divdof1} and
\begin{align}
\label{eq:divdofproof2}
\int_f \boldsymbol  v\cdot\boldsymbol q \dd S, & \quad q\in  \mathbb B_{k}^3(f; 
\begin{pmatrix}
 r^{\texttt{v}}\\
 r^e
\end{pmatrix}
), f\in \Delta_2(T), \\
\label{eq:divdofproof3}
\int_K \boldsymbol v \cdot \boldsymbol q \dx, &\quad \boldsymbol q\in \mathbb B_k^{3}(T; 
\boldsymbol r_+
).
\end{align}
On each face, we use the decomposition $\mathbb R^3 = \spa \{\bs t_f^1, \bs t_f^2\} \oplus \spa \{ \bs n_f\}$ and move the tangential components into the bubble space $\mathbb B_k^{\div}(T; \boldsymbol r )$. Therefore only the normal component \eqref{eq:divdof2} is left. 

When $r^{\texttt{v}}\geq 0, r^e= -1, r^f = -1$, we consider the DoFs for $\bs r_+ =  (r^{\texttt{v}}, 0, 0)^{\intercal}$. That is \eqref{eq:divdof0}, volume DoF \eqref{eq:divdofproof3}, and the edge and face DoFs
\begin{align}
\label{eq:divdofproof1}
\int_e \boldsymbol  v\cdot\boldsymbol q \dd s, & \quad \bs q\in  \mathbb B_{k}^3(e; 
r^{\texttt{v}}), e\in \Delta_1(T), \\
\label{eq:divdofproof22}
\int_f \boldsymbol  v\cdot\boldsymbol q \dd S, & \quad \bs q\in  \mathbb B_{k}^3(f; 
\begin{pmatrix}
 r^{\texttt{v}}\\
 r^e_+
\end{pmatrix}
), f\in \Delta_2(T).
\end{align}
As before on each face, we move the tangential components into the bubble space $\mathbb B_k^{\div}(T; \boldsymbol r)$ and keep only normal component with the test function $q\in  \mathbb B_{k}(f; 
\begin{pmatrix}
 r^{\texttt{v}}\\
 r^e_+
\end{pmatrix}
)$. On each edge $e$, we use the decomposition $\mathbb R^3 = \spa \{\bs n_{f_1}, \bs n_{f_2}\} \oplus \spa \{ \bs t_e\}$ where $f_1, f_2$ are two faces containing $e$. Again the tangential component $\mathbb B_{k}(e; 
r^{\texttt{v}}) \otimes \spa \{ \bs t_e\}$ is moved into the bubble space $\mathbb B_k^{\div}(T; \boldsymbol r )$. The normal components will be redistributed to the two faces containing $e$ so that $\mathbb B_{k}(f; 
\begin{pmatrix}
 r^{\texttt{v}}\\
r^e_+
\end{pmatrix}
) \Oplus_{e\in \Delta_1(f)} \mathbb B_{k}(e; 
r^{\texttt{v}}) = 
\mathbb B_{k}(f; 
\begin{pmatrix}
 r^{\texttt{v}}\\
r^e
\end{pmatrix}
)
$
for $r^e = -1$, which leads to \eqref{eq:divdof2}. When $ r^{\texttt{v}} = -1$, we can redistribute $3$ components of a vector into $3$ faces containing that vertex so that  \eqref{eq:divdof2} still holds. We refer to \cite[Fig. 3]{Chen;Huang:2021Geometric} for an illustration. 
\end{proof}

\begin{example}[$H(\div)$-conforming element]\rm
We recover the following known $H(\div)$-conforming finite elements:
\begin{enumerate}[(i)]
\item When $k\geq1$, $\boldsymbol{r}=-1$, it is Brezzi-Douglas-Marini (BDM) element~\cite{BrezziDouglasMarini1986,BrezziDouglasDuranFortin1987} and denoted by ${\rm BDM}_k$.
\item When $k\geq2$, $\boldsymbol{r}=(0,-1,-1)^{\intercal}$, it is Stenberg's element~~\cite{Stenberg2010}.
\end{enumerate}
\end{example}


\begin{corollary}\label{cor:spacedec}
 Let $\boldsymbol r= (r^{\texttt{v}}, r^e, r^f)^{\intercal}$  with $r^{f}= -1$, $r^{e}\geq -1$, $r^{\texttt{v}}\geq\max\{2r^{e},-1\}$, and nonnegative integer $k\geq 2r_+^{\texttt{v}}+1$. We have 
the dimension formula
\begin{align*}
\dim\mathbb V^{\div}_{k}(\mathcal T_h; \bs r)=\dim\mathbb V^3_{k}(\mathcal T_h; \bs r_+) - 3[ r^{\texttt{v}}=-1 ]|\Delta_0(\mathcal T_h)| &\\
\qquad\quad  
-3[r^e=-1](k-2r_+^{\texttt{v}}-1)|\Delta_1(\mathcal T_h)| &\\
\qquad\quad +\left (-2C_2(k,\bs r_+ )+3[r^e=-1](k-2r_+^{\texttt{v}}-1)+3[ r^{\texttt{v}}=-1 ]\right )|\Delta_2(\mathcal T_h)| & \\
\qquad\quad +\left (8C_2(k,\bs r_+)+6[r^e=-1](k-2r_+^{\texttt{v}}-1)\right )|\Delta_3(\mathcal T_h)| &.
\end{align*}
\end{corollary}
\begin{proof}

When counting the dimension, we compare $\mathbb V^{\div}_{k}(\mathcal T_h; \bs r) $ with the continuous element $\mathbb V^3_{k}(\mathcal T_h; \bs r_+)$.
As $r^f = -1$, the two tangential components of the face DoFs are considered as interior DoFs and thus subtracted from coefficients of $\Delta_2(\mathcal T_h)$. The cumulation of $4$ faces tangential bubbles contributes to the increase $8C_2(k,\bs r_+)$ in the coefficient of $\Delta_3(\mathcal T_h)$. Similarly when $r^e=-1$, we add total $6$ tangential edge  bubbles to the interior and redistribute the two normal components of edge bubbles to each face. When $r^{\texttt{v}} = -1$, the three components of the vector function at vertices are redistributed to three faces containing that vertex. Therefore facewisely we add $3(k-2r_+^{\texttt{v}}-1)$ edge DoFs and $3$ vertices DoFs. When $r^e=-1$, all $3$ components of edge DoFs of a vector  are removed and when $r^{\texttt{v}} = -1$, all $3$ components of a vector are removed.
\end{proof}

\section{Div Stability between Finite Elements Spaces}\label{sec:divstability}

\subsection{Overview}
To simplify the notation, introduce $r\ominus n := \max\{r-n, -1\}$ so that the result will stagnate at $-1$ when $r-n \leq -1$. 
In this section, for two smoothness parameters $(\bs r_2, \bs r_3)$ with relation $\bs r_3 \geq \bs r_2\ominus 1$, we aim to prove the so-called {\em div stability}, i.e. the div operator is surjective
\begin{equation}\label{eq:div}
\div\mathbb V^{\div}_{k}(\boldsymbol{r}_2)=\mathbb V^{L^2}_{k-1}(\boldsymbol{r}_3).   
\end{equation}
Additional conditions on $\bs r_2, \bs r_3$ are needed to prove the div stability \eqref{eq:div}. For example, $\bs r_2 = 0, \bs r_3 = -1$ is the notorious Stokes finite element pair for which the div stability is hard to verify and may require additional conditions on the triangulation. While $\bs r_2 = -1, \bs r_3 = -1$ corresponds to the div stability for the BDM element which is relatively easy. We shall call $(\bs r_2, \bs r_3)$ a div stable pair if \eqref{eq:div} holds.

\begin{table}[htp]
	\centering
	\caption{Examples of div stability for $\bs r_3 = \bs r_2\ominus 1$.}
	\renewcommand{\arraystretch}{1.5}
	\begin{tabular}{@{} c c  c c @{}}
	\toprule
		 $(r_2^{\texttt{v}}, r_2^e, r_2^f)$ &  $(r_3^{\texttt{v}}, r_3^e, r_3^f)$ & Results & Stronger Constraint \\

 \hline

$(r_2^{\texttt{v}}, -1,-1)$ & $(r_3^{\texttt{v}}, -1,-1)$ & Lemma \ref{lm:divbubbleontor2em1}\\
$(1, 0,-1)$ & $(0, -1,-1)$ & Lemma \ref{lm:divbubbleontor2fm1} & $r_{2}^{\texttt{v}}\geq 1$\\
$(2, 1, 0)$ & $(1, 0,-1)$ & Lemma \ref{lm:divbubbleontosimple1} & $r_{2}^{e}\geq 2r_2^f + 1$\\
$(2, 1, -1)$ & $(1, 0,-1)$ & Corollary \ref{cor:divbubbleontosimple1} & $r_{2}^{e}\geq 1, r_2^f =- 1$\\
$(0, 0, -1)$ & $(-1, -1,-1)$ & Not valid 
\medskip \\

\bottomrule
	\end{tabular}
	\label{table:derhamexamples}
\end{table}

The essential difficulty is the div stability of bubble spaces
$$
\div\mathbb B_{k}^{\div}(T;\boldsymbol{r}_2)=\mathbb B_{k-1}(T; \boldsymbol{r}_3)/\mathbb R,
$$
where $\mathbb B_{k-1}(T; \boldsymbol{r}_3)/\mathbb R=\mathbb B_{k-1}(T; \boldsymbol{r}_3)\cap L_0^2(T)$.
Let us refine the notation $S_{\ell} (f, \bs r)$ to $S_{\ell} (f, \bs r, k)$ to include the degree of polynomial. When $\dim\mathbb B_{k-1}(T; \boldsymbol{r}_3)=1$, we have $\div\mathbb B_{k}^{\div}(T;\boldsymbol{r}_2)=\mathbb B_{k-1}(T; \boldsymbol{r}_3)\cap L_0^2(T)$ as $\mathbb B_{k-1}(T; \boldsymbol{r}_3)\cap L_0^2(T)=\{0\}$. So we only consider the case $\dim\mathbb B_{k-1}(T; \boldsymbol{r}_3)>1$, i.e., $|S_3(T,\boldsymbol{r}_3, k-1) |\geq 2$.

Similar to Lemma \ref{lm:L20}, we have
$$
\mathbb B_{k-1}(T; \boldsymbol{r}_3)\cap L_0^2(T)=\mathrm{span}\{\lambda^{\alpha}/\alpha!-\lambda^{\beta}/\beta!: \alpha, \beta\in S_3(T,\boldsymbol{r}_3, k-1), \dist (\alpha,\beta) = 1\},
$$
as the sub-graph $\mathcal G(S_3(T,\boldsymbol{r}_3,k-1))$ is connected. 
It suffices to prove that: given $$p(\alpha,\beta) = \lambda^{\alpha}/\alpha!-\lambda^{\beta}/\beta!, \quad \alpha, \beta\in S_{3}(T,\boldsymbol{r}_3,k-1), \dist(\alpha,\beta)=1,$$ we can find a function 
$$
\bs u\in\mathbb B_{k}^{\div}(T;\boldsymbol{r}_2)\quad \text{s.t. } \div \bs u =  p.
$$ 
Without loss of generality, we assume 
$$
\beta=\alpha+\epsilon_{01}=(\alpha_0+1,\alpha_1-1,\alpha_2,\alpha_3).
$$

\subsection{Div stability of bubble spaces}
We start from a simple case $r_{2}^e=-1, r_{2}^f=-1$ as tangential components on edges and faces are included in the div bubble space; see \eqref{eq:bubbledecomprfn1}. 
\begin{lemma}\label{lm:divbubbleontor2em1}
Assume 
\begin{equation*}
k\geq\max\{2r_{2}^{\texttt{v}}+1,1\}, \quad r_{2}^{\texttt{v}}\geq-1, \quad r_{2}^e=-1, \quad r_{2}^f=-1, 
\end{equation*}
and $\bs r_3 = \bs r_2 \ominus  1$. Assume $\dim\mathbb B_{k-1}(T; \boldsymbol{r}_3)\geq1$.
It holds that
$$
\div\mathbb B_{k}^{\div}(T;\boldsymbol{r}_2)=\mathbb B_{k-1}(T; \boldsymbol{r}_3)/\mathbb R.
$$
\end{lemma}
\begin{proof}
Taking
 $\bs u = \lambda^{\alpha+\epsilon_{0}}\boldsymbol{t}_{1,0}/(\beta!\alpha_1)$, we have $\div \bs u = p$. The edge div bubble function $\lambda_0\lambda_1\boldsymbol{t}_{1,0}\in \boldsymbol{H}_0(\div,T)$. Writing $\lambda^{\alpha + \epsilon_0}\boldsymbol{t}_{1,0}=(\lambda_0^{\alpha_0}\lambda_1^{\beta_1}\lambda_2^{\alpha_2}\lambda_3^{\alpha_3})\lambda_0\lambda_1\boldsymbol{t}_{1,0}$, we conclude $\lambda^{\alpha + \epsilon_0}\boldsymbol{t}_{1,0}\in \boldsymbol{H}_0(\div,T)$. Next we verify $\alpha + \epsilon_0\in S_{3}(T,\boldsymbol{r}_2,k)$ by considering the distance to vertices as follows
$$
\dist( \alpha+\epsilon_{0}, \texttt{v}_i) = \dist( \alpha, \texttt{v}_i)+1 > r_2^{\texttt{v}}\quad\textrm{ for } i=1,2,3,
$$
$$
\dist( \alpha+\epsilon_{0}, \texttt{v}_0) =\dist(\alpha, \texttt{v}_0) =\dist( \beta, \texttt{v}_0)+1 > r_2^{\texttt{v}}.
$$
\end{proof}

Next we set $r_{2}^e= 0$. The tangential component of edge bubbles will be excluded from $\mathbb B_{k}^{\div}(T;\boldsymbol{r}_2)$. The nodes should be away from edges which in turn requires condition $r_{2}^{\texttt{v}}\geq 1$ stronger than the standard one $r_{2}^{\texttt{v}}\geq 2r_2^e\geq 0$. 

\begin{lemma}\label{lm:divbubbleontor2fm1}
Assume 
\begin{equation*}
k\geq2r_{2}^{\texttt{v}}+1, \quad r_{2}^{\texttt{v}}\geq 1, \quad r_{2}^e=0, \quad r_{2}^f=-1, 
\end{equation*}
and $\bs r_3 = \bs r_2\ominus 1$. Assume $\dim\mathbb B_{k-1}(T; \boldsymbol{r}_3)\geq1$.
It holds that
$$
\div\mathbb B_{k}^{\div}(T;\boldsymbol{r}_2)=\mathbb B_{k-1}(T; \boldsymbol{r}_3)/\mathbb R.
$$
\end{lemma}
\begin{figure}[htbp]
\begin{center}
\includegraphics[width=6cm]{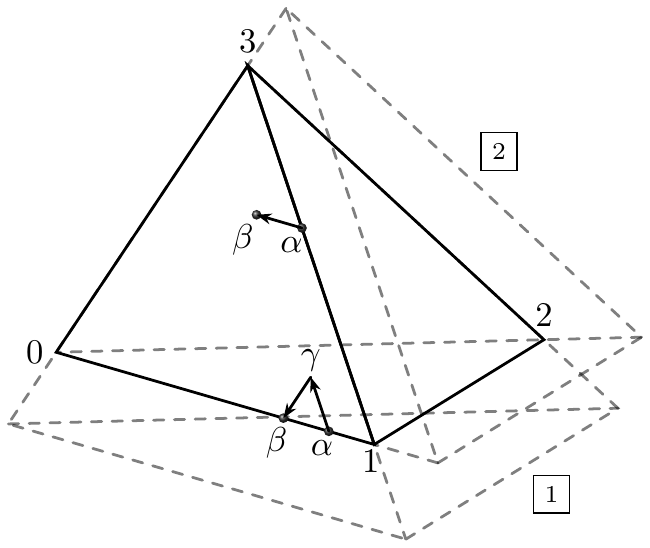}
\caption{Two cases when $r_{2}^{\texttt{v}}\geq 1, r_{2}^e=0, r_{2}^f=-1$: \step 1 $\alpha_2=\alpha_3=0$; \step 2 $\alpha_2+\alpha_3\geq1$. The dash line represents different extension of the simplicial lattice so that lattice nodes $\alpha, \beta$ are away from the edges in the extended lattice.}
\label{fig:edgecontinuous}
\end{center}
\end{figure}
\begin{proof}
\step 1 Consider case $\alpha_2=\alpha_3=0$ and consequently $\alpha_0 + \alpha_1 = k-1$. In this case, $\alpha,\beta\in f_{01}$, and $\dist(\alpha,\texttt{v}_i)> r_3^{\texttt{v}};$ See Fig. \ref{fig:edgecontinuous}.

By Lemma \ref{lm:alphabetagamma},  we can choose
\begin{equation}\label{eq:ualphabetagamma}
\bs u = \frac{1}{\gamma_!\alpha_1} \lambda^{\alpha + \epsilon_3}\boldsymbol{t}_{1,3}+\frac{1}{\gamma_!\beta_0}\lambda^{\beta + \epsilon_3}\boldsymbol{t}_{3,0}, \quad \gamma=\alpha+\epsilon_{31},
\end{equation}
and verify $\lambda^{\alpha + \epsilon_3}\boldsymbol{t}_{1,3}, \lambda^{\beta + \epsilon_3}\boldsymbol{t}_{3,0}\in \mathbb B_{k}^{\div}(T;\boldsymbol{r}_2)$. We focus on $\lambda^{\alpha + \epsilon_3}\boldsymbol{t}_{1,3}$ as $\lambda^{\beta + \epsilon_3}\boldsymbol{t}_{3,0}$ is symbolically identical. Write $\lambda^{\alpha + \epsilon_3}\boldsymbol{t}_{1,3}=(\lambda_0^{\alpha_0}\lambda_1^{\beta_1}\lambda_2^{\alpha_2}\lambda_3^{\alpha_3})\lambda_1\lambda_3\boldsymbol{t}_{1,3}$. As $\lambda_1\lambda_3\boldsymbol{t}_{1,3}\in \boldsymbol{H}_0(\div,T)$, we conclude $\lambda^{\alpha + \epsilon_3}\boldsymbol{t}_{1,3}\in \boldsymbol{H}_0(\div,T)$. Next we verify $\alpha + \epsilon_3\in S_{3}(T,\boldsymbol{r}_2,k).$

\smallskip
\noindent {\em Distance to vertices.} For vertices on the plane $f_{012}$, the distance is increased by $1$ as $\alpha_3\to \alpha_3 + 1$. Then  
$$
\dist( \alpha+\epsilon_3, \texttt{v}_i) = \dist( \alpha, \texttt{v}_i)+1 > r_3^{\texttt{v}}+1 = r_2^{\texttt{v}}\quad\textrm{ for } i=0,1,2.
$$
The distance to $\texttt{v}_3$ is $k-1$ which is far larger than $r_2^{\texttt{v}}$:
$$
\dist( \alpha+\epsilon_3, \texttt{v}_3  ) = \alpha_0 + \alpha_1 + \alpha_2 = k-1 > r_2^{\texttt{v}}.
$$

\smallskip
\noindent {\em Distance to edges.}  Similarly when computing the distance to edges not containing $\texttt{v}_3$, $\alpha_3 \to \alpha_3 + 1$ will increase the distance by $1$:
$$
\dist( \alpha+\epsilon_3, f_{ij}) = \dist( \alpha, f_{ij})+1\geq 1>r_2^e=0 \quad\textrm{ for }\; i,j\neq3.
$$
When computing the distance to $f_{03}$, we use the distance to vertices on face $f_{012}$: 
$$
\dist (\alpha, \texttt{v}_0) = \alpha_1 + \alpha_2 + \alpha_3 = \alpha_1 > r_3^{\texttt{v}}\geq 0.
$$ 
Then
$$
\dist( \alpha+\epsilon_3, f_{03}) = \alpha_1 + \alpha_2 = \alpha_1 > 0 = r_2^e.
$$
The bound $\dist( \alpha+\epsilon_3, f_{13}) > 0$ is similar. The distance to edge $f_{23}$ is far away as
$$
\dist( \alpha+\epsilon_3, f_{23}) =  \alpha_0 + \alpha_1 = k-1 \geq 2r_2^{\texttt{v}} >0.
$$

\medskip

\step 2 Consider case $\alpha_2+\alpha_3\geq1$. Namely $\dist (\alpha, f_{01}) \geq 1$. 
Setting
 $\bs u = \lambda^{\alpha+\epsilon_{0}}\boldsymbol{t}_{1,0}/(\beta!\alpha_1)$ and by Lemma~\ref{lm:alphabeta}, we have $\div \bs u = p$. Again we have $\lambda^{\alpha + \epsilon_{0}}\boldsymbol{t}_{1,0}\in \boldsymbol{H}_0(\div,T)$. Only need to show $\alpha+\epsilon_{0}\in S_2(T,\bs r_2, k)$. The simplical lattice containing $\alpha+\epsilon_{0}$ is extended in $\alpha_0$ direction; see Fig. \ref{fig:edgecontinuous}. 

\smallskip
\noindent {\em Distance to vertices.} We have
$$
\dist( \alpha+\epsilon_{0}, \texttt{v}_i) = \dist( \alpha, \texttt{v}_i)+1 > r_3^{\texttt{v}}+1 = r_2^{\texttt{v}}\quad\textrm{ for } i=1,2,3,
$$
$$
\dist( \alpha+\epsilon_{0}, \texttt{v}_0) =\dist(\alpha, \texttt{v}_0) =\dist( \beta, \texttt{v}_0)+1 > r_3^{\texttt{v}}+1 = r_2^{\texttt{v}}.
$$

\smallskip
\noindent {\em Distance to edges.}  We have
\begin{align*}
\dist( \alpha+\epsilon_{0}, f_{ij}) &= \dist( \alpha, f_{ij})+1> r_3^e + 1\geq1  > r_2^e=0 \quad\textrm{ for }\; 1\leq i,j\leq3,\\
\dist( \alpha+\epsilon_{0}, f_{01}) &= \alpha_2+\alpha_3\geq1  > r_2^e,\\
\dist( \alpha+\epsilon_{0}, f_{0i}) &= \dist( \alpha, f_{0i})= \dist(\beta, f_{0i})+1\geq1  > r_2^e\quad\textrm{ for }\; i=2,3.
\end{align*}
\end{proof}

We then move to the most difficult case: the velocity is continuous and the pressure is discontinuous. Super-smoothness on vertices and edges is added to ensure the div stability. In the following, $\bs r_2\geq (2,1,0)$ and $k\geq 5$. 

\begin{lemma}\label{lm:divbubbleontosimple1}
Assume 
\begin{equation}\label{eq:boundr2}
k\geq2r_{2}^{\texttt{v}}+1, \quad r_{2}^{\texttt{v}}\geq2r_{2}^e, \quad r_{2}^e\geq2r_{2}^f+1, \quad r_{2}^f\geq0, 
\end{equation}
and $\bs r_3 = \bs r_2 - 1$. Assume $\dim\mathbb B_{k-1}(T; \boldsymbol{r}_3)\geq1$.
It holds that
$$
\div\mathbb B_{k}^3(T;\boldsymbol{r}_2)=\mathbb B_{k-1}(T; \boldsymbol{r}_3)/\mathbb R.
$$
\end{lemma}
\begin{proof}

Without loss of generality, assume $\beta=\alpha+\epsilon_{01}=(\alpha_0+1,\alpha_1-1,\alpha_2,\alpha_3) \in S_{3}(T,\boldsymbol{r}_3,k-1)$ in the lattice for the pressure. We shall sort the nodes by the distance to the edge $f_{01}$, i.e., the plane $L(f_{01}, s)$ from $s = r_2^e$ to $k-1-r_2^e$; see Fig. \ref{fig:bubblediv}. 

\medskip

\step 1 
We first consider the case: $\alpha_3 = r_2^f$ or $\alpha_2 = r_2^f$. Without loss of generality, we discuss $\alpha_3 = r_2^f$ in detail. To push the node into $S_{3}(T, \bs r_2, k)$, we need to increase the distance to the face $f_3 = f_{012}$ by one, i.e, lift the nodes one level higher in $\alpha_3$ direction by changing $\alpha$ to $\alpha + \epsilon_3$; see Fig. \ref{fig:bubblediv}. 

It remains to verify $\bs u\in \mathbb B_{k}^3(T;\boldsymbol{r}_2)$, i.e., $\alpha+\epsilon_3, \beta+\epsilon_3 \in S_{3}(T, \bs r_2, k)\subset \mathbb T^n_{k}$. We focus on $\alpha+\epsilon_3$ as $\beta+\epsilon_3$ is symbolically identical.  

\medskip

\noindent {\em Distance to vertices.} For $\alpha\in S_{3}(T,\boldsymbol{r}_3,k-1)\subset \mathbb T^n_{k-1}$,  $i=0,1,2$, 
$$
\dist (\alpha, \texttt{v}_i) > r_3^{\texttt{v}} \iff \alpha_i \leq k -1 - (r_3^{\texttt{v}} +1) = k-1 - r_2^{\texttt{v}} \iff \dist(\alpha+\epsilon_3, \texttt{v}_i ) >  r_2^{\texttt{v}}.
$$ 
So only $\texttt{v}_3$ is left. As we assume $\alpha_3 = r_2^f$, 
$$
\dist( \alpha+\epsilon_3, \texttt{v}_3  ) = k - \dist( \alpha+\epsilon_3, f_3  ) = k - (\alpha_3+1) = k - r_2^f -1 \geq 2r_2^{\texttt{v}} - r_2^f > r_2^{\texttt{v}}.
$$

\smallskip
\noindent {\em Distance to edges}. For edges on the face $f_{012}$, w.l.o.g. take edge $f_{01}$, as no $\alpha_3$ involved, we have the equivalence of the bound 
$$
\dist(\alpha, f_{01}) > r^e_3 \iff \alpha_0 + \alpha_1 \leq k-1 - (r^e_3+1) \iff \dist(\alpha + \epsilon_3, f_{01}) > r^e_2.
$$
To estimate the distance to other edges, w.l.o.g. consider the edge $f_{03}$, we use the bound of the distance to vertices. From
$$
\dist(\alpha, \texttt{v}_0) > r_3^{\texttt{v}} \iff \alpha_1 + \alpha_2 + \alpha_3 > r_3^{\texttt{v}},
$$ 
the fact $\alpha_3 = r_2^f$, and the bound \eqref{eq:boundr2} on $\bs r_2$, we conclude 
$$
\dist(\alpha + \epsilon_3, f_{03})  = \alpha_1 + \alpha_2 > r_3^{\texttt{v}} - r_2^f \geq r_2^e.
$$

\smallskip
\noindent {\em Distance to faces}. Obviously the distance to $f_3 = f_{012}$ is increased by $1$, i.e.
$$
\dist ( \alpha + \epsilon_3, f_3) = \alpha_3 + 1 = r_2^f+1 > r_2^f.
$$ 
But other $\alpha_i, i\neq 3,$ remains unchanged. We will use the bound of the distance to edges. Again w.l.o.g. consider face $f_2 = f_{013}$. From 
$$
\dist(\alpha, f_{01}) = \alpha_2 + \alpha_3 > r_3^e = r_2^e - 1,
$$
the fact $\alpha_3 = r_2^f$, and the bound \eqref{eq:boundr2} on $\bs r_2$, we conclude 
$$
\dist( \alpha + \epsilon_3, f_2 ) = \alpha_2 > r_2^e - r_2^f - 1\geq r_2^f.
$$
The last inequality is the motivation to have the stronger constraint $r_2^e\geq 2r_2^f + 1$ in \eqref{eq:boundr2}.

\begin{figure}[htbp]
\begin{center}
\includegraphics[width=2.8in]{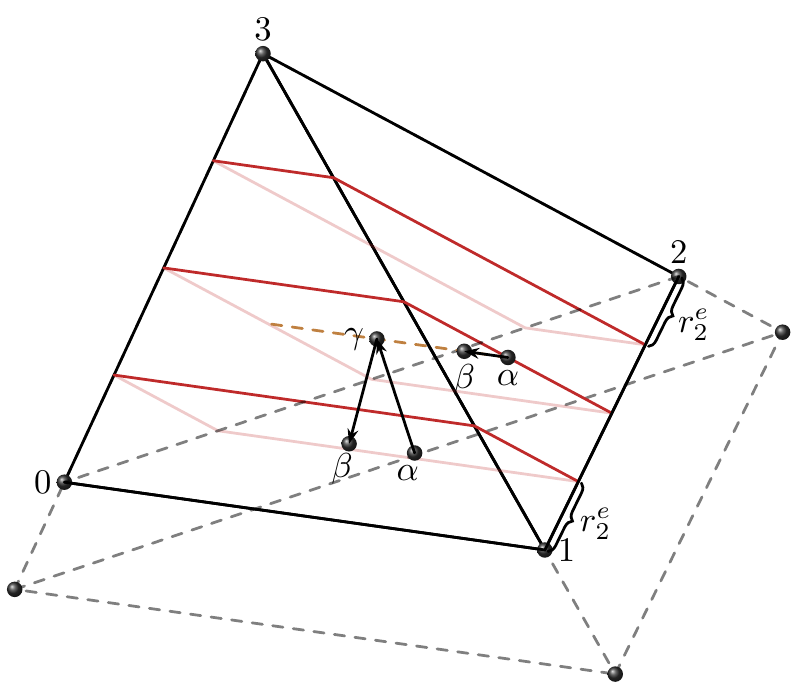}
\caption{Different location of $\alpha, \beta$.}
\label{fig:bubblediv}
\end{center}
\end{figure}

In summary, for vertices and edges on the face $f_{012}$, the upper bound on the sum of indices  automatically holds as no $\alpha_3$ is involved. 
When estimate the distance to edges and faces containing $\texttt{v}_3$, we use the fact $\alpha\in L(f_3, r_2^f)$ is in the face bubble $S_{2}( f_3,\begin{pmatrix}
r_3^{\texttt{v}} \\
r_3^e
\end{pmatrix}-r_2^f, k-1 - r_2^f)$;
see Fig. \ref{fig:facebubble} (c). 

\medskip

\begin{figure}[htbp]
\begin{center}
\includegraphics[width=8cm]{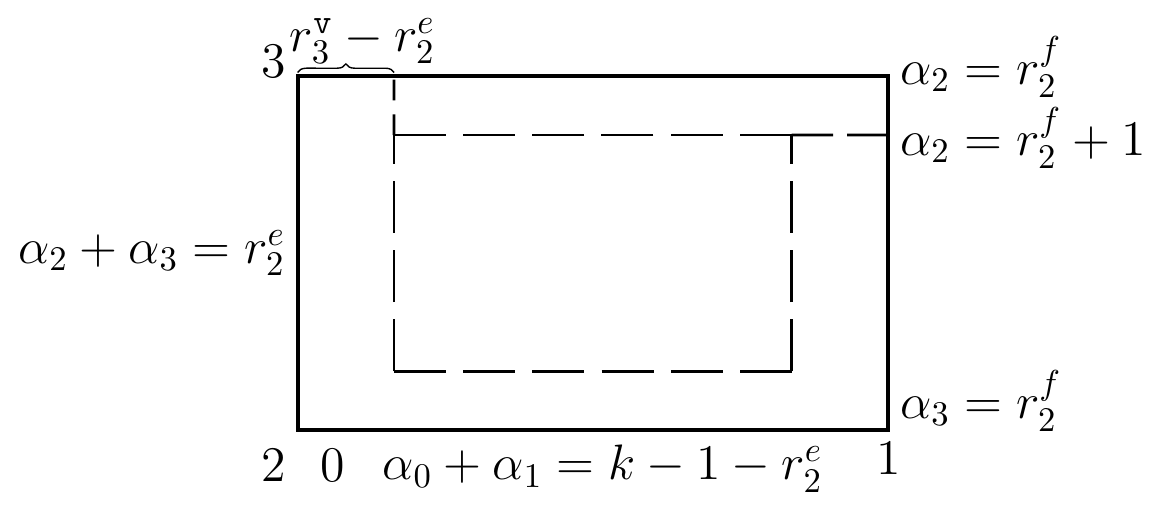}
\caption{On the cut plane $\alpha_2 + \alpha_3 = r_2^e$.}
\label{fig:case2}
\end{center}
\end{figure}

\step 2 Consider case $\alpha_2+\alpha_3=r_{2}^e$ and consequently $\alpha_0 + \alpha_1 = k-1-r_2^e$. As we have considered the case $\alpha_2 = r_2^f$ or $\alpha_3 = r_2^f$ in Case 1, we can further assume $\alpha_2, \alpha_3\geq r_2^f + 1> r_2^f$. Notice that now $r_{2}^e = \alpha_2+\alpha_3 \geq2r_2^f+2$ which means if $r_{2}^e = 2r_2^f+1$ only, either $\alpha_2$ or $\alpha_3 = r_2^f$ which is covered by Case 1.

We still choose $\bs u$ by \eqref{eq:ualphabetagamma} and verify $\alpha+\epsilon_3, \beta+\epsilon_3 \in S_{3}(T, \bs r_2, k)$.

\medskip

\noindent {\em Distance to vertices.} Again we only need to consider the distance to $\texttt{v}_3$ which is minimized when $\alpha_2 = r_2^f+1> r_2^f$; see Fig. \ref{fig:case2}. Algebraically, we have
$$
\dist( \alpha+\epsilon_3, \texttt{v}_3  ) = \alpha_0 + \alpha_1 + \alpha_2 > \alpha_0 + \alpha_1 + r_2^f = k-1 - r_2^e + r_2^f \geq r_2^{\texttt{v}}.
$$

\smallskip
\noindent {\em Distance to edges.}  The distance to edge $f_{01}$ and $f_{23}$ is easy to bound as $\alpha_2+\alpha_3=r_{2}^e$:
\begin{align*}
\dist( \alpha+\epsilon_3, f_{01}) &= \alpha_2 + \alpha_3 + 1 = r_2^e+1 > r_2^e, \\
\dist( \alpha+\epsilon_3, f_{23}) &= \alpha_0 + \alpha_1 = k - 1 - r_2^e  > r_2^e.
\end{align*}
Without loss of generality consider $\dist ( \alpha+\epsilon_3, f_{13}) = \alpha_0 + \alpha_2$. From the distance to the vertex and the fact $\alpha_2 + \alpha_3 = r_2^e$, we have the lower bound
\begin{equation}\label{eq:alpha0lowerbound}
\dist (\alpha, \texttt{v}_1) = \alpha_0 + \alpha_2 + \alpha_3 > r_3^{\texttt{v}} \quad \Longrightarrow \quad \alpha_0 > r_3^{\texttt{v}} - r_2^e.
\end{equation}
Together with $\alpha_2\geq r_2^f+1$, we have
$$
\dist ( \alpha+\epsilon_3, f_{13}) = \alpha_0 + \alpha_2 > r_3^{\texttt{v}} - r_2^e + r_2^f+1\geq r_2^e.
$$

\smallskip
\noindent {\em Distance to faces.}  As we assume $\alpha_2, \alpha_3 > r_2^f$, we have
$$
\dist ( \alpha+\epsilon_3, f_i) > r_2^f, \quad i = 2,3. 
$$
The lower bound \eqref{eq:alpha0lowerbound} implies
$$
\dist ( \alpha+\epsilon_3, f_0) = \alpha_0 > r_3^{\texttt{v}} - r_2^e\geq r_2^f .
$$
Similar to \eqref{eq:alpha0lowerbound}, $\dist ( \alpha, \texttt{v}_{0}) > r_3^{\texttt{v}}$ will imply the same lower bound on $\alpha_1$ and $\dist ( \alpha+\epsilon_3, f_1) > r_2^f$. 

In summary, when considering the set $L(f_{01}, r_2^e)\cap S_3(T, \bs r_3, k-1)$, the index is well separated from the boundary. See the dash-line in Fig. \ref{fig:case2}.


\begin{figure}[htbp]
\begin{center}
\includegraphics[width=6cm]{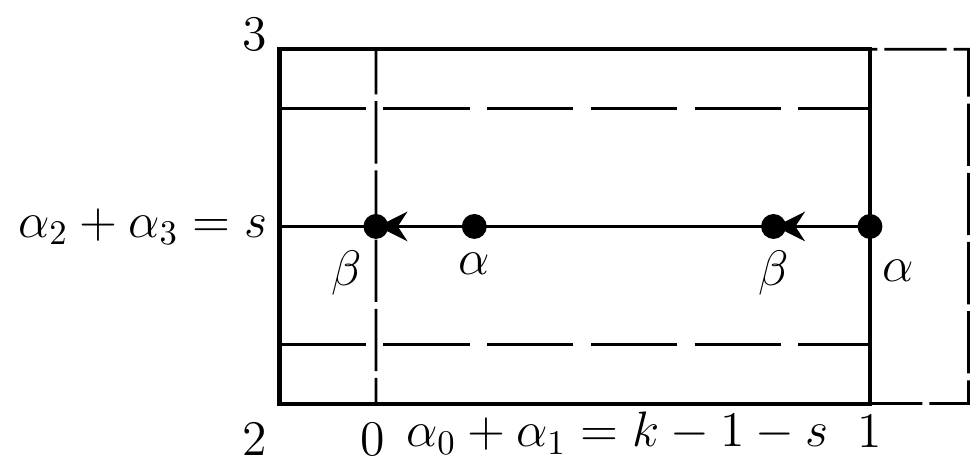}
\caption{On the cut plane $\alpha_2 + \alpha_3 = s \geq r_2^e + 1$.}
\label{fig:case3}
\end{center}
\end{figure}

\step 3 Now only the case $\{L(f_{01}, s), r_{2}^e + 1 \leq s  \leq k -1 - r_2^e \cap S_3(T, \bs r_3, k-1) \}$ is left which implies $\alpha_2 + \alpha_3 \geq r_{2}^e + 1$. See the middle cut plane in Fig. \ref{fig:bubblediv} (a).  

%
%

After Case 1, we can  assume $\alpha_2, \alpha_3 \geq r_2^f+1> r_2^f$. The node $\alpha$ can be on the plane $L(f_{123}, r_2^f)$ and thus lift in $\alpha_3$ direction will not push it into the interior. We choose to extend in $\alpha_0$ direction and set
 $\bs u = \lambda^{\alpha+\epsilon_{0}}\boldsymbol{t}_{1,0}/(\beta!\alpha_1)$. By Lemma \ref{lm:alphabeta}, $\div \bs u = p$. We verify $\alpha+\epsilon_{0}\in S_{3}(T, \bs r_2, k)$. 


\medskip

\noindent {\em Distance to vertices.} The trouble case is $\dist ( \alpha+\epsilon_{0}, \texttt{v}_0)$ as $\alpha_0+1$ is closer to $\texttt{v}_0$. We will use the fact that $\beta$ is closer to $\texttt{v}_0$, i.e., $\beta_0 = \alpha_0 +1$, and 
$$
\dist(\beta, \texttt{v}_0) > r_3^{\texttt{v}} \Longrightarrow k-1-\beta_0 > r_3^{\texttt{v}},
$$
to conclude the desired bound
$$
\dist ( \alpha+\epsilon_{0}, \texttt{v}_0) = \alpha_1+\alpha_2+\alpha_3 = k-\beta_0 > r_2^{\texttt{v}}.
$$

\smallskip
\noindent {\em Distance to edges.}  For edges $f_{ij}$ not containing $\texttt{v}_0$, no change on $\alpha_i, \alpha_j$, and thus 
$$
\dist(\alpha, f_{ij}) > r_3^e \Longrightarrow \alpha_i + \alpha_j \leq k -1 - r_3^e - 1 = k - r_2^e - 1 \Longrightarrow \dist(\alpha +\epsilon_{0}, f_{ij}) > r_2^e.
$$
Consider edge $f_{01}$. By $\dist(\alpha, f_{01})  = \alpha_2 + \alpha_3 \geq r_2^e+1$, we have 
$$
\dist(\alpha +\epsilon_{0}, f_{01}) = \alpha_2 + \alpha_3 \geq r_2^e + 1 > r_2^e. 
$$
Consider edge $f_{03}$. We use bound for $\beta$
$$
\dist (\beta, f_{03}) > r_3^e \Longrightarrow \beta_0 + \beta_3 = \alpha_0 + 1 + \alpha_3 \leq k-1- r_3^e - 1 = k-1- r_2^e,
$$
to conclude $\dist( \alpha+\epsilon_{0}, f_{03}) > r_2^e$. The  $\dist( \alpha+\epsilon_{0}, f_{02}) $ is similar. 

\smallskip
\noindent {\em Distance to faces.}  Again the distance to $f_0$ is easy as
$$
\dist (\alpha + \epsilon_{0}, f_0) = \alpha_0 + 1 > r_3^f + 1 = r_2^f.
$$ 
The distance to faces $f_2$ and $f_3$ is from the assumption $\alpha_2, \alpha_3 \geq r_2^f+1> r_2^f$. So only need to check $\dist (\alpha + \epsilon_{0}, f_1) = \alpha_1$.  Again we compare with $\beta$. By $\dist (\beta, f_1) = \beta_1 > r_3^f$, we have
$$
\dist (\alpha + \epsilon_{0}, f_1) = \alpha_1 = \beta_1+1 > r_3^f + 1 = r_2^f.
$$

In summary, when $\alpha_2+\alpha_3 \geq r_2^e+1$, we can choose a simple velocity field from $\alpha$ to $\beta$ and use the distance bound of $\beta$ to derive the desired distance bound of $\alpha$; see Fig. \ref{fig:case3}.
\end{proof}


When changing $r_2^f=0$ to $r_2^f=-1$, we add more tangential div bubble functions into the bubble space $\mathbb B_{k}^{\div}(T;\boldsymbol{r}_2)$ and thus the following $\div$ stability result is trivial.
\begin{corollary}\label{cor:divbubbleontosimple1}
Assume 
\begin{equation*}
k\geq2r_{2}^{\texttt{v}}+1, \quad r_{2}^{\texttt{v}}\geq2r_{2}^e, \quad r_{2}^e\geq 1, \quad r_{2}^f=-1, 
\end{equation*}
and $\bs r_3 = \bs r_2\ominus 1$. Assume $\dim\mathbb B_{k-1}(T; \boldsymbol{r}_3)\geq1$.
It holds that
$$
\div\mathbb B_{k}^{\div}(T;\boldsymbol{r}_2)=\mathbb B_{k-1}(T; \boldsymbol{r}_3)/\mathbb R.
$$
\end{corollary}
\begin{proof}
Noting that $\mathbb B_{k}^3(T; (\bs r_2)_+)\subset\mathbb B_{k}^{\div}(T;\boldsymbol{r}_2)$, we have
$$
\div\mathbb B_{k}^3(T;(\bs r_2)_+)\subseteq\div\mathbb B_{k}^{\div}(T;\boldsymbol{r}_2)\subseteq\mathbb B_{k-1}(T; \boldsymbol{r}_3)/\mathbb R.
$$
By Lemma~\ref{lm:divbubbleontosimple1}, $\div\mathbb B_{k}^3(T;(\bs r_2)_+)=\mathbb B_{k-1}(T; \boldsymbol{r}_3)/\mathbb R$, which ends the proof.
\end{proof}

Integrate results in Lemmas~\ref{lm:divbubbleontor2em1}-\ref{lm:divbubbleontor2fm1}, Lemma~\ref{lm:divbubbleontosimple1} and Corollary~\ref{cor:divbubbleontosimple1} into the following theorem.
\begin{theorem}\label{thm:divbubbleonto}
Assume $k\geq\max\{2r_{2}^{\texttt{v}}+1,1\},$
\begin{equation}\label{eq:boundr2fordivbubble}
\begin{cases}
 r_2^f\geq 0, & \,r_{2}^e\geq2r_{2}^f+1, \quad r_{2}^{\texttt{v}}\geq 2r_{2}^e,\\
 r_2^f = -1, & 
\begin{cases}
 r_2^e \geq 1, & r_{2}^{\texttt{v}}\geq 2r_{2}^e,\\
 r_2^e \in \{0, -1\}, & r_{2}^{\texttt{v}}\geq 2r_{2}^e + 1,
\end{cases}
\end{cases}
\end{equation}
and $\bs r_3 = \bs r_2\ominus 1$. Assume $\dim\mathbb B_{k-1}(T; \boldsymbol{r}_3)\geq1$.
It holds that
\begin{equation*}
\div\mathbb B_{k}^{\div}(T;\boldsymbol{r}_2)=\mathbb B_{k-1}(T; \boldsymbol{r}_3)/\mathbb R.
\end{equation*}
\end{theorem}

\subsection{Div stability with equality constraint}
For simplicity, hereafter we will omit the triangulation $\mathcal T_h$ in the notation of global finite element spaces. 
For example, $\mathbb V^{\div}_{k}(\mathcal T_h; \boldsymbol{r}_2)$ will be abbreviated as $\mathbb V^{\div}_{k}(\boldsymbol{r}_2)$.

\begin{lemma}
Let $(\bs r_2, \bs r_3), \bs r_3 = \bs r_2\ominus 1,$ and $\bs r_2$ satisfy \eqref{eq:boundr2fordivbubble}. Assume $k$ is a large enough integer satisfying $k\geq\max\{2r_{2}^{\texttt{v}}+1,1\},$ $\dim\mathbb B_{k-1}(T; \boldsymbol{r}_3)\geq1$ for all $T\in \mathcal T_h$, and $\dim\mathbb B_{k}(f;\begin{pmatrix}
r_2^{\texttt{v}} \\
r_2^e
\end{pmatrix})\geq1$ for all $f\in \Delta_{2}(\mathcal T_h)$.
It holds that
\begin{equation}\label{eq:divonto3dsimple}
\div\mathbb V^{\div}_{k}(\boldsymbol{r}_2)=\mathbb V^{L^2}_{k-1}(\boldsymbol{r}_3).   
\end{equation}
\end{lemma}
\begin{proof}
It is obvious that $\div\mathbb V^{\div}_{k}(\boldsymbol{r}_2)\subseteq\mathbb V^{L^2}_{k-1}(\boldsymbol{r}_3)$. For $p_h\in \mathbb V^{L^2}_{k-1}(\boldsymbol{r}_3)\subset H^{r_{3}^f+1}(\Omega)$, we are going to construct $\boldsymbol{v}=(v_1, v_2, v_3)^{\intercal}\in\mathbb V^{\div}_{k}(\boldsymbol{r}_2)$ s.t. $\div \bs v = p_h$. To motivate the construction, consider 1-D case. Given values $p_h^{(j)}(x), j=0,\ldots, m$, to construct $u$ satisfying $u'=p_h$, we simply set $u^{(j+1)}(x) = p_h^{(j)}(x), j=0,\ldots, m$ and $u(x) = 0$. 

For vector function $\bs v$,  on each lower sub-simplex, we will choose a different frame and pick up one direction to assign the derivative relation. 

\noindent$\bullet\,$ For $\texttt{v}\in \Delta_{0}(\mathcal T_h)$, we use the default Cartesian coordinate and write $\bs v = (v_1, v_2, v_3)^{\intercal}$. Set 
\begin{align}
\nabla^j(\partial_1v_1)(\texttt{v})=\nabla^jp_h(\texttt{v}), & \quad j=0,\ldots, r_{3}^{\texttt{v}},  \label{eq:dof0}
\end{align}
and all other DoFs are zero.
Then 
\begin{equation}\label{eq:202110131}
\nabla^j(\div\boldsymbol{v})(\texttt{v})= \nabla^j(\partial_1v_1)(\texttt{v}) = \nabla^jp_h(\texttt{v}), \quad j=0,\ldots, r_{3}^{\texttt{v}}.   
\end{equation}

\smallskip

\noindent$\bullet\,$ 
For $e\in \Delta_{1}(\mathcal T_h)$, we use the frame $(\bs t, \bs n_1, \bs n_2)$, where $\bs t$ is a tangential vector of $e$ and $\bs n_1, \bs n_2$ are two linearly independent normal vectors of $e$. Set edge DoFs of $\bs v\cdot \bs t$ and $\bs v\cdot \bs n_2$ to zero. Together with DoFs \eqref{eq:dof0} on vertices, $\bs v\cdot \bs t|_e$ is determined. Then set the DoF for $\bs v\cdot \bs n_1$ by
\begin{align*}
\int_e \boldsymbol{v}\cdot\boldsymbol{n}_1\ q \dd s&=0,  \quad q\in \mathbb P_{k- 2(r_{2}^{\texttt{v}}+1)} (e), \quad\textrm{ if }\; r_2^e\geq0,\\
\int_e \frac{\partial^{j+1}(\boldsymbol{v}\cdot\boldsymbol{n}_1)}{\partial n_1\partial n_1^{i}\partial n_2^{j-i}}\ q \dd s&= \int_e \frac{\partial^jp_h}{\partial n_1^i\partial n_2^{j-i}}\ q \dd s 
- \int_e \frac{\partial^{j+1}(\boldsymbol{v}\cdot\boldsymbol{t})}{\partial t\, \partial n_1^i\partial n_2^{j-i}}\ q\dd s,  \\
& \quad\;\; q\in \mathbb P_{k- 2(r_{2}^{\texttt{v}}+1)+j+1} (e), 0\leq i\leq j\leq r_{3}^e.
\end{align*}
Then by $\div\boldsymbol{v}=\partial_{t}(\boldsymbol{v}\cdot\boldsymbol{t})+\partial_{n_1}(\boldsymbol{v}\cdot\boldsymbol{n}_1)+\partial_{n_2}(\boldsymbol{v}\cdot\boldsymbol{n}_2)$ and the vanishing edge DoFs for $\bs v\cdot \bs n_2$, we have
\begin{equation}\label{eq:202110132}
\int_e \frac{\partial^j(\div\boldsymbol{v}-p_h)}{\partial n_1^i\partial n_2^{j-i}}\ q \dd s=0,  \quad q\in \mathbb P_{k-1- 2(r_{3}^{\texttt{v}}+1)+j} (e), 0\leq i\leq j\leq r_{3}^e.
\end{equation}

\smallskip

\noindent$\bullet\,$  For $f\in \Delta_{2}(\mathcal T_h)$, we choose two tangential vectors $\bs t_1,\bs t_2$ and a normal vector $\bs n_f$ as the local frame. Set the face DoFs for $\boldsymbol{v}\cdot\boldsymbol{t}_i, i=1,2$ as zero. Together with edge and vertice DoFs, the tangential component $\Pi_f \partial_n^j\bs v$, for $j=0,\ldots,r_2^f$, is determined and thus $\partial_n^j(\div_f \bs v)|_f$ is well-defined. Then we set 
\begin{align}
\int_f \partial_{n}^{j}\partial_n(\boldsymbol{v}\cdot\boldsymbol{n})\ q \dd S&=\int_f\partial_{n}^j(p_h-\div_f\boldsymbol{v})\ q \dd S,   \notag\\
& \quad\qquad q\in \mathbb B_{k-1-j}(f; \begin{pmatrix}
r_3^{\texttt{v}} \\
r_3^e
\end{pmatrix}-j), j = 0,\ldots, r_{3}^f.\notag
\end{align}
Then by $\div\boldsymbol{v}=\partial_{n}(\boldsymbol{v}\cdot\boldsymbol{n})+\div_f\boldsymbol{v}$, we have
\begin{equation}\label{eq:202110133}
\int_f \partial_{n}^{j}(\div\boldsymbol{v}-p_h)\ q \dd S=0,  \quad q\in \mathbb B_{k-1-j}(f; \begin{pmatrix}
r_3^{\texttt{v}} \\
r_3^e
\end{pmatrix}-j), j = 0,\ldots, r_{3}^f.
\end{equation}
Notice that DoFs $\int_f \bs v\cdot \bs n \dd S$ remains open as we assume $\dim\mathbb B_{k}(f;\begin{pmatrix}
r_2^{\texttt{v}} \\
r_2^e
\end{pmatrix})\geq1$. Recall that there exists a $\widetilde{\boldsymbol{v}}\in \boldsymbol{H}^{r_{3}^f+2}(\Omega;\mathbb R^3)$ such that $\div\widetilde{\boldsymbol{v}}=p_h$. Then set
\begin{equation}
\int_f (\boldsymbol{v}\cdot\boldsymbol{n})\ q \dd S =\int_f (\widetilde{\boldsymbol{v}}\cdot\boldsymbol{n})\ q \dd S, \quad  q\in \mathbb P_0(f)\oplus\mathbb B_{k}(f;\begin{pmatrix}
r_2^{\texttt{v}} \\
r_2^e
\end{pmatrix})/\mathbb R. \label{eq:facedof}
\end{equation}
By $\div\widetilde{\boldsymbol{v}}=p_h$ and \eqref{eq:facedof},
\begin{equation}\label{eq:202110134-1}
\int_T(\div\boldsymbol{v}-p_h) \dx=\int_T\div(\boldsymbol{v}-\widetilde{\boldsymbol{v}}) \dx=0.
\end{equation}

\noindent$\bullet\,$ 
For $T\in \mathcal T_h$, we split the 
\begin{align*}
\int_T \div\boldsymbol{v}\, q \dx&=\int_T p_h\, q \dx, \quad q\in \mathbb B_{k-1}(\bs r_3)/\mathbb R,
\\
\int_T \boldsymbol{v}\cdot\boldsymbol{q} \dx&=0, \quad \boldsymbol{q}\in \mathbb B^{\div}_{k}(\boldsymbol{r}_2)\cap \ker(\div).
\end{align*}
By Theorem~\ref{thm:divbubbleonto}, the mapping div is surjective between bubble spaces. 
Together with \eqref{eq:202110134-1},
\begin{equation}\label{eq:202110134}
\int_T(\div\boldsymbol{v}-p_h)\ q \dx=0,  \quad q\in \mathbb B_{k-1}(\boldsymbol{r}_3).
\end{equation}
Finally combining \eqref{eq:202110131}-\eqref{eq:202110133} and \eqref{eq:202110134} and the uni-solvence for $\mathbb V^{L^2}_{k-1}(\boldsymbol{r}_3)$, we conclude $\div\boldsymbol{v}=p_h$.
\end{proof}

\begin{example}\rm
We choose 
$\boldsymbol{r}_2=(2,1,0)^{\intercal}, \boldsymbol{r}_3=(1,0,-1)^{\intercal}$
and polynomial degree $k\geq 6$, to get a stable Stokes-pair,
$$
(\mathbb V^{\div}_{k}(
\begin{pmatrix}
 2\\
 1\\
 0
\end{pmatrix}
), \mathbb V^{L^2}_{k-1}(\begin{pmatrix}
 1\\
 0\\
 -1
\end{pmatrix}
)).
$$ 
The pressure element is discontinuous on faces but continuous on edges and differentiable at vertices. Notice that the lower bound on $k$ is increased from $2r_2^{\texttt{v}}+ 1 = 5$ to $6$ to include a face bubble DoF so that $\div \bs u$ will contain piecewise constant. This is the Stokes element constructed by Neilan in~\cite{Neilan2015}. 
%
\end{example}

\begin{example}\rm
We choose 
$\boldsymbol{r}_2=(0, -1,-1)^{\intercal}, \boldsymbol{r}_3=(-1,-1,-1)^{\intercal}$
and polynomial degree $k\geq 2$, to get a stable pair for mixed Poisson problem,
$$
(\mathbb V^{\div}_{k}(
\begin{pmatrix}
 0\\
 -1\\
 -1
\end{pmatrix}
), \mathbb V^{L^2}_{k-1}(\begin{pmatrix}
 -1\\
 -1\\
 -1
\end{pmatrix}
)).
$$ 
The $H(\div)$-conforming element $\mathbb V^{\div}_{k}(
\begin{pmatrix}
 0\\
 -1\\
 -1
\end{pmatrix}
)$ is the so-called Stenberg's element~~\cite{Stenberg2010}. The lower bound  $k\geq 2$ is to include a face bubble DoF so that $\div \bs u$ will contain piecewise constant. For BDM element, i.e. $\boldsymbol{r}_2=\bs r_3 = \bs{-1}$, $k\geq 1$ is enough to ensure the div stability. 
\end{example}

\subsection{Div stability with inequality constraint}
We consider more general cases with an inequality constraint on the smoothness parameters $\bs r_2$ and $\bs r_3$:
\begin{equation*}
\bs r_2\geq -1, \quad \quad \boldsymbol{r}_3\geq \bs r_2\ominus 1.
\end{equation*}
To define the finite element spaces, we further require 
\begin{align*}
&k\geq\max\{2r_{2}^{\texttt{v}}+1,1\}, \quad \bs r_{2} \text{ satisfies } \eqref{eq:boundr2fordivbubble}, \\
&k-1 \geq 2r_{3}^{\texttt{v}}+1, \quad r_3^{\texttt{v}} \geq 2 r_3^e,  \quad r_3^e\geq 2r_3^f. & 
\end{align*}

As $\boldsymbol{r}_3\geq \bs r_2\ominus 1,$ we have the relation $\mathbb V^{L^2}_{k-1}(\bs r_3)\subseteq \mathbb V^{L^2}_{k-1}(\bs r_2\ominus 1)$. By the div stability~\eqref{eq:divonto3dsimple} established for the larger space $\mathbb V^{L^2}_{k-1}(\bs r_2\ominus 1)$, we can define a subspace 
$$
\mathbb V^{\div}_{k}(\bs r_2, \bs r_3) \subseteq \mathbb V^{\div}_{k}(\bs r_2), \quad s.t. \,  \div \mathbb V^{\div}_{k}(\bs r_2, \bs r_3) = \mathbb V^{L^2}_{k-1}(\bs r_3). 
$$
Such subspace $\mathbb V^{\div}_{k}(\bs r_2, \bs r_3)$ always exists. The difficulty is to give a finite element definition in terms of local DoFs. 

We use BDM element $\bs r_2 = -1$ as an example to explain the change of DoFs. Due to the div stability of the bubble space, we can write DoFs for BDM element as
\begin{align*}
\int_f \boldsymbol  v\cdot\boldsymbol n \, q \dd S,& \quad q\in  \mathbb P_k(f) = \mathbb B_{k}(f; \begin{pmatrix}
	r_2^{\texttt{v}} \\
	r_2^e
	\end{pmatrix}), f\in \Delta_2(T), \\
%
\int_T \div \boldsymbol v \, q \dx, &\quad \boldsymbol q\in \mathbb P_{k-1}(T)/\mathbb R = \mathbb B_{k-1}(T; \bs r_2)/\mathbb R,\\
%
\int_T \boldsymbol{v}\cdot\boldsymbol{q} \dx, &\quad \boldsymbol{q}\in \mathbb B^{\div}_{k}(T;\bs r_2)\cap \ker(\div).
\end{align*}
The range of div operator is the discontinuous $\mathbb P_{k-1}$ element. 
Now choose $\bs r_3\geq -1$, we increase smoothness of $\div \bs v$ on vertices, edges, and faces, and in turn shrink the interior moments:
\begin{align}
\label{eq:mBDMv}
\nabla^j\div\boldsymbol{v}(\texttt{v}),  & \quad j= 0,\ldots, r_3^{\texttt{v}}, \\
\label{eq:mBDMe}
\int_e \frac{\partial^{j}(\div\boldsymbol{v})}{\partial n_1^{i}\partial n_2^{j-i}} \, q \dd s, &\quad q\in \mathbb P_{k-1 - 2(r_3^{\texttt{v}} +1)+j}(e), 0\leq i\leq j\leq r_3^{e}, \\
\label{eq:mBDM1}
\int_f \boldsymbol  v\cdot\boldsymbol n \, q \dd S,& \quad q\in  \mathbb P_k(f) = \mathbb B_{k}(f; \begin{pmatrix}
	r_2^{\texttt{v}} \\
	r_2^e
	\end{pmatrix}), f\in \Delta_2(T), \\
\label{eq:mBDMf}
\int_f \partial_n^j(\div\boldsymbol{v})\ q \dd S, &\quad  q\in \mathbb B_{k-1 - j} (f;\begin{pmatrix}
r_3^{\texttt{v}} \\
r_3^e
\end{pmatrix}-j),  0\leq j\leq r_3^{f},\\
\label{eq:mBDM2}
\int_T \div \boldsymbol v \, q \dx, &\quad q\in \mathbb B_{k-1}(T; \boldsymbol{r}_3)/\mathbb R,\\
\label{eq:mBDM3}
\int_T \boldsymbol{v}\cdot\boldsymbol{q} \dx, &\quad \boldsymbol{q}\in \mathbb B^{\div}_{k}(T;\bs r_2)\cap \ker(\div).
\end{align}
The added DoFs \eqref{eq:mBDMv}, \eqref{eq:mBDMe}, and \eqref{eq:mBDMf}-\eqref{eq:mBDM2} on $\div \bs v$ will determine $\div \bs v$ up to a constant and the sum of number of these DoFs is always $\dim \mathbb P_{k-1}(T)-1$. The rest DoFs \eqref{eq:mBDM1} and \eqref{eq:mBDM3} are independent of $\bs r_3$. Hence the total number of DoFs remains unchanged. The uni-solvence is also easy as the modified DoFs is to determine $\div \bs v$. 

We then explain the general case.  
For an edge $e$, we use the frame $(\bs t, \bs n_1, \bs n_2)$, where $\bs t$ is a tangential vector of $e$ and $\bs n_1, \bs n_2$ are two linearly independent normal vectors of $e$. For a face $f$, we choose two tangential vectors $\bs t_1,\bs t_2$ and a normal vector $\bs n$ as the local frame. We first add DoFs on $\div \bs v\in \mathbb V^{L^2}_{k-1}(\bs r_3)$ to the original DoFs \eqref{eq:divdof0}-\eqref{eq:divdof3} and then remove redundant DoFs. 
For example, on a face $f$ 
\begin{equation*}
\div\boldsymbol{v}=\partial_{t_1}(\boldsymbol{v}\cdot\boldsymbol{t}_1)+\partial_{t_2}(\boldsymbol{v}\cdot\boldsymbol{t}_2)+\partial_{n}(\boldsymbol{v}\cdot\boldsymbol{n}) = \div_f \Pi_f \bs v + \partial_{n}(\boldsymbol{v}\cdot\boldsymbol{n}),
\end{equation*}
where $\Pi_f$ is the projection to the plane containing face $f$. 
If $\Pi_f \bs v$ is known, then $\partial_{n}(\boldsymbol{v}\cdot\boldsymbol{n}) = \div \bs v - \div_f \Pi_f \bs v$ can be determined. For normal derivatives, exchange the ordering of derivative, i.e. write $\partial^j_n(\div \bs v) = \div \partial^j_n \bs v$ and apply the above argument to remove $\partial_n^{j}(\bs v\cdot \bs n)$ for $j=1,\ldots, r_2^f$. Notice that DoF on $\bs v\cdot \bs n$ is still needed as $\div \bs v$ only gives constraint on derivatives. 

The situation on edges is more complicated. We write the normal derivative as $D^{\alpha}_n$ with $\alpha = (\alpha_1, \alpha_2)\in \mathbb T_j^1(e)$ for $j=0,1,\ldots r_2^e$. As $\div\boldsymbol{v}=\partial_t(\boldsymbol{v}\cdot\boldsymbol{t})+\partial_{n_1}(\boldsymbol{v}\cdot\boldsymbol{n}_1)+\partial_{n_2}(\boldsymbol{v}\cdot\boldsymbol{n}_2)$, for each $\beta \in \mathbb T_{j}^1(e), j=0,\ldots, r_3^f$, we can write
\begin{equation}\label{eq:divedge}
D^{\beta}_n(\div \bs v) = \partial_t D^{\beta}_n(\bs v\cdot \bs t) +\partial_{n_1}D^{\beta}_n(\bs v\cdot \bs n_1) +\partial_{n_2}D^{\beta}_n(\bs v\cdot \bs n_2).
\end{equation}
On edge $e$, DoFs of $D^{\beta}_n(\bs v\cdot \bs t)$ and $\nabla^j\bs v(\texttt{v})$ will determine the tangential component $D^{\beta}_n(\bs v\cdot \bs t)|_e\in \mathbb P_{k-j}(e)$ and consequently $\partial_t D^{\beta}_n(\bs v\cdot \bs t)|_e$. The normal derivative of $\bs v\cdot \bs n_1$ can be written as
$$
\partial_{n_1}D^{\beta}_n(\bs v\cdot \bs n_1) = D_n^{\alpha} (\bs v\cdot \bs n_1), \quad \alpha = \beta + \epsilon_1, 1\leq |\alpha | \leq r_3^e + 1\geq r_2^e.
$$
Providing DoFs on $D_n^{\alpha} (\bs v\cdot \bs n_1)$ for all $0 \leq |\alpha| \leq r_2^e$, we can then determine the third component in \eqref{eq:divedge}
$$
\partial_{n_2}D^{\beta}_n(\bs v\cdot \bs n_2)= D_n^{\alpha} (\bs v\cdot \bs n_2) , \quad \alpha = \beta + \epsilon_2, \alpha = (\alpha_1, \alpha_2), \alpha_2\geq 1, 1\leq |\alpha | = j \leq r_2^e.
$$
But the lattice node $(\alpha_1, 0)$ is missing, i.e., DoFs on $\partial_{n_1}^j(\boldsymbol{v}\cdot \boldsymbol{n}_2)$ for $j=0,1,\ldots, r_2^e$ should be still included. 

We are in the position to present finite element description of $\mathbb V^{\div}_{k}(\bs r_2, \bs r_3)$. Take $\mathbb P_{k}(T;\mathbb R^3)$ as the space of shape functions.
The degrees of freedom are 
\begin{align}
\nabla^i\boldsymbol{v}(\texttt{v}), & \quad i=0,\ldots, r_2^{\texttt{v}}, \label{eq:3dCrdivfemdofV1}\\
\nabla^j\div\boldsymbol{v}(\texttt{v}),  & \quad j=\max\{r_2^{\texttt{v}},0\},\ldots, r_3^{\texttt{v}}, \label{eq:3dCrdivfemdofV2}\\
%
\int_e \partial_{n_1}^j(\boldsymbol{v}\cdot \boldsymbol{n}_2)\,q \dd s, &\quad q\in \mathbb P_{k - 2(r_2^{\texttt{v}} +1)+j}(e), 0\leq j\leq r_2^e, \label{eq:3dCrdivfemdofE1}\\
\int_e \frac{\partial^{j}(\boldsymbol{v}\cdot\boldsymbol{t})}{\partial n_1^{i}\partial n_2^{j-i}}\, q \dd s, &\quad q\in \mathbb P_{k - 2(r_2^{\texttt{v}} +1)+j}(e), 0\leq i\leq j\leq r_2^e, \label{eq:3dCrdivfemdofE2}\\
\int_e \frac{\partial^{j}(\boldsymbol{v}\cdot\boldsymbol{n}_{1})}{\partial n_1^{i}\partial n_2^{j-i}}\, q \dd s, &\quad q\in \mathbb P_{k - 2(r_2^{\texttt{v}} +1)+j}(e), 0\leq i\leq j\leq r_2^e, \label{eq:3dCrdivfemdofE3}\\
\int_e \frac{\partial^{j}(\div\boldsymbol{v})}{\partial n_1^{i}\partial n_2^{j-i}} \, q \dd s, &\quad q\in \mathbb P_{k-1 - 2(r_3^{\texttt{v}} +1)+j}(e), 0\leq i\leq j\leq r_3^{e}, \label{eq:3dCrdivfemdofE4}\\
\int_f \boldsymbol{v}\cdot\boldsymbol{n}\, q\dd S, &\quad  q\in\mathbb P_0(f)\oplus\mathbb B_{k} (f;\begin{pmatrix}
r_2^{\texttt{v}} \\
r_2^e
\end{pmatrix})/\mathbb R,\label{eq:3dCrdivfemdofF1}\\
\int_f \partial_n^j(\boldsymbol{v}\cdot\boldsymbol{t}_{\ell})\ q \dd S, &\quad  q\in \mathbb B_{k - j} (f;\begin{pmatrix}
r_2^{\texttt{v}} \\
r_2^e
\end{pmatrix}-j),  0\leq j\leq r_2^{f}, \ell =1,2, \label{eq:3dCrdivfemdofF2}\\
\int_f \partial_n^j(\div\boldsymbol{v})\ q \dd S, &\quad  q\in \mathbb B_{k-1 - j} (f;\begin{pmatrix}
r_3^{\texttt{v}} \\
r_3^e
\end{pmatrix}-j),  0\leq j\leq r_3^{f}, \label{eq:3dCrdivfemdofF3}\\
\int_T \div\boldsymbol{v}\, q\dx, &\quad q\in \mathbb B_{k-1}(\boldsymbol{r}_3)/\mathbb R, \label{eq:3dCrdivfemdofT1}\\
\int_T \boldsymbol{v}\cdot\boldsymbol{q} \dx, &\quad \boldsymbol{q}\in \mathbb B^{\div}_{k}(\boldsymbol{r}_2)\cap \ker(\div) \label{eq:3dCrdivfemdofT2}
\end{align}
for each $\texttt{v}\in \Delta_{0}(T)$, $e\in \Delta_{1}(T)$ and $f\in \Delta_{2}(T)$.

\begin{lemma}\label{lem:3dCrdivfemunisolvence}
The DoFs~\eqref{eq:3dCrdivfemdofV1}-\eqref{eq:3dCrdivfemdofT2} are uni-solvent for $\mathbb P_{k}(T;\mathbb R^3)$.  
\end{lemma}
\begin{proof}
The number of DoFs~\eqref{eq:3dCrdivfemdofV2}, \eqref{eq:3dCrdivfemdofE4} and~\eqref{eq:3dCrdivfemdofF3}-\eqref{eq:3dCrdivfemdofT1} to determine $\div \bs v\in \mathbb V^{L^2}_{k-1}(\bs r_3)$ is
$
\dim \mathbb P_{k-1}(T)-1-4{r_2^{\texttt{v}}+2\choose3},
$
which is constant with respect to $\boldsymbol{r}_3$. Hence the number of DoFs~\eqref{eq:3dCrdivfemdofV1}-\eqref{eq:3dCrdivfemdofT2} is also constant with respect to $\boldsymbol{r}_3$. 
To count the dimension, we only need to consider case $\boldsymbol{r}_3= \bs r_2\ominus 1$. 
Now the number of DoFs \eqref{eq:3dCrdivfemdofE1}-\eqref{eq:3dCrdivfemdofE4} equals to that of 
$$
\int_e \frac{\partial^{j}\boldsymbol{v}}{\partial n_1^{i}\partial n_2^{j-i}} \cdot\boldsymbol{q} \dd s,  \quad e\in \Delta_1(T), \boldsymbol{q} \in \mathbb P_{k - 2(r_2^{\texttt{v}}+1) + j}^3(e), 0\leq i\leq j\leq r_2^{e}.
$$
As a result the number of DoFs~\eqref{eq:3dCrdivfemdofV1}-\eqref{eq:3dCrdivfemdofT2} equals to $\dim\mathbb P_{k}(T;\mathbb R^3)$.

Take $\boldsymbol{v}\in\mathbb P_{k}(T;\mathbb R^3)$ and assume all the DoFs~\eqref{eq:3dCrdivfemdofV1}-\eqref{eq:3dCrdivfemdofT2} vanish. 
The vanishing DoF~\eqref{eq:3dCrdivfemdofF1} implies $\div\boldsymbol{v}\in L_0^2(T)$.
By the vanishing DoFs~\eqref{eq:3dCrdivfemdofV1}-\eqref{eq:3dCrdivfemdofV2}, \eqref{eq:3dCrdivfemdofE3} and~\eqref{eq:3dCrdivfemdofF3}-\eqref{eq:3dCrdivfemdofT1}, we get $\div\boldsymbol{v}=0$.
Since $\div\boldsymbol{v}=\partial_t(\boldsymbol{v}\cdot\boldsymbol{t})+\partial_{n_1}(\boldsymbol{v}\cdot\boldsymbol{n}_1)+\partial_{n_2}(\boldsymbol{v}\cdot\boldsymbol{n}_2)$ for each edge $e$ and $\div\boldsymbol{v}=\partial_n(\boldsymbol{v}\cdot\boldsymbol{n})+\div_f(\Pi_f\boldsymbol{v})$ for each face $f$, it follows from the vanishing DoFs~\eqref{eq:3dCrdivfemdofV1}, \eqref{eq:3dCrdivfemdofE1}-\eqref{eq:3dCrdivfemdofE3}, and~\eqref{eq:3dCrdivfemdofF1}-\eqref{eq:3dCrdivfemdofF2} that $\boldsymbol{v}\in\mathbb B^{\div}_{k}(\boldsymbol{r}_2)\cap \ker(\div)$. Therefore $\boldsymbol{v}=\boldsymbol{0}$ holds from the vanishing DoF~\eqref{eq:3dCrdivfemdofT2}.  
\end{proof}

Define global $H(\div)$-conforming finite element space
\begin{align*}
\mathbb V^{\div}_{k}(\boldsymbol{r}_2, \boldsymbol{r}_3) = \{\boldsymbol{v}\in \boldsymbol{L}^2(\Omega;\mathbb R^3):&\, \boldsymbol{v}|_T\in\mathbb P_{k}(T;\mathbb R^3)\;\forall~T\in\mathcal T_h, \\
&\textrm{ all the DoFs~\eqref{eq:3dCrdivfemdofV1}-\eqref{eq:3dCrdivfemdofF3} are single-valued} \}.
\end{align*}

When $\boldsymbol{r}_3 = \bs r_2\ominus 1$, we have $\mathbb V^{\div}_{k}(\boldsymbol{r}_2,\bs r_2\ominus 1)=\mathbb V^{\div}_{k}(\boldsymbol{r}_2)$. Although the DoFs defining these two finite element spaces are in different forms, from the proof of Lemma \ref{lem:3dCrdivfemunisolvence}, they can express each other by linear combinations. 

When $\boldsymbol{r}_3\geq \bs r_2\ominus 1$, we have 
\begin{equation*}
\mathbb V^{\div}_{k}(\boldsymbol{r}_2,\boldsymbol{r}_3) \subseteq \mathbb V^{\div}_{k}(\boldsymbol{r}_2,\bs r_2\ominus 1)=\mathbb V^{\div}_{k}(\boldsymbol{r}_2).
\end{equation*}
Namely additional smoothness on $\div \bs v$ is imposed in space $\mathbb V^{\div}_{k}(\boldsymbol{r}_2,\boldsymbol{r}_3)$. 

\begin{theorem}
Let $\bs r_2\geq -1$ satisfy \eqref{eq:boundr2fordivbubble} and $\boldsymbol{r}_3\geq \bs r_2\ominus 1$ be a valid smoothness parameter, i.e. $r_3^{\texttt{v}} \geq 2 r_3^e, r_3^e\geq 2r_3^f$. 
Let $k$ be a sufficiently large integer so that $k\geq \max\{2r_{2}^{\texttt{v}}+1,  2r_{3}^{\texttt{v}}+2,1\}$, $\dim\mathbb B_{k-1}(T; \boldsymbol{r}_3)\geq1$, and $\dim\mathbb B_{k}(f;\begin{pmatrix}
r_2^{\texttt{v}} \\
r_2^e
\end{pmatrix})\geq1$.
It holds that
\begin{equation}\label{eq:divonto3d}
\div\mathbb V^{\div}_{k}(\boldsymbol{r}_2,\boldsymbol{r}_3)=\mathbb V^{L^2}_{k-1}(\boldsymbol{r}_3).   
\end{equation}
\end{theorem}
\begin{proof}
It is apparent that $\div\mathbb V^{\div}_{k}(\boldsymbol{r}_2,\boldsymbol{r}_3)\subseteq\mathbb V^{L^2}_{k-1}(\boldsymbol{r}_3)$. We are going to prove the div operator is surjective. 


For $p_h\in \mathbb V^{L^2}_{k-1}(\boldsymbol{r}_3)\subset H^{r_{3}^f+1}(\Omega)$, there exists a $\boldsymbol{v}\in \boldsymbol{H}^{r_{3}^f+2}(\Omega;\mathbb R^3)$ such that $\div\boldsymbol{v}=p_h$.
Take $\boldsymbol{v}_h\in\mathbb V^{\div}_{k}(\boldsymbol{r}_2,\boldsymbol{r}_3)$ such that all DoFs \eqref{eq:3dCrdivfemdofV1}-\eqref{eq:3dCrdivfemdofT2} vanish except
\begin{align*}
\nabla^j(\partial_1v_{h,1})(\texttt{v})&=\nabla^jp_h(\texttt{v}), \qquad \qquad \; j=0,\ldots, r_{2}^{\texttt{v}}-1, \\
\nabla^j\div\boldsymbol{v}_h(\texttt{v})&=\nabla^jp_h(\texttt{v}),   \qquad\quad\quad\;\; j=\max\{r_2^{\texttt{v}},0\},\ldots, r_3^{\texttt{v}}, \\
\int_e \frac{\partial^{j}(\div\boldsymbol{v}_h)}{\partial n_1^{i}\partial n_2^{j-i}} \, q \dd s&=\int_e \frac{\partial^{j}p_h}{\partial n_1^{i}\partial n_2^{j-i}} \, q \dd s, \quad q\in \mathbb P_{k-1 - 2(r_3^{\texttt{v}} +1)+j}(e), 0\leq i\leq j\leq r_3^{e}, \\
\int_f \boldsymbol{v}_h\cdot\boldsymbol{n}\dd S&=\int_f \boldsymbol{v}\cdot\boldsymbol{n}\dd S,  \\
\int_f \partial_n^j(\div\boldsymbol{v}_h)\ q \dd S&=\int_f \partial_n^jp_h\ q \dd S, \qquad\quad\, q\in \mathbb B_{k-1 - j} (f;\begin{pmatrix}
r_3^{\texttt{v}} \\
r_3^e
\end{pmatrix}-j),  0\leq j\leq r_3^{f}, \\
\int_T \div\boldsymbol{v}_h\, q\dx&=\int_T p_h\, q\dx, \qquad\quad\quad\; q\in \mathbb B_{k-1}(\boldsymbol{r}_3)/\mathbb R,
\end{align*}
for all $\texttt{v}\in \Delta_{0}(T_h)$, $e\in \Delta_{1}(T_h)$, $f\in \Delta_{2}(T_h)$ and $T\in T_h$.
Then it holds $\div \bs v_h = p_h$. 
\end{proof}

We shall call such $(\bs r_2, \bs r_3)$ a div stable pair. Namely $\bs r_2\geq -1$ satisfies \eqref{eq:boundr2fordivbubble} and $\boldsymbol{r}_3\geq \bs r_2\ominus 1$ is a valid smoothness parameter.

\begin{example}\rm
Taking $k\geq 4$, $\boldsymbol{r}_2=-1$, and $\boldsymbol{r}_3=0,$ we get a stable pair for mixed Poisson problem but with continuous displacement. That is we can construct a subspace of BDM space with the range of $\div$ is continuous. The degree $k\geq 4$ is to ensure $\dim \mathbb B_{k-1}(T,\bs r_3)\geq 1$. 
\end{example}

\begin{example}\rm
Taking $k\geq 6$, $\boldsymbol{r}_2=(2,1,0)^{\intercal}$, and $\boldsymbol{r}_3=(1,0,0)^{\intercal},$ we get a stable Stokes-pair with continuous pressure element
$$
(\mathbb V^{\div}_{k}(
\begin{pmatrix}
 2\\
 1\\
 0
\end{pmatrix},
\begin{pmatrix}
 1\\
 0\\
 0
\end{pmatrix}
), \mathbb V^{L^2}_{k-1}(
\begin{pmatrix}
 1\\
 0\\
 0
\end{pmatrix}
)),
$$
which is a generalization of the two-dimensional Falk-Neilan Stokes element constructed in~\cite{FalkNeilan2013} to three dimensions. 
\end{example}

\section{Finite Element de Rham and Stokes Complexes}\label{sec:femderhamcomplex}
In this section we shall construct several finite element de Rham and Stokes complexes. 
\subsection{Exactness of a complex of finite dimensional spaces}
\begin{lemma}\label{lm:abstract}
Let $\mathcal P$ and $\mathcal V_i$ be finite-dimensional linear spaces for $i=0,\ldots,3$ and 
\begin{equation}\label{eq:generalexactsequence}
\mathcal P \xrightarrow{\subset}\mathcal V_0 \xrightarrow{\dd_0} \mathcal V_1 \xrightarrow{\dd_1} \mathcal V_2 \xrightarrow{\dd_2} \mathcal V_{3}\rightarrow 0,
\end{equation}
be a complex. Assume three out of the four conditions for the exactness of the complex hold
\begin{align*}
\mathcal P &=\mathcal V_0\cap\ker(\dd_0) \\
\dd_0\mathcal V_0 &= \mathcal V_1\cap \ker(\dd_1)\\
\dd_1\mathcal V_1 &= \mathcal V_2\cap \ker(\dd_2)\\
\dd_2\mathcal V_2 &=\mathcal V_3,
\end{align*}
and the dimensions satisfy
\begin{equation}\label{eq:abstractdimenidentity}
\dim \mathcal P - \dim \mathcal V_0 + \dim \mathcal V_1 - \dim \mathcal V_2 + \dim \mathcal V_3 = 0,
\end{equation}
then complex~\eqref{eq:generalexactsequence} is exact. 
\end{lemma}
\begin{proof}
By \eqref{eq:abstractdimenidentity}, we have
\begin{align*}
&\quad-(\dim\mathcal V_0\cap\ker(\dd_0)-\dim\mathcal P) + (\dim\mathcal V_1\cap \ker(\dd_1)-\dim\dd_0\mathcal V_0) \\
&\quad-(\dim\mathcal V_2\cap \ker(\dd_2)-\dim\dd_1\mathcal V_1) + (\dim\mathcal V_3-\dim\dd_2\mathcal V_2) \\
&=\dim \mathcal P - \dim \mathcal V_0 + \dim \mathcal V_1 - \dim \mathcal V_2 + \dim \mathcal V_3 = 0,
\end{align*}
as required.
\end{proof}

A polynomial de Rham complex on tetrahedron $T$ is, for $k\geq 1$,
\begin{equation}\label{eq:polyderham}
\mathbb R\xrightarrow{\subset} \mathbb P_{k+2}(T)\xrightarrow{\grad}\mathbb P_{k+1}(T;\mathbb R^3)\xrightarrow{\curl}\mathbb P_{k}(T;\mathbb R^3) \xrightarrow{\div} \mathbb P_{k-1}(T)\xrightarrow{}0,
\end{equation}
where 
$\mathbb P_{k}(T;\mathbb X):=\mathbb P_{k}(T)\otimes\mathbb X$ for $\mathbb X$ being vector space $\mathbb R^3$, tensor space $\mathbb M$, or symmetric tensor space $\mathbb S$.

By Lemma~\ref{lm:abstract}, the exactness of polynomial complex \eqref{eq:polyderham} can be verified by the identity
\begin{equation}\label{eq:dimensionpolyderham}
1 - {k+5 \choose 3} + 3{k+4 \choose 3} - 3{k+3 \choose 3} + {k+2 \choose 3} = 0,
\end{equation}
and the fact 
\begin{align*}
\mathbb R &= \ker(\grad), \\
\mathbb P_{k+1}(T;\mathbb R^3)\cap \ker(\curl) &= \grad \mathbb P_{k+2}(T),\\
\div \mathbb P_{k}(T;\mathbb R^3) &= \mathbb P_{k-1}(T).
\end{align*}
The first two are trivial. On the last one, use the fact $\div(\bs x \mathbb P_{k-1}(T)) = \mathbb P_{k-1}(T)$.

\subsection{Graft de Rham complexes}
We first present an abstract definition of space 
$$
\mathbb V_k^{\curl}(\bs r_1) = \mathbb V_k^{3}(\bs r_1)\cap H(\curl,\Omega). 
$$
We call $(\bs r_0, \bs r_1, \bs r_2, \bs r_3)$ a valid de Rham smoothness sequence if the following sequence, with a sufficiently large degree $k$,  
\begin{equation}\label{eq:femderhamcomplex3d1}
\mathbb R\xrightarrow{\subset}\mathbb V^{\grad}_{k+2}(\boldsymbol{r}_0)\xrightarrow{\grad}\mathbb V^{\curl}_{k+1}(\boldsymbol{r}_1)\xrightarrow{\curl}\mathbb V^{\div}_{k}(\boldsymbol{r}_2)\xrightarrow{\div}\mathbb V^{L^2}_{k-1}(\boldsymbol{r}_3)\to0
\end{equation}
is an exact Hilbert complex. For $\hat{\bs r}_2\geq \bs r_2$ and $\hat{\bs r}_2$ is a valid smoothness vector, we can define the subspace 
$$\mathbb V_{k+1}^{\curl}(\bs r_1, \hat{\bs r}_2 ) = \{\bs v\in \mathbb V_{k+1}^{\curl}(\bs r_1) \mid \curl \bs v\in \mathbb V_k^{\div}(\hat{\bs r}_2) \cap \ker(\div)\}.$$ Such space is well defined as $\mathbb V_k^{\div}(\hat{\bs r}_2) \cap \ker(\div)\subseteq \mathbb V_k^{\div}({\bs r}_2) \cap \ker(\div)$ and $(\bs r_0, \bs r_1, \bs r_2, \bs r_3)$ is a valid de Rham smoothness sequence implies $\curl \mathbb V^{\curl}_{k+1}(\boldsymbol{r}_1) =  \mathbb V_k^{\div}({\bs r}_2) \cap \ker(\div)$.

Recall that $(\bs r_2, \bs r_3)$ is called a div stable pair if $\div \mathbb V^{\div}_k(\bs r_2, \bs r_3) = \mathbb V_{k-1}^{L^2}(\bs r_3)$ for a sufficiently large degree $k$. 
\begin{theorem}
Assume $(\bs r_0, \bs r_1, \bs r_2, \bs r_3)$ is a valid de Rham smoothness sequence and $(\hat{\bs r}_2, \hat{\bs r}_3)$ is a div stable smoothness pair with $\hat{\bs r}_2\geq \bs r_2$, $\hat{\bs r}_3\geq \bs r_3$. Then
$(\bs r_0, \bs r_1, \hat{\bs r}_2, \hat{\bs r}_3)$ is also a valid de Rham smoothness sequence in the sense that the following complex
\begin{equation*}
\mathbb R\xrightarrow{\subset}\mathbb V^{\grad}_{k+2}(\boldsymbol{r}_0)\xrightarrow{\grad}\mathbb V^{\curl}_{k+1}(\boldsymbol{r}_1, \hat{\boldsymbol{r}}_2)\xrightarrow{\curl}\mathbb V^{\div}_{k}(\hat{\boldsymbol{r}}_2, \hat{\boldsymbol{r}}_3)\xrightarrow{\div}\mathbb V^{L^2}_{k-1}(\hat{\boldsymbol{r}}_3)\to0
\end{equation*}
is exact.
\end{theorem}
\begin{proof}
 Exactness of \eqref{eq:femderhamcomplex3d1} implies $\ker(\div)\cap \mathbb V_k^{\div}(\bs r_2) =\curl \mathbb V_{k+1}^{\curl}(\bs r_1)$. Then $\ker(\div)\cap \mathbb V_{k}^{\div}(\hat{\bs r}_2)\subseteq \ker(\div)\cap \mathbb V_{k}^{\div}(\bs r_2)=\curl \mathbb V_{k+1}^{\curl}(\bs r_1)$, by which we get $$\ker(\div)\cap \mathbb V_{k}^{\div}(\hat{\bs r}_2) =\curl \mathbb V_{k+1}^{\curl}(\bs r_1, \hat{\bs r}_2).$$
As in $\mathbb V^{\curl}_{k+1}(\boldsymbol{r}_1, \hat{\boldsymbol{r}}_2)$, only the range of $\curl$ operator is changed, the relation $\mathbb V^{\curl}_{k+1}(\boldsymbol{r}_1, \hat{\boldsymbol{r}}_2) \cap \ker(\curl) = \grad \mathbb V^{\grad}_{k+2}(\boldsymbol{r}_0)$ still holds. The relation $\div \mathbb V^{\div}_{k}(\hat{\boldsymbol{r}}_2, \hat{\boldsymbol{r}}_3) = \mathbb V^{L^2}_{k-1}(\hat{\boldsymbol{r}}_3)$ is from the assumption  $(\hat{\bs r}_2, \hat{\bs r}_3)$ is div stable.
\end{proof}

In the following examples, we shall further simplify the notation by presenting the smoothness parameters only and skip the space notation which should be clear from the context. 

\begin{example}\rm 
The standard de Rham complex is $(\bs 0, \bs {-1}, \bs{-1}, \bs{-1})$. Take the stable div pair $\hat{\boldsymbol{r}}_2=(2,1,0)^{\intercal}, \hat{\boldsymbol{r}}_3=(1,0,-1)^{\intercal}$, we obtain the finite element Stokes complex 
 $$
\begin{pmatrix}
 0\\
 0\\
 0
\end{pmatrix}
\xrightarrow{\grad}
\begin{pmatrix}
 -1\\
 -1\\
 -1
\end{pmatrix}
\xrightarrow{\curl}
\begin{pmatrix}
 2\\
 1\\
 0
\end{pmatrix}
\xrightarrow{\div} 
\begin{pmatrix}
 1\\
 0\\
 -1
\end{pmatrix}.
 $$
Space $\mathbb V_{k+1}^{\curl}(\begin{pmatrix}
 -1\\
 -1\\
 -1
\end{pmatrix},
\begin{pmatrix}
 2\\
 1\\
 0
\end{pmatrix})$, for $k+1\geq 7$, is similar but slightly different from the one constructed in~\cite{ZhangZhang2020}, whose shape function space coincides with the first kind N\'ed\'elec element.



Use $\hat{\boldsymbol{r}}_2=(2,1,0)^{\intercal}, \hat{\boldsymbol{r}}_3=(1,0,0)^{\intercal}$, we get another finite element Stokes complex ending with a continuous element 
 $$
\begin{pmatrix}
 0\\
 0\\
 0
\end{pmatrix}
\xrightarrow{\grad}
\begin{pmatrix}
 -1\\
 -1\\
 -1
\end{pmatrix}
\xrightarrow{\curl}
\begin{pmatrix}
 2\\
 1\\
 0
\end{pmatrix}
\xrightarrow{\div} 
\begin{pmatrix}
 1\\
 0\\
 0
\end{pmatrix}.
 $$
As $\mathbb V_{k}^{\div}(\hat{\bs r}_2)\subset \bs H^1(\Omega;\mathbb R^3)$, the space $\mathbb V_{k+1}^{\curl}(\bs r_1, \hat{\bs r}_2 )\subset \bs H(\grad\curl,\Omega):=\{\bs v\in \bs H(\curl,\Omega), \curl \bs v\in \bs H^1(\Omega;\mathbb R^3)\}$ can be used to discretize the quad-curl problem~\cite{ZhangZhang2020,ZhengHuXu2011}.  
\end{example}

\subsection{$H(\curl)$-conforming finite elements}
Next we give a finite element description for $\mathbb V_{k+1}^{\curl}(\bs r_1, \bs r_2)$ with $\bs r_2\geq \bs r_1\ominus 1$. We should keep DoFs for $\curl \bs v\in \mathbb V^{\div}_k(\bs r_2)$, combine DoFs for $\mathbb V^3_{k+1}(\bs r_1)$, and eliminate linearly dependent ones. 

On edge $e$, we choose frame $\{\boldsymbol{t},\boldsymbol{n}_1,\boldsymbol{n}_2\}$ and to facilitate the calculation simplify as $(x_0, x_1, x_2)$. Then $\bs u = (u_0, u_1, u_2)^{\intercal}$ with $u_0 = \bs u\cdot \bs t, u_1 = \bs u\cdot \bs n_1, u_2 = \bs u\cdot \bs n_2$, and 
$\curl \bs u = (\partial_1 u_2 - \partial_2 u_1, \partial_2 u_0 - \partial_0 u_2, \partial_0 u_1 - \partial_1 u_0)^{\intercal}$. Apply the normal derivative $D^{\alpha}_n$ to $\curl\boldsymbol{u}$, we obtain, for $\alpha\in \mathbb T^1_j(e), j = 0,\ldots, r_2^e$, 
\begin{align}
D^{\alpha}_n \curl \bs u = & \big ( D_n^{(\alpha_1+1, \alpha_2)} u_2 - D_n^{(\alpha_1, \alpha_2+1)} u_1, \label{eq:Dcurl0}\\
&\, D_n^{(\alpha_1, \alpha_2+1)} u_0 - \partial_0D_n^{(\alpha_1, \alpha_2)} u_2, \label{eq:Dcurl1}\\
&\, \partial_0D_n^{(\alpha_1, \alpha_2)} u_1 - D_n^{(\alpha_1+1, \alpha_2)} u_0 \big )^{\intercal}. \label{eq:Dcurl2} 
\end{align}
DoFs on $D^{\alpha}_n \curl \bs u$ are given and thus $D^{\alpha}_n \curl \bs u$ is considered as known on edge $e$. 
Include DoFs on $D_n^{\alpha}u_1, 0\leq |\alpha|\leq r_1^e$, then $\partial_0 D_n^{\alpha}u_1$ is  known on edge $e$. Linear combination with \eqref{eq:Dcurl0}, we can determine  $D_n^{\alpha}u_2$, $1\leq |\alpha|\leq r_1^e, \alpha_1\geq 1$ but $\partial_{n_2}^j u_2, j = 0,\ldots, r_1^e$ are left and thus should be included in the DoF. Linear combination with \eqref{eq:Dcurl2}, we can determine $D^{\alpha}_n u_0$ for $1\leq |\alpha|\leq r_1^e$ with $\alpha_1\geq 1$. Linear combination with \eqref{eq:Dcurl1}, we also know $D^{\alpha}_n u_0$ for $1\leq |\alpha|\leq r_1^e$ with $\alpha_2\geq 1$. So only $\alpha = (0,0)$ is left. Namely DoF on $u_0= \bs u\cdot \bs t$ should be included explicitly. 

We then move to faces and present formulae on the normal and tangential component of $\curl \bs u$. 
For smooth scalar function $v$ and face $f$ with unit normal vector $\boldsymbol{n}$, define surface gradient
$$
\nabla_f v:=\Pi_f(\nabla v)= \nabla v - (\partial_n v)\boldsymbol{n}.
$$
For smooth vector function $\boldsymbol{u}$, define surface rotation
$$
\mathrm{rot}_f \boldsymbol{u}:=(\boldsymbol{n}\times\nabla)\cdot\boldsymbol{u}=(\curl\boldsymbol{u})\cdot\boldsymbol{n}.
$$
Clearly it holds $\mathrm{rot}_f \boldsymbol{u}=\mathrm{rot}_f(\Pi_f\boldsymbol{u})$.

\begin{lemma}
On face $f$, for smooth enough function $\bs u$, it holds that
\begin{align}
\label{eq:curlun} \bs n\cdot (\nabla \times \bs u) & = \mathrm{rot}_f(\Pi_f \bs u),\\
\label{eq:curluT} \boldsymbol{n}\times(\nabla\times\boldsymbol{u})&=\nabla_f(\boldsymbol{u}\cdot\boldsymbol{n})-\partial_n(\Pi_f\boldsymbol{u}),\\
\label{eq:DcurluT} \partial_n^j (\boldsymbol{n}\times(\nabla\times\boldsymbol{u}))&=\nabla_f(\partial_n^j\boldsymbol{u}\cdot\boldsymbol{n})- \Pi_f\partial_n^{j+1}\boldsymbol{u}, \quad j\geq 0.
\end{align}
\end{lemma}
\begin{proof}
Identity \eqref{eq:curlun} is indeed the definition of $\mathrm{rot}_f$. 
By a direct computation, \eqref{eq:curluT} follows from
\begin{align*}	
\boldsymbol{n}\times(\nabla\times\boldsymbol{u})&=\nabla(\boldsymbol{u}\cdot\boldsymbol{n})-\partial_n\boldsymbol{u}\\
&=\nabla_f(\boldsymbol{u}\cdot\boldsymbol{n})+\partial_n(\boldsymbol{u}\cdot\boldsymbol{n})\boldsymbol{n}-\big(\partial_n(\Pi_f\boldsymbol{u})+\partial_n(\boldsymbol{u}\cdot\boldsymbol{n})\boldsymbol{n}\big) \\
&=\nabla_f(\boldsymbol{u}\cdot\boldsymbol{n})-\partial_n(\Pi_f\boldsymbol{u}).
\end{align*}
Exchange partial derivatives $\partial_n^j$ with $n\times (\nabla \times \cdot)$ to get \eqref{eq:DcurluT}.
\end{proof}
We always include DoFs for $\curl \bs u$. So by \eqref{eq:curlun} $\mathrm{rot}_f \Pi_f \bs u$ can be determined. By the Helmholtz decomposition of a vector function on the face, the tangential component  $\Pi_f \bs u$ can be determined by $\mathrm{rot}_f \Pi_f \bs u$ and the moment with $\grad_f\mathbb B_{k+2} (f;\begin{pmatrix}
r_1^{\texttt{v}} \\
r_1^e
\end{pmatrix}+1)$. The normal derivative of the normal component $\partial_n^j(\bs u\cdot \bs n)$, for $j=0,1,\ldots, r_1^f$, will be included as DoFs. Then $ \nabla_f(\partial_n^j \boldsymbol{u}\cdot\boldsymbol{n})$ can be computed on face. Thanks to \eqref{eq:DcurluT}, the normal derivative of the tangential component $\partial_n^j \Pi_f \bs u$ can be determined. 

We are in the position to present a finite element description for the space $\mathbb V_{k+1}^{\curl}(\bs r_1, \bs r_2)$. 
Take $\mathbb P_{k+1}(T;\mathbb R^3)$ as the space of shape functions.
The degrees of freedom are
\begin{align}
\nabla^i\boldsymbol{v}(\texttt{v}), & \quad i=0,\ldots, r_1^{\texttt{v}}, \label{eq:3dCrcurlfemdofV1}\\
\nabla^j(\curl\boldsymbol{v})(\texttt{v}),  & \quad j=\max\{r_1^{\texttt{v}},0\},\ldots, r_2^{\texttt{v}}, \label{eq:3dCrcurlfemdofV2}\\
\int_e \boldsymbol{v}\cdot \boldsymbol{t}\,q \dd s, &\quad q\in \mathbb P_{k-1 - 2r_1^{\texttt{v}}}(e), \label{eq:3dCrcurlfemdofE1}\\
\int_e \frac{\partial^{j}(\boldsymbol{v}\cdot\boldsymbol{n}_{1})}{\partial n_1^{i}\partial n_2^{j-i}}\, q \dd s, &\quad q\in \mathbb P_{k -1- 2r_1^{\texttt{v}}+j}(e), 0\leq i\leq j\leq r_1^e, \label{eq:3dCrcurlfemdofE2}\\
\int_e \partial_{n_2}^{j}(\boldsymbol{v}\cdot\boldsymbol{n}_{2})\, q \dd s, &\quad q\in \mathbb P_{k -1- 2r_1^{\texttt{v}}+j}(e), 0\leq j\leq r_1^e, \label{eq:3dCrcurlfemdofE3}\\
\int_e \partial_{n_1}^j((\curl\boldsymbol{v})\cdot \boldsymbol{n}_2)\,q \dd s, &\quad q\in \mathbb P_{k - 2(r_2^{\texttt{v}} +1)+j}(e),  0\leq j\leq r_2^e, \label{eq:3dCrcurlfemdofE4}\\
\int_e \frac{\partial^{j}((\curl\boldsymbol{v})\cdot\boldsymbol{t})}{\partial n_1^{i}\partial n_2^{j-i}}\, q \dd s, &\quad q\in \mathbb P_{k - 2(r_2^{\texttt{v}} +1)+j}(e), 0\leq i\leq j\leq r_2^e, \label{eq:3dCrcurlfemdofE5}\\
\int_e \frac{\partial^{j}((\curl\boldsymbol{v})\cdot\boldsymbol{n}_1)}{\partial n_1^{i}\partial n_2^{j-i}}\, q \dd s, &\quad q\in \mathbb P_{k - 2(r_2^{\texttt{v}} +1)+j}(e), 0\leq i\leq j\leq r_2^e, \label{eq:3dCrcurlfemdofE6}\\
\int_f (\Pi_f\boldsymbol{v})\cdot\boldsymbol{q}\dd S, &\quad \boldsymbol{q}\in\grad_f\mathbb B_{k+2} (f;\begin{pmatrix}
r_1^{\texttt{v}} \\
r_1^e
\end{pmatrix}+1),\label{eq:3dCrcurlfemdofF1}\\
\int_f \partial_n^j(\boldsymbol{v}\cdot\boldsymbol{n})\, q\dd S, &\quad  q\in\mathbb B_{k+1-j} (f;\begin{pmatrix}
r_1^{\texttt{v}} \\
r_1^e
\end{pmatrix}-j), 0\leq j\leq r_1^f, \label{eq:3dCrcurlfemdofF2}\\
\int_f (\curl\boldsymbol{v})\cdot\boldsymbol{n}\,q\dd S, &\quad q\in\mathbb B_{k}(f;\begin{pmatrix}
r_2^{\texttt{v}} \\
r_2^e
\end{pmatrix})/\mathbb R,\label{eq:3dCrcurlfemdofF3}\\
\int_f \partial_n^j((\curl\boldsymbol{v})\cdot\boldsymbol{t}_{\ell})\ q \dd S, &\quad  q\in \mathbb B_{k - j} (f;\begin{pmatrix}
r_2^{\texttt{v}} \\
r_2^e
\end{pmatrix}-j),  0\leq j\leq r_2^{f}, \ell=1,2, \label{eq:3dCrcurlfemdofF4}\\
\int_T (\curl\boldsymbol{v})\cdot\boldsymbol{q} \dx, &\quad \boldsymbol{q}\in \mathbb B^{\div}_{k}(\boldsymbol{r}_2)\cap \ker(\div), \label{eq:3dCrcurlfemdofT1} \\
\int_T \boldsymbol{v}\cdot\boldsymbol{q} \dx, &\quad \boldsymbol{q}\in \grad\mathbb B_{k+2}(\boldsymbol{r}_1+1) \label{eq:3dCrcurlfemdofT2}
\end{align}
for each $\texttt{v}\in \Delta_{0}(T)$, $e\in \Delta_{1}(T)$ and $f\in \Delta_{2}(T)$.

\begin{lemma}\label{lem:3dCrcurlfemunisolvence}
The DoFs~\eqref{eq:3dCrcurlfemdofV1}-\eqref{eq:3dCrcurlfemdofT2} are uni-solvent for $\mathbb P_{k+1}(T;\mathbb R^3)$.  
\end{lemma}
\begin{proof}
Since $\nabla(\curl\boldsymbol{v})$ is trace-free, the number of DoF~\eqref{eq:3dCrcurlfemdofV2} at one vertex is 
$$
3{r_2^{\texttt{v}}+3\choose3}-3{r_1^{\texttt{v}}+2\choose3}-\left({r_2^{\texttt{v}}+2\choose3}-{r_1^{\texttt{v}}+1\choose3}\right).
$$
Then thanks to the proof of Lemma~\ref{lem:3dCrdivfemunisolvence},
the sum of the number of DoFs~\eqref{eq:3dCrcurlfemdofV2}, \eqref{eq:3dCrcurlfemdofE4}-\eqref{eq:3dCrcurlfemdofE6} and~\eqref{eq:3dCrcurlfemdofF3}-\eqref{eq:3dCrcurlfemdofT1} is 
\begin{align*}
&\quad3\dim \mathbb P_{k}(T)-\left(\dim \mathbb P_{k-1}(T)-1-4{r_2^{\texttt{v}}+2\choose3}\right)-4-12{r_2^{\texttt{v}}+3\choose3}\\
&\quad+ 12{r_2^{\texttt{v}}+3\choose3}-12{r_1^{\texttt{v}}+2\choose3}-4{r_2^{\texttt{v}}+2\choose3}+4{r_1^{\texttt{v}}+1\choose3} \\
&=3\dim \mathbb P_{k}(T)-\dim \mathbb P_{k-1}(T)-3-12{r_1^{\texttt{v}}+2\choose3}+4{r_1^{\texttt{v}}+1\choose3},
\end{align*}
which is constant with respect to $\boldsymbol{r}_2$. Hence the sum of the number of DoFs~\eqref{eq:3dCrcurlfemdofV1}-\eqref{eq:3dCrcurlfemdofT2} is also constant with respect to $\boldsymbol{r}_2$. 
It suffices to consider case $\boldsymbol{r}_2= \boldsymbol{r}_1\ominus1$ to count the dimension.
Now the number of DoFs \eqref{eq:3dCrcurlfemdofE3}-\eqref{eq:3dCrcurlfemdofE6} equals to that of 
\begin{align}
\int_e \frac{\partial^{j}(\boldsymbol{v}\cdot\boldsymbol{t})}{\partial n_1^{i}\partial n_2^{j-i}}\, q \dd s, &\quad q\in \mathbb P_{k -1- 2r_1^{\texttt{v}}+j}(e), 0\leq i\leq j, 1\leq j\leq r_1^e, \label{eq:curldoftemp1}\\
\int_e \frac{\partial^{j}(\boldsymbol{v}\cdot\boldsymbol{n}_{2})}{\partial n_1^{i}\partial n_2^{j-i}}\, q \dd s, &\quad q\in \mathbb P_{k -1- 2r_1^{\texttt{v}}+j}(e), 0\leq i\leq j\leq r_1^e. \label{eq:curldoftemp2}
\end{align}
As a result the number of DoFs~\eqref{eq:3dCrcurlfemdofV1}-\eqref{eq:3dCrcurlfemdofT2} equals to $\dim\mathbb P_{k+1}(T;\mathbb R^3)$.

Take $\boldsymbol{v}\in\mathbb P_{k+1}(T;\mathbb R^3)$ and assume all the DoFs~\eqref{eq:3dCrcurlfemdofV1}-\eqref{eq:3dCrcurlfemdofT2} vanish. 
The vanishing DoF~\eqref{eq:3dCrcurlfemdofE1} implies $(\curl\boldsymbol{v})\cdot\boldsymbol{n}|_f\in L_0^2(f)$ for $f\in\Delta_2(T)$.
By the vanishing DoFs~\eqref{eq:3dCrcurlfemdofV1}-\eqref{eq:3dCrcurlfemdofV2}, \eqref{eq:3dCrcurlfemdofE4}-\eqref{eq:3dCrcurlfemdofE6} and~\eqref{eq:3dCrcurlfemdofF3}-\eqref{eq:3dCrcurlfemdofT1}, we get $\curl\boldsymbol{v}=\boldsymbol{0}$.

For edge $e\in\Delta_1(T)$ with frame $\{\boldsymbol{t},\boldsymbol{n}_1,\boldsymbol{n}_2\}$, we have
\begin{align*}
\curl\boldsymbol{v}&=\curl((\boldsymbol{v}\cdot\boldsymbol{t})\boldsymbol{t}+(\boldsymbol{v}\cdot\boldsymbol{n}_1)\boldsymbol{n}_1+(\boldsymbol{v}\cdot\boldsymbol{n}_2)\boldsymbol{n}_2) \\
&= -(\boldsymbol{t}\times\nabla)(\boldsymbol{v}\cdot\boldsymbol{t})-(\boldsymbol{n}_1\times\nabla)(\boldsymbol{v}\cdot\boldsymbol{n}_1)-(\boldsymbol{n}_2\times\nabla)(\boldsymbol{v}\cdot\boldsymbol{n}_2),
\end{align*}
which combined with $\curl\boldsymbol{v}=\boldsymbol{0}$ implies
\begin{align}
(\curl\boldsymbol{v})\cdot\boldsymbol{t}&=\partial_{n_1}(\boldsymbol{v}\cdot\boldsymbol{n}_2)-\partial_{n_2}(\boldsymbol{v}\cdot\boldsymbol{n}_1)=0, \label{eq:edgecurldecompt} \\
(\curl\boldsymbol{v})\cdot\boldsymbol{n}_1&=\partial_{n_2}(\boldsymbol{v}\cdot\boldsymbol{t})-\partial_{t}(\boldsymbol{v}\cdot\boldsymbol{n}_2)=0,\label{eq:edgecurldecompn1}\\
(\curl\boldsymbol{v})\cdot\boldsymbol{n}_2&=\partial_{t}(\boldsymbol{v}\cdot\boldsymbol{n}_1)-\partial_{n_1}(\boldsymbol{v}\cdot\boldsymbol{t})=0. \label{eq:edgecurldecompn2}
\end{align}
Then it follows from the vanishing DoFs \eqref{eq:3dCrcurlfemdofE2}-\eqref{eq:3dCrcurlfemdofE3} that
\eqref{eq:curldoftemp1}-\eqref{eq:curldoftemp2} vanish.

Similarly for face $f\in\Delta_2(T)$ with frame $\{\boldsymbol{n},\boldsymbol{t}_1,\boldsymbol{t}_2\}$, we have
\begin{align*}
\curl\boldsymbol{v}&=\curl((\boldsymbol{v}\cdot\boldsymbol{n})\boldsymbol{n}+(\boldsymbol{v}\cdot\boldsymbol{t}_1)\boldsymbol{t}_1+(\boldsymbol{v}\cdot\boldsymbol{t}_2)\boldsymbol{t}_2) \\
&= -(\boldsymbol{n}\times\nabla)(\boldsymbol{v}\cdot\boldsymbol{n})-(\boldsymbol{t}_1\times\nabla)(\boldsymbol{v}\cdot\boldsymbol{t}_1)-(\boldsymbol{t}_2\times\nabla)(\boldsymbol{v}\cdot\boldsymbol{t}_2),
\end{align*}
which combined with $\curl\boldsymbol{v}=\boldsymbol{0}$ implies
\begin{align*}
(\curl\boldsymbol{v})\cdot\boldsymbol{n}&=\partial_{t_1}(\boldsymbol{v}\cdot\boldsymbol{t}_2)-\partial_{t_2}(\boldsymbol{v}\cdot\boldsymbol{t}_1)=0,  \\
(\curl\boldsymbol{v})\cdot\boldsymbol{t}_1&=\partial_{t_2}(\boldsymbol{v}\cdot\boldsymbol{n})-\partial_{n}(\boldsymbol{v}\cdot\boldsymbol{t}_2)=0,\\
(\curl\boldsymbol{v})\cdot\boldsymbol{t}_2&=\partial_{n}(\boldsymbol{v}\cdot\boldsymbol{t}_1)-\partial_{t_1}(\boldsymbol{v}\cdot\boldsymbol{n})=0.
\end{align*}
Then it follows from the vanishing DoFs \eqref{eq:3dCrcurlfemdofF1}-\eqref{eq:3dCrcurlfemdofF2} that
$$
(\partial_n^j\boldsymbol{v})|_f=\boldsymbol{0} \quad\textrm{ for }\; 0\leq j\leq r_1^f. 
$$

Finally $\boldsymbol{v}=\boldsymbol{0}$ holds from the vanishing DoF~\eqref{eq:3dCrcurlfemdofT2}.  
\end{proof}

Define the global $H(\curl)$-conforming finite element space
\begin{align*}
\mathbb V_{k+1}^{\curl}(\bs r_1, \bs r_2 )=\{\bs v\in\bs L^2(\Omega;\mathbb R^3):&\, \bs v|_{T}\in\mathbb P_{k+1}(T;\mathbb R^3) \textrm{ for each } T\in\mathcal T_h \\
& \textrm{ all the DoFs \eqref{eq:3dCrcurlfemdofV1}-\eqref{eq:3dCrcurlfemdofF4} are single-valued}\}.
\end{align*}
By the proof of Lemma~\ref{lem:3dCrcurlfemunisolvence}, we have 
$$
\mathbb V_{k+1}^{\curl}(\bs r_1, \bs r_2)\subseteq\mathbb V_{k+1}^{\curl}(\bs r_1), \quad \curl\mathbb V_{k+1}^{\curl}(\bs r_1, \bs r_2)\subseteq\mathbb V_{k}^{\div}(\bs r_2, \bs r_3).
$$
 We illustrate $\mathbb V_{k+1}^{\curl}(\bs r_1, \bs r_2)\subset\boldsymbol{H}(\curl,\Omega)$. First consider $r_2^e\geq0$. For $\boldsymbol{v}\in\mathbb V_{k+1}^{\curl}(\bs r_1, \bs r_2)$, DoFs~\eqref{eq:3dCrcurlfemdofV1}-\eqref{eq:3dCrcurlfemdofV2} and \eqref{eq:3dCrcurlfemdofE4}-\eqref{eq:3dCrcurlfemdofE6} determine $(\nabla^j(\curl\boldsymbol{v}))|_e$ for edge $e$ and $j=0,\ldots,r_2^e$. Due to identities \eqref{eq:edgecurldecompt}-\eqref{eq:edgecurldecompn2}, DoFs~\eqref{eq:3dCrcurlfemdofV1} and \eqref{eq:3dCrcurlfemdofE1}-\eqref{eq:3dCrcurlfemdofE3} then determine $(\boldsymbol{v}\cdot\boldsymbol{t})|_e$ and $(\nabla^j\boldsymbol{v})|_e$ for $j=0,\ldots,r_1^e$. Finally, $\boldsymbol{v}\in\boldsymbol{H}(\curl,\Omega)$ follows from DoFs \eqref{eq:3dCrcurlfemdofF1} and \eqref{eq:3dCrcurlfemdofF3}.

 When $r_2^e=-1$, we have $r_1^e\in\{-1,0\}$. For $\boldsymbol{v}\in\mathbb V_{k+1}^{\curl}(\bs r_1, \bs r_2)$, DoFs~\eqref{eq:3dCrcurlfemdofV1} and \eqref{eq:3dCrcurlfemdofE1}-\eqref{eq:3dCrcurlfemdofE3} determine $(\boldsymbol{v}\cdot\boldsymbol{t})|_e$ for $r_1^e=-1$ and $\boldsymbol{v}|_e$ for $r_1^e=0$. Then $\boldsymbol{v}\in\boldsymbol{H}(\curl,\Omega)$ follows from DoFs \eqref{eq:3dCrcurlfemdofF1} and \eqref{eq:3dCrcurlfemdofF3}.

\subsection{Finite element de Rham and Stokes complexes with decay smoothness}

We consider the decay smoothness sequence with lower bound $-1$
\begin{equation*}
\bs r_0, \quad \bs r_1 = \bs r_0\ominus 1, \quad \bs r_2 = \bs r_1 \ominus 1, \quad \bs r_3 = \bs r_2 \ominus 1.
\end{equation*}

We first consider the case $\bs r_0^f = m\geq 2$. Then $\bs r_i = \bs r_{i-1} - 1\geq 0$ for $i=1,2$. The polynomial degree starts from $k\geq 8m+1\geq 17$.

\begin{lemma} \label{lm:dimension}
Let $r_2^f\geq 0, r_2^{e}\geq 2 \, r_2^f + 2, r_2^{\texttt{v}}\geq 2 \, r_2^e + 2, k\geq 2 r_2^{\texttt{v}} + 3
$, and let $ \bs r_0 = \bs r_1+1, \bs r_1 = \bs r_2+1, \bs r_3 = \bs r_2-1$.
Write
\begin{align*}
\dim \mathbb V^{\grad}_{k+2}(\mathcal T_h; \boldsymbol{r}_0) &= C_{00}|\Delta_0(\mathcal T_h)| + C_{01}|\Delta_1(\mathcal T_h)| + C_{02}|\Delta_2(\mathcal T_h)| + C_{03}|\Delta_3(\mathcal T_h)|,\\
\dim \mathbb V^{\curl}_{k+1}(\mathcal T_h; \boldsymbol{r}_1) &= C_{10}|\Delta_0(\mathcal T_h)| + C_{11}|\Delta_1(\mathcal T_h)| + C_{12}|\Delta_2(\mathcal T_h)| + C_{13}|\Delta_3(\mathcal T_h)|,\\
\dim \mathbb V^{\div}_{k}(\mathcal T_h; \boldsymbol{r}_2)&= C_{20}|\Delta_0(\mathcal T_h)| + C_{21}|\Delta_1(\mathcal T_h)| + C_{22}|\Delta_2(\mathcal T_h)| + C_{23}|\Delta_3(\mathcal T_h)|,\\
\dim \mathbb V^{L^2}_{k-1}(\mathcal T_h; \boldsymbol{r}_3)&= C_{30}|\Delta_0(\mathcal T_h)| + C_{31}|\Delta_1(\mathcal T_h)| + C_{32}|\Delta_2(\mathcal T_h)| + C_{33}|\Delta_3(\mathcal T_h)|.
\end{align*}
Then 
\begin{equation*}
C_{ij} = {3 \choose i} C_j(k+2-i, \boldsymbol{r}_i), \quad i, j=0,1,2,3,
\end{equation*}
satisfy the alternating sum identity
\begin{equation*}
C_{0i} - C_{1i} + C_{2i} - C_{3i} = (-1)^i, \quad i=0,1,2,3.
\end{equation*}
\end{lemma}
\begin{proof}
For the column of  $|\Delta_0(\mathcal T_h)|$, by Lemma~\ref{lem:Cr3dfemdimension} and \eqref{eq:dimensionpolyderham} with $k=r_2^{\texttt{v}}$,
$$
C_{00} - C_{10} + C_{20} - C_{30} = {r_2^{\texttt{v}}+5 \choose 3} - 3{r_2^{\texttt{v}}+4 \choose 3} + 3{r_2^{\texttt{v}}+3 \choose 3} - {r_2^{\texttt{v}}+2 \choose 3}=1.
$$
For the column of  $|\Delta_1(\mathcal T_h)|$, by Lemma~\ref{lem:Cr3dfemdimension} and \eqref{eq:dimensionpolyderham} with $k=r_2^{e}-1$,
\begin{align*}
&\quad C_{01} - C_{11} + C_{21} - C_{31}\\
&=(k+r_2^e-2r_2^{\texttt{v}}-1)\left [{r_2^{e}+4 \choose 2} - 3{r_2^{e}+3 \choose 2} + 3{r_2^{e}+2 \choose 2} - {r_2^{e}+1 \choose 2}\right ] \\
& \quad -{r_2^e+4 \choose 3} + 3{r_2^e+3 \choose 3} - 3{r_2^e+2 \choose 3} + {r_2^e+1 \choose 3}=-1.
\end{align*}
For the column of  $|\Delta_2(\mathcal T_h)|$, by Lemma~\ref{lem:Cr3dfemdimension} and \eqref{eq:dimensionpolyderham},
\begin{align*}
&\quad C_{02} - C_{12} + C_{22} - C_{32}\\
&={k+5 \choose 3} - 3{k+4 \choose 3} + 3{k+3 \choose 3} - {k+2 \choose 3} \\
& \quad -3\left [{r^{\texttt{v}}+5 \choose 3} - 3{r^{\texttt{v}}+4 \choose 3} + 3{r^{\texttt{v}}+3 \choose 3} - {r^{\texttt{v}}+2 \choose 3}\right ] \\
& \quad -3\left [{k-2r^{\texttt{v}}-3 \choose 3} - 3{k-2r^{\texttt{v}}-2 \choose 3} + 3{k-2r^{\texttt{v}}-1 \choose 3} - {k-2r^{\texttt{v}} \choose 3}\right ] \\
&=1-3+3=1.
\end{align*}
For the column of  $|\Delta_3(\mathcal T_h)|$,
applying \eqref{eq:dimensionpolyderham} again,
\begin{align*}
&\quad C_{03} - C_{13} + C_{23} - C_{33}\\
&={k+5 \choose 3} - 3{k+4 \choose 3} + 3{k+3 \choose 3} - {k+2 \choose 3} -4(C_{00} - C_{10} + C_{20} - C_{30})\\
&\quad -6(C_{01} - C_{11} + C_{21} - C_{31}) -4(C_{02} - C_{12} + C_{22} - C_{32}) \\
&=1-4+6-4=-1.
\end{align*}
This ends the proof.  
\end{proof}

We summarize the coefficients $C_{ij}$ in Table \ref{table:dimension}. 

\begin{table}[htbp]
	\centering
\caption{Dimensions of finite element spaces.}
	\renewcommand{\arraystretch}{1.8}
	\begin{tabular}{@{} c c c c c c @{}}
	\toprule
& $\mathbb V^{\grad}_{k+2}(\boldsymbol{r}_0)$ &  $\mathbb V^{\curl}_{k+1}(\boldsymbol{r}_1)$ &   $\mathbb V^{\div}_{k}(\boldsymbol{r}_2)$ & $\mathbb V^{L^2}_{k-1}(\boldsymbol{r}_3)$ & $\sum_{i=0}^3 (-1)^iC_{ij}$\\
\hline
$|\Delta_0(\mathcal T_h)|$ & $C_0(k+2, \bs r_0)$ & $3C_0(k+1, \bs r_1)$ & $3C_0(k, \bs r_2)$ & $C_0(k-1, \bs r_3)$ & $1$ \\
$|\Delta_1(\mathcal T_h)|$ & $C_1(k+2, \bs r_0)$ & $3C_1(k+1, \bs r_1)$ & $3C_1(k, \bs r_2)$ & $C_1(k-1, \bs r_3)$ & $-1$\\
$|\Delta_2(\mathcal T_h)|$ & $C_2(k+2, \bs r_0)$ & $3C_2(k+1, \bs r_1)$  & $3C_2(k, \bs r_2)$ & $C_2(k-1, \bs r_3)$ & $1$\\
$|\Delta_3(\mathcal T_h)|$ & $C_3(k+2, \bs r_0)$ & $3C_3(k+1,\bs r_1)$  & $3C_3(k,\bs r_2)$ & $C_3(k-1, \bs r_3)$ & $-1$\\
%
\bottomrule
\end{tabular}
\label{table:dimension}
\end{table}%

We first consider the case $r_2^f\geq 0$ so that $\mathbb V^{\div}_{k}(\boldsymbol{r}_2)\subset \bs H^1(\Omega;\mathbb R^3)$ and present the following finite element Stokes complex. 
\begin{theorem}\label{lm:femderhamcomplex}
Let $r_2^f\geq 0, r_2^{e}\geq 2 \, r_2^f + 2, r_2^{\texttt{v}}\geq 2 \, r_2^e + 2, k\geq 2 r_2^{\texttt{v}} + 3
$, and let $\bs r_0 = \bs r_1+1, \bs r_1 = \bs r_2+1, \bs r_3 = \bs r_2-1$. The finite element Stokes complex 
\begin{equation}\label{eq:femderhamcomplex3d}
\mathbb R\xrightarrow{\subset}\mathbb V^{\grad}_{k+2}(\boldsymbol{r}_0)\xrightarrow{\grad}\mathbb V^{\curl}_{k+1}(\boldsymbol{r}_1)\xrightarrow{\curl}\mathbb V^{\div}_{k}(\boldsymbol{r}_2)\xrightarrow{\div}\mathbb V^{L^2}_{k-1}(\boldsymbol{r}_3)\to0
\end{equation}
is exact.
\end{theorem}
\begin{proof}
By construction~\eqref{eq:femderhamcomplex3d} is a complex, and  
$$
\grad\mathbb V^{\grad}_{k+2}(\boldsymbol{r}_0)=\mathbb V^{\curl}_{k+1}(\boldsymbol{r}_1)\cap\ker(\curl).
$$
Thanks to \eqref{eq:divonto3dsimple}, $\div\mathbb V^{\div}_{k}(\boldsymbol{r}_2)=\mathbb V^{L^2}_{k-1}(\boldsymbol{r}_3)$.
By Lemma~\ref{lm:dimension} and the Euler's formula,
\begin{align*}
& \quad\; 1 - \dim\mathbb V^{\grad}_{k+2}(\boldsymbol{r}_0) + \dim\mathbb V^{\curl}_{k+1}(\boldsymbol{r}_1) - \dim\mathbb V^{\div}_{k}(\boldsymbol{r}_2) + \dim\mathbb V^{L^2}_{k-1}(\boldsymbol{r}_3) \\
&=  1 - |\Delta_0(\mathcal T_h)| + |\Delta_1(\mathcal T_h)| - |\Delta_2(\mathcal T_h)| + |\Delta_3(\mathcal T_h)| = 0.
\end{align*}
Therefore the exactness of complex~\eqref{eq:femderhamcomplex3d} follows from Lemma~\ref{lm:abstract}.
\end{proof}

We then consider the case $r_2^f=-1$ and thus $\mathbb V^{\div}_{k}(\boldsymbol{r}_2)\subset \bs H(\div,\Omega)$. 
\begin{theorem}\label{lm:femderhamcomplexr2fm1}
	Let $\bs r_{2}$ satisfy \eqref{eq:boundr2fordivbubble} and $r_2^f=-1$. 
	Let $k\geq\max\{2 r_1^{\texttt{v}} + 1,1\}$, and let $\bs r_0 = \bs r_1+1\geq0, \bs r_2 = \bs r_1\ominus1, \bs r_3 = \bs r_2\ominus1$. Assume all $\bs r_i$ are valid smoothness vectors. The finite element de Rham complex
	\begin{equation}\label{eq:femderhamcomplex3dr2fm1}
	\mathbb R\xrightarrow{\subset}\mathbb V^{\grad}_{k+2}(\boldsymbol{r}_0)\xrightarrow{\grad}\mathbb V^{\curl}_{k+1}(\boldsymbol{r}_1)\xrightarrow{\curl}\mathbb V^{\div}_{k}(\boldsymbol{r}_2)\xrightarrow{\div}\mathbb V^{L^2}_{k-1}(\boldsymbol{r}_3)\to0
	\end{equation}
	is exact.
	\end{theorem}
	\begin{proof}
	By construction~\eqref{eq:femderhamcomplex3dr2fm1} is a complex, and  
	$$
	\grad\mathbb V^{\grad}_{k+2}(\boldsymbol{r}_0)=\mathbb V^{\curl}_{k+1}(\boldsymbol{r}_1)\cap\ker(\curl).
	$$
	Thanks to \eqref{eq:divonto3dsimple}, $\div\mathbb V^{\div}_{k}(\boldsymbol{r}_2)=\mathbb V^{L^2}_{k-1}(\boldsymbol{r}_3)$.
	Then we count the dimensions. By comparing DoFs of $\mathbb V^{\div}_{k}(\boldsymbol{r}_2)$ and $\mathbb V^{L^2}_{k-1}(\boldsymbol{r}_3)$,
	\begin{align*}
	&\quad \dim\mathbb V^{\div}_{k}(\boldsymbol{r}_2)-\dim\mathbb V^{L^2}_{k-1}(\boldsymbol{r}_3)\\
	&=|\Delta_0(\mathcal T_h)|\left(3{r_2^{\texttt{v}}+3\choose3}-{r_2^{\texttt{v}}+2\choose3}\right) + |\Delta_1(\mathcal T_h)|\sum_{j=0}^{r_2^e}(2j+3)(k-2r_2^{\texttt{v}}-1+j) \\
	&\quad + |\Delta_2(\mathcal T_h)|\dim\mathbb B_{k} (f;\begin{pmatrix}
		r_2^{\texttt{v}} \\
		r_2^e
		\end{pmatrix}) + |\Delta_3(\mathcal T_h)|(\dim\mathbb B^{\div}_{k}(\boldsymbol{r}_2)\cap \ker(\div) -1).
	\end{align*}
	Then it follows from DoFs \eqref{eq:3dCrcurlfemdofV1}-\eqref{eq:3dCrcurlfemdofT2} of space $\mathbb V^{\curl}_{k+1}(\boldsymbol{r}_1)$ that
	\begin{align*}
		&\quad \dim\mathbb V^{\curl}_{k+1}(\boldsymbol{r}_1)-\dim\mathbb V^{\div}_{k}(\boldsymbol{r}_2)+\dim\mathbb V^{L^2}_{k-1}(\boldsymbol{r}_3)\\
		&=|\Delta_0(\mathcal T_h)|\left(3{r_1^{\texttt{v}}+3\choose3}-3{r_1^{\texttt{v}}+2\choose3}+{r_1^{\texttt{v}}+1\choose3}\right) \\
		&\quad +|\Delta_1(\mathcal T_h)|(k-2r_1^{\texttt{v}})+ |\Delta_1(\mathcal T_h)|\sum_{j=0}^{r_1^e}(j+2)(k-2r_1^{\texttt{v}}+j) \\
		&\quad +|\Delta_2(\mathcal T_h)|\dim\mathbb B_{k+2} (f;\begin{pmatrix}
			r_0^{\texttt{v}} \\
			r_0^e
			\end{pmatrix})+\chi(r_1^f=0)|\Delta_2(\mathcal T_h)|\dim\mathbb B_{k+1} (f;\begin{pmatrix}
				r_1^{\texttt{v}} \\
				r_1^e
				\end{pmatrix}) \\
		&\quad- |\Delta_2(\mathcal T_h)| + |\Delta_3(\mathcal T_h)|(\dim\mathbb B_{k+2}(\boldsymbol{r}_0) +1).
		\end{align*}
As a result, by the Euler's formula,
	\begin{align*}
	& \quad\; 1 - \dim\mathbb V^{\grad}_{k+2}(\boldsymbol{r}_0) + \dim\mathbb V^{\curl}_{k+1}(\boldsymbol{r}_1) - \dim\mathbb V^{\div}_{k}(\boldsymbol{r}_2) + \dim\mathbb V^{L^2}_{k-1}(\boldsymbol{r}_3) \\
	&=  1 - |\Delta_0(\mathcal T_h)| + |\Delta_1(\mathcal T_h)| - |\Delta_2(\mathcal T_h)| + |\Delta_3(\mathcal T_h)| = 0.
	\end{align*}
	Therefore the exactness of complex~\eqref{eq:femderhamcomplex3dr2fm1} follows from Lemma~\ref{lm:abstract}.
	\end{proof}

%
%
%
%
%

\subsection{Finite element de Rham and Stokes complex with inequality constraint}
We consider the most general case.
\begin{theorem}\label{thm:femderhamcomplex3dgeneral}
Let $\bs r_0 \geq 0, \bs r_1 = \bs r_0 -1, \bs r_2\geq\bs r_1\ominus1, \bs r_3\geq \bs r_2\ominus1$. Assume all $\bs r_i$ are valid smoothness vectors, i.e.,
$$
r_1^{\texttt{v}}\geq2r_1^e+1,\; r_1^{e}\geq2r_1^f+1,\; r_2^{\texttt{v}}\geq2r_2^e,\; r_2^{e}\geq2r_2^f, \; r_3^{\texttt{v}}\geq2r_3^e,\; r_3^{e}\geq2r_3^f.
$$
Assume $(\bs r_2, \bs r_3)$ is a div stable pair. 
Assume $k\geq\max\{2 r_1^{\texttt{v}} + 1,2 r_2^{\texttt{v}} + 1,2 r_3^{\texttt{v}} + 2,1\} 
$, $\dim\mathbb B_{k-1}(T;\boldsymbol{r}_3)\geq1$, and $\dim\mathbb B_{k}(f;\begin{pmatrix}
r_2^{\texttt{v}} \\
r_2^e
\end{pmatrix})\geq1$. The finite element complex \begin{equation}\label{eq:femderhamcomplex3dgeneral}
\mathbb R\xrightarrow{\subset}\mathbb V^{\grad}_{k+2}(\boldsymbol{r}_0)\xrightarrow{\grad}\mathbb V^{\curl}_{k+1}(\boldsymbol{r}_1,\boldsymbol{r}_2)\xrightarrow{\curl}\mathbb V^{\div}_{k}(\boldsymbol{r}_2,\boldsymbol{r}_3)\xrightarrow{\div}\mathbb V^{L^2}_{k-1}(\boldsymbol{r}_3)\to0
\end{equation}
is exact.
\end{theorem}
\begin{proof}
By construction~\eqref{eq:femderhamcomplex3dgeneral} is a complex, and  
$$
\mathbb V^{\curl}_{k+1}(\boldsymbol{r}_1,\boldsymbol{r}_2)\cap\ker(\curl)=\mathbb V^{\curl}_{k+1}(\boldsymbol{r}_1)\cap\ker(\curl)=\grad\mathbb V^{\grad}_{k+2}(\boldsymbol{r}_0).
$$
Thanks to \eqref{eq:divonto3d}, $\div\mathbb V^{\div}_{k}(\boldsymbol{r}_2,\boldsymbol{r}_3)=\mathbb V^{L^2}_{k-1}(\boldsymbol{r}_3)$.
By Lemma~\ref{lm:abstract}, it suffices to prove
\begin{equation}\label{eq:femderhamcomplex3ddimenidentity}
\dim \mathbb V^{\grad}_{k+2}(\boldsymbol{r}_0) - \dim\mathbb V^{\curl}_{k+1}(\boldsymbol{r}_1,\boldsymbol{r}_2) + \dim\mathbb V^{\div}_{k}(\boldsymbol{r}_2,\boldsymbol{r}_3) - \dim\mathbb V^{L^2}_{k-1}(\boldsymbol{r}_3) = 1.
\end{equation}
Through comparing DoFs, we find that $\dim\mathbb V^{\div}_{k}(\boldsymbol{r}_2,\boldsymbol{r}_3) - \dim\mathbb V^{L^2}_{k-1}(\boldsymbol{r}_3)$ is constant with respect to $\boldsymbol{r}_3$, which
$$
\dim\mathbb V^{\div}_{k}(\boldsymbol{r}_2,\boldsymbol{r}_3) - \dim\mathbb V^{L^2}_{k-1}(\boldsymbol{r}_3)=\dim\mathbb V^{\div}_{k}(\boldsymbol{r}_2) - \dim\mathbb V^{L^2}_{k-1}(\boldsymbol{r}_2\ominus1).
$$
Similarly, since $\dim\mathbb V^{\curl}_{k+1}(\boldsymbol{r}_1,\boldsymbol{r}_2) - \dim\mathbb V^{\div}_{k}(\boldsymbol{r}_2) + \dim\mathbb V^{L^2}_{k-1}(\boldsymbol{r}_2\ominus1)$ is constant with respect to $\boldsymbol{r}_2$, we have
\begin{align*}	
&\quad\dim\mathbb V^{\curl}_{k+1}(\boldsymbol{r}_1,\boldsymbol{r}_2) - \dim\mathbb V^{\div}_{k}(\boldsymbol{r}_2) + \dim\mathbb V^{L^2}_{k-1}(\boldsymbol{r}_2\ominus1) \\
&=\dim\mathbb V^{\curl}_{k+1}(\boldsymbol{r}_1) - \dim\mathbb V^{\div}_{k}(\boldsymbol{r}_1\ominus1) + \dim\mathbb V^{L^2}_{k-1}(\boldsymbol{r}_1\ominus2).
\end{align*}
Combining the last two identities yields
\begin{align*}	
&\quad -\dim\mathbb V^{\curl}_{k+1}(\boldsymbol{r}_1,\boldsymbol{r}_2)+\dim\mathbb V^{\div}_{k}(\boldsymbol{r}_2,\boldsymbol{r}_3) - \dim\mathbb V^{L^2}_{k-1}(\boldsymbol{r}_3) \\
&=-\dim\mathbb V^{\curl}_{k+1}(\boldsymbol{r}_1)+\dim\mathbb V^{\div}_{k}(\boldsymbol{r}_1\ominus1) - \dim\mathbb V^{L^2}_{k-1}(\boldsymbol{r}_1\ominus2).
\end{align*}
Therefore \eqref{eq:femderhamcomplex3ddimenidentity} follows from Theorem~\ref{lm:femderhamcomplex}.
\end{proof}

\begin{example}\rm 
Taking $\boldsymbol{r}_0=(1,0,0)^{\intercal}$, $\boldsymbol{r}_1=\boldsymbol{r}_0-1$, $\boldsymbol{r}_2=\boldsymbol{r}_1\ominus1$, $\boldsymbol{r}_3=\boldsymbol{r}_2\ominus1$ and $k\geq1$, we obtain the Hermite family finite element de Rham complex in~\cite{Christiansen;Hu;Hu:2018finite}
	 $$
	\begin{pmatrix}
	 1\\
	 0\\
	 0
	\end{pmatrix}
	\xrightarrow{\grad}
	\begin{pmatrix}
	 0\\
	 -1\\
	 -1
	\end{pmatrix}
	\xrightarrow{\curl}
	\begin{pmatrix}
	 -1\\
	 -1\\
	 -1
	\end{pmatrix}
	\xrightarrow{\div} 
	\begin{pmatrix}
	 -1\\
	 -1\\
	 -1
	\end{pmatrix}.
	 $$
	 Taking $\boldsymbol{r}_0=(2,1,0)^{\intercal}$, $\boldsymbol{r}_1=\boldsymbol{r}_0-1$, $\boldsymbol{r}_2=\boldsymbol{r}_1\ominus1$, $\boldsymbol{r}_3=\boldsymbol{r}_2\ominus1$ and $k\geq3$, we obtain the Argyris family finite element de Rham complex in~\cite{Christiansen;Hu;Hu:2018finite}
	 $$
	\begin{pmatrix}
	 2\\
	 1\\
	 0
	\end{pmatrix}
	\xrightarrow{\grad}
	\begin{pmatrix}
	 1\\
	 0\\
	 -1
	\end{pmatrix}
	\xrightarrow{\curl}
	\begin{pmatrix}
	 0\\
	 -1\\
	 -1
	\end{pmatrix}
	\xrightarrow{\div} 
	\begin{pmatrix}
	 -1\\
	 -1\\
	 -1
	\end{pmatrix}.
	 $$
\end{example}

\begin{example}\rm 
Taking $\boldsymbol{r}_0=(4,2,1)^{\intercal}$, $\boldsymbol{r}_1=\boldsymbol{r}_0-1$, $\boldsymbol{r}_2=\boldsymbol{r}_1-1$, $\boldsymbol{r}_3=\boldsymbol{r}_2\ominus1$ and $k\geq7$, the finite element de Rham complex 
$$
\begin{pmatrix}
	 4\\
	 2\\
	 1
\end{pmatrix}
\xrightarrow{\grad}
\begin{pmatrix}
 3\\
 1\\
 0
\end{pmatrix}
\xrightarrow{\curl}
\begin{pmatrix}
	2\\
	0\\
	-1
\end{pmatrix}
\xrightarrow{\div} 
\begin{pmatrix}
	 1\\
	 -1\\
	 -1
\end{pmatrix}
$$
can be used to discretize the decoupled formulation of the biharmonic equation in three dimensions in \cite[Section~3.2]{ChenHuang2018}.
\end{example}

\begin{example}\rm 
Taking $\boldsymbol{r}_0=(4,2,1)^{\intercal}$, $\boldsymbol{r}_2=(2,1,0)^{\intercal}$, $\boldsymbol{r}_1=\boldsymbol{r}_0-1$, $\boldsymbol{r}_3=\boldsymbol{r}_2-1$ and $k\geq7$, we obtain the Stokes complex in~\cite{Neilan2015}
$$
\begin{pmatrix}
	 4\\
	 2\\
	 1
\end{pmatrix}
\xrightarrow{\grad}
\begin{pmatrix}
 3\\
 1\\
 0
\end{pmatrix}
\xrightarrow{\curl}
\begin{pmatrix}
	2\\
	1\\
	0
\end{pmatrix}
\xrightarrow{\div} 
\begin{pmatrix}
	 1\\
	 0\\
	 -1
\end{pmatrix}.
$$
\end{example}

\subsection{Commutative diagram}
To construct a commutative diagram for finite element complex \eqref{eq:femderhamcomplex3dgeneral}, we adjust DoFs~\eqref{eq:C13d0}-\eqref{eq:C13d3} of $\mathbb V_{k+2}^{\grad}(\bs r_0)$ in consideration of DoFs~\eqref{eq:3dCrcurlfemdofV1}-\eqref{eq:3dCrcurlfemdofT2} of $\mathbb V_{k+1}^{\curl}(\bs r_1, \bs r_2)$. We present new DoFs for $\mathbb V_{k+2}^{\grad}(\bs r_0)$ as follows:
\begin{align}
	\nabla^ju(\texttt{v}), & \quad j=0,\ldots, r_0^{\texttt{v}}, \label{eq:3dCrgradfemdofV1}\\
	\int_e \partial_tu\,q \dd s, &\quad q\in \mathbb P_{k - 2r_0^{\texttt{v}}+1}(e)/\mathbb R, \label{eq:3dCrgradfemdofE1}\\
	\int_e \frac{\partial^{j}u}{\partial n_1^{i}\partial n_2^{j-i}}\, q \dd s, &\quad q\in \mathbb P_{k- 2r_0^{\texttt{v}}+j}(e), 0\leq i\leq j, 1\leq j\leq r_0^e, \label{eq:3dCrgradfemdofE2}\\
	\int_f (\grad_fu)\cdot\boldsymbol{q}\dd S, &\quad \boldsymbol{q}\in\grad_f\mathbb B_{k+2} (f;\begin{pmatrix}
	r_0^{\texttt{v}} \\
	r_0^e
	\end{pmatrix}),\label{eq:3dCrgradfemdofF1}\\
	\int_f \partial_n^ju\, q\dd S, &\quad  q\in\mathbb B_{k+2-j} (f;\begin{pmatrix}
	r_0^{\texttt{v}} \\
	r_0^e
	\end{pmatrix}-j), 1\leq j\leq r_0^f, \label{eq:3dCrgradfemdofF2}\\
	\int_T (\grad u)\cdot\boldsymbol{q} \dx, &\quad \boldsymbol{q}\in \grad\mathbb B_{k+2}(\boldsymbol{r}_0) \label{eq:3dCrgradfemdofT1}
	\end{align}
	for each $\texttt{v}\in \Delta_{0}(T)$, $e\in \Delta_{1}(T)$ and $f\in \Delta_{2}(T)$.
	
\begin{lemma}\label{lem:3dCrgradfemunisolvence}
The DoFs~\eqref{eq:3dCrgradfemdofV1}-\eqref{eq:3dCrgradfemdofT1} are uni-solvent for $\mathbb P_{k+2}(T)$.  
\end{lemma}
\begin{proof}
By comparing DoFs~\eqref{eq:C13d0}-\eqref{eq:C13d3} and DoFs~\eqref{eq:3dCrgradfemdofV1}-\eqref{eq:3dCrgradfemdofT1},
the number of DoFs~\eqref{eq:3dCrgradfemdofV1}-\eqref{eq:3dCrgradfemdofT1} equals to $\dim\mathbb P_{k+2}(T)$.

Assume $u\in\mathbb P_{k+2}(T)$ and all the DoFs~\eqref{eq:3dCrgradfemdofV1}-\eqref{eq:3dCrgradfemdofT1} vanish.
The vanishing DoFs~\eqref{eq:3dCrgradfemdofV1}-\eqref{eq:3dCrgradfemdofE1} imply $\int_e u\,q \dd s=0$ for $q\in \mathbb P_{k - 2r_0^{\texttt{v}}}(e)$. Thanks to Theorem \ref{thm:Cr3dfemunisolvence}, we get from the vanishing DoFs~\eqref{eq:3dCrgradfemdofV1} and \eqref{eq:3dCrgradfemdofE2} that $\nabla^ju|_e=0$ for $0\leq j\leq r_0^e$ and $e\in\Delta_1(T)$.
Then $u|_f\in\mathbb B_{k+2} (f;\begin{pmatrix}
	r_0^{\texttt{v}} \\
	r_0^e
	\end{pmatrix})$ for $f\in\Delta_2(T)$, which combined with \eqref{eq:3dCrgradfemdofF1} yields $u|_f=0$. Applying Theorem \ref{thm:Cr3dfemunisolvence} again, it follows from  the vanishing DoF \eqref{eq:3dCrgradfemdofF2} that $\partial_n^ju|_f=0$ for $0\leq j\leq r_0^f$, i.e. $u\in\mathbb B_{k+2}(\boldsymbol{r}_0)$. Thus $u=0$ holds from the vanishing DoF \eqref{eq:3dCrgradfemdofT1}.
\end{proof}

Define $I_h^{\grad}, I_h^{\curl}, I_h^{\div}$, and $I_h^{L^2}$ as the canonical interpolation operators using the DoFs. 

\begin{corollary}
With the same setting in Theorem \ref{thm:femderhamcomplex3dgeneral}, the following diagram is commutative. 
$$
\begin{array}{c}
\xymatrix{
\mathbb R \ar[r]^-{\subset} & \mathcal C^{\infty}(\Omega) \ar[d]^{I_h^{\grad}} \ar[r]^-{\grad}                & \mathcal C^{\infty}(\Omega;\mathbb R^3) \ar[d]^{I_h^{\curl}}   \ar[r]^-{\curl} & \mathcal C^{\infty}(\Omega;\mathbb R^3) \ar[d]^{I_h^{\div}}   \ar[r]^-{\div} & \ar[d]^{I_h^{L^2}} \mathcal C^{\infty}(\Omega) \ar[r]^{} & 0 \\
 \mathbb R \ar[r]^-{\subset} & \mathbb V_{k+2}^{\grad}(\bs r_0) \ar[r]^{\grad}
                &  \mathbb V_{k+1}^{\curl}(\bs r_1, \bs r_2)   \ar[r]^{\curl} & \mathbb V_{k}^{\div}(\bs r_2, \bs r_3)  \ar[r]^{\div} &  \mathbb V_{k-1}^{L^2}(\bs r_3) \ar[r]^{}& 0.    }
\end{array}
$$
\end{corollary}
\begin{proof}
\step 1  We first prove
\begin{equation}\label{eq:divcommuting}
\div(I_h^{\div}\boldsymbol{v})= I_h^{L^2}(\div\boldsymbol{v})\quad\forall~\boldsymbol{v}\in\mathcal C^{\infty}(\Omega;\mathbb R^3).
\end{equation}
By comparing DoFs~\eqref{eq:C13d0}-\eqref{eq:C13d3} for $\mathbb V_{k-1}^{L^2}(\bs r_3)$, and DoFs~\eqref{eq:3dCrdivfemdofV1}-\eqref{eq:3dCrdivfemdofV2}, \eqref{eq:3dCrdivfemdofE4} and~\eqref{eq:3dCrdivfemdofF3}-\eqref{eq:3dCrdivfemdofT1} for $\mathbb V_{k}^{\div}(\bs r_2,\bs r_3)$, it suffices to prove
$$
\int_T\div(I_h^{\div}\boldsymbol{v})\dx=\int_T\div\boldsymbol{v}\dx\quad\forall~T\in\mathcal T_h,
$$
which is an immediate result of DoF \eqref{eq:3dCrdivfemdofF1}.

\step 2  Next we prove
\begin{equation}\label{eq:curlcommuting}
\curl(I_h^{\curl}\boldsymbol{v})= I_h^{\div}(\curl\boldsymbol{v})\quad\forall~\boldsymbol{v}\in\mathcal C^{\infty}(\Omega;\mathbb R^3).
\end{equation}
By \eqref{eq:divcommuting}, we have $\div(\curl(I_h^{\curl}\boldsymbol{v}))=\div(I_h^{\div}(\curl\boldsymbol{v}))=0$.
By comparing DoFs~\eqref{eq:3dCrdivfemdofV1}, \eqref{eq:3dCrdivfemdofE1}-\eqref{eq:3dCrdivfemdofE3}, \eqref{eq:3dCrdivfemdofF1}-\eqref{eq:3dCrdivfemdofF2} and~\eqref{eq:3dCrdivfemdofT2} for $\mathbb V_{k}^{\div}(\bs r_2,\bs r_3)$, and DoFs~\eqref{eq:3dCrcurlfemdofV1}-\eqref{eq:3dCrcurlfemdofV2}, \eqref{eq:3dCrcurlfemdofE4}-\eqref{eq:3dCrcurlfemdofE6} and~\eqref{eq:3dCrcurlfemdofF3}-\eqref{eq:3dCrcurlfemdofT1} for $\mathbb V_{k+1}^{\curl}(\bs r_1,\bs r_2)$, it suffices to prove
$$
\int_f\curl(I_h^{\curl}\boldsymbol{v})\cdot\boldsymbol{n}\dd S=\int_f(\curl\boldsymbol{v})\cdot\boldsymbol{n}\dd S\quad\forall~f\in\Delta_2(\mathcal T_h),
$$
which is an immediate result of DoF \eqref{eq:3dCrcurlfemdofE1}.

\step 3  Finally we prove
\begin{equation}\label{eq:gradcommuting}
\grad(I_h^{\grad}u)= I_h^{\curl}(\grad u)\quad\forall~u\in\mathcal C^{\infty}(\Omega).
\end{equation}
By \eqref{eq:curlcommuting}, we have $\curl(\grad(I_h^{\grad}u))=\curl(I_h^{\curl}(\grad u))=\boldsymbol{0}$.
By comparing DoFs~\eqref{eq:3dCrcurlfemdofV1}, \eqref{eq:3dCrcurlfemdofE1}-\eqref{eq:3dCrcurlfemdofE3}, \eqref{eq:3dCrcurlfemdofF1}-\eqref{eq:3dCrcurlfemdofF2} and~\eqref{eq:3dCrcurlfemdofT2} for $\mathbb V_{k+1}^{\curl}(\bs r_1,\bs r_2)$, and DoFs~\eqref{eq:3dCrgradfemdofV1}-\eqref{eq:3dCrgradfemdofT1} for $\mathbb V_{k+2}^{\grad}(\bs r_0)$, it suffices to prove
$$
\int_e\partial_t(I_h^{\grad}u)\dd s=\int_e\partial_tu\dd s\quad\forall~e\in\Delta_1(\mathcal T_h),
$$
which is an immediate result of DoF \eqref{eq:3dCrgradfemdofV1}.

Combining \eqref{eq:divcommuting}-\eqref{eq:gradcommuting} will end the proof.
\end{proof}

\appendix
\section{Smooth Finite Elements in Arbitrary Dimension}\label{sec:geodecompnd}
In a recent work~\cite{huConstructionConformingFinite2021}, Hu, Lin and Wu have constructed a $C^m$-conforming finite elements on simplexes in arbitrary dimension. It unifies the scattered results ~\cite{BrambleZlamal1970,Zenisek1970,ArgyrisFriedScharpf1968} in two dimensions,~\cite{Zenisek1974a,Zhang2009a,Lai;Schumaker:2007Trivariate} in three dimensions, and~\cite{Zhang2016a} in four dimensions. In this appendix, we use the simplicial lattice to give a geometric decomposition of the finite element spaces constructed in~\cite{huConstructionConformingFinite2021} and consequently give a simplified construction different from~\cite{huConstructionConformingFinite2021}.
The smoothness at sub-simplexes is exponentially increasing as the dimension decreases
$$
r_{n}=0,\;\; r_{n-1}=m,\;\; r_{\ell}\geq 2r_{\ell+1} \; \textrm{ for } \ell=n-2,\ldots, 0.
$$
And the degree of polynomial $k\geq 2r_0+1 \geq 2^n m + 1$. The key in the construction is a non-overlapping decomposition of the simplicial lattice in which each component will be used to determine the normal derivatives on lower sub-simplexes.

 
Our approach is closely related to the multivariate splines on triangulations~\cite{Chui;Lai:1990Multivariate,Lai;Schumaker:2007Spline}. 
For example, a systematical construction of $C^m$ element in $n=2,3$ can be also found in the book~\cite[Section 8.1 for 2D and Section 18.11 for 3D]{Lai;Schumaker:2007Spline}. 
The major difference, which is also the art of designing finite elements, is the choice of DoFs. In the multivariate splines, DoFs are chosen as function values or derivatives at some points, as the major question studied there is the interpolation of data, while the integral form on sub-simplexes enables us to prove the unisolvence easily and has the advantage for constructing finite element de Rham complexes as we have done in three dimensions. 

\subsection{Important relation}
The first important relation is: for $\alpha\in \mathbb T^{n}_k, \beta \in \mathbb N^{1:n}$, we have
\begin{equation*}
D^{\beta} \lambda^{\alpha}|_{f} = 0, \quad \text{ if } \dist(\alpha, f) > |\beta|.
\end{equation*}
Namely the polynomial $\lambda^{\alpha}$ vanishes on $f$ to order $\dist(\alpha, f)$.

The second one is the one-to-one mapping of the space $\spa\{ \lambda^{\alpha} = \lambda_f^{\alpha_f}\lambda_{f^*}^{\alpha_{f^*}} , \alpha\in L(f,s), i.e., \alpha \in \mathbb T^{n}_k, | \alpha_{f^*}| = s \}$ to the following DoFs, by changing $\alpha_{f^*}$ to $\beta$, 
\begin{equation*}
\int_f  \frac{\partial^{\beta} u}{\partial n_f^{\beta}} \, \lambda_f^{\alpha_f} \dd s \quad \forall~\alpha_f\in \mathbb T^{\ell}_{k-s}(f), \beta \in \mathbb N^{1:n-\ell}, |\beta | = s.
\end{equation*}

\subsection{A decomposition of the simplicial lattice}
We explain the requirement $r_{\ell-1}\geq 2r_{\ell}$. 
\begin{lemma}\label{lm:appdisjoint}
Let $T$ be an $n$-dimensional simplex. For $\ell = 1,\ldots, n-1$, if $r_{\ell-1}\geq 2r_{\ell}$, the sub-sets $\{ D(f, r_{\ell}) \backslash \left [ \cup_{e\in \Delta_{\ell - 1}(f)} D(e, r_{\ell - 1})\right ], f\in \Delta_{\ell} (T)\}$ are disjoint.
\end{lemma}
\begin{proof}
Consider two different sub-simplices $ f, \tilde f \in \Delta_{\ell} (T)$. The dimension of their intersection is at most $\ell - 1$. Therefore $f\cap \tilde f\subseteq e$ for some $e\in \Delta_{\ell -1}(f)$. Then $e^*\subseteq (f\cap \tilde f )^* = f^*\cup \tilde f^*$. For $\alpha \in D(f, r_{\ell})\cap D(\tilde f, r_{\ell})$, we have $|\alpha_{e^*}| \leq |\alpha_{f^*}| + |\alpha_{\tilde f^*}|\leq 2r_{\ell}\leq r_{\ell - 1}$. Therefore we have shown the intersection region $D(f, r_{\ell})\cap D(\tilde f, r_{\ell})\subseteq \cup_{e\in \Delta_{\ell - 1}(f)} D(e,r_{\ell-1})$ and the result follows.
%
\end{proof}

Next we remove $D(e, r_{i})$ from $D(f, r_{\ell})$ for all $e\in \Delta_{i}(T)$ and $i=0,1,\ldots, \ell-1$. 
\begin{lemma} \label{lm:appDeltaf=DeltaT}
Given integer $m\geq 0$, let non-negative integer array $\bs r=(r_0,r_1, \cdots, r_n)$ satisfy
$$
r_{n}=0,\;\; r_{n-1}=m,\;\; r_{\ell}\geq 2r_{\ell+1} \; \textrm{ for } \ell=n-2,\ldots, 0.
$$
Let $k\geq 2r_0+1 \geq 2^n m + 1$. For $\ell = 1,\dots, n-1,$
\begin{equation}\label{eq:appDeltaf=DeltaT}
D(f, r_{\ell}) \backslash \left [ \bigcup_{i=0}^{\ell-1}\bigcup_{e\in \Delta_{i}(f)}D(e, r_{i}) \right ] = D(f, r_{\ell}) \backslash \left [ \bigcup_{i=0}^{\ell-1}\bigcup_{e\in \Delta_{i}(T)}D(e, r_{i}) \right ] .
\end{equation}
\end{lemma}
\begin{proof}
In \eqref{eq:Deltaf=DeltaT}, the relation $\supseteq$ is obvious as $\Delta_i(f)\subseteq \Delta_i(T)$. 

To prove $\subseteq$, it suffices to show for $\alpha \in D(f, r_{\ell}) \backslash \left [ \bigcup_{i=0}^{\ell-1}\bigcup_{e\in \Delta_{i}(f)}D(e, r_{i}) \right ]$, it is not in $D(e,r_i)$ for $e\in \Delta_i(T)$ and $e\not\in\Delta_i(f)$. 

By definition, 
$$
|\alpha_{f^*}|\leq r_{\ell},\; |\alpha_{e}|\leq k - r_{i}-1 \; \textrm{ for all }e\in\Delta_i(f), i=0,\ldots,\ell-1.
$$
For each $e\in\Delta_i(T)$ but $e\not\in\Delta_i(f)$, the dimension of the intersection $e \cap f$ is at most $i-1$. It follows from $r_{j}\geq 2r_{j+1}$ and $k\geq 2r_0+1$ that: when $i>0$,
$$
|\alpha_{e}|=|\alpha_{e\cap f}|+|\alpha_{e\cap f^*}|\leq k - r_{i-1}-1+r_{\ell}\leq k - r_{i}-1,
$$
and when $i=0$,
$$
|\alpha_{e}|=|\alpha_{e\cap f^*}|\leq r_{\ell}\leq k - r_{i}-1.
$$
So $|\alpha_{e^*}| > r_i$. We conclude that $\alpha \not\in D(e, r_i)$ for all $e\in\Delta_i(T)$ and \eqref{eq:Deltaf=DeltaT} follows. 
\end{proof}

We are in the position to present our main decomposition. 
\begin{theorem}\label{th:appdecT}
Given integer $m\geq 0$, let non-negative integer array $\bs r=(r_0,r_1, \cdots, r_n)$ satisfy
$$
r_{n}=0,\;\; r_{n-1}=m,\;\; r_{\ell}\geq 2r_{\ell+1} \; \textrm{ for } \ell=n-2,\ldots, 0.
$$
Let $k\geq 2r_0+1 \geq 2^n m + 1$. Then we have the following direct decomposition of the simplicial lattice  on an $n$-dimensional simplex $T$:
\begin{equation*}
  \mathbb T^{n}_k(T) = \Oplus_{\ell = 0}^{n}\Oplus_{f\in \Delta_{\ell}(T)} S_{\ell}(f),
\end{equation*}
where
\begin{align*}
S_0(\texttt{v}) &=  D(\texttt{v}, r_0), \\
S_{\ell}(f) &= D(f, r_{\ell}) \backslash \left [ \bigcup_{i=0}^{\ell-1}\bigcup_{e\in \Delta_{i}(f)}D(e, r_{i}) \right ], \; \ell = 1,\dots, n-1, \\
S_n(T) & = \mathbb T^{n}_k(T) \backslash  \left [  \bigcup_{i=0}^{n-1}\bigcup_{f\in \Delta_{i}(T)}D(f, r_{i}) \right ].
\end{align*}
Consequently we have the following geometric decomposition of $\mathbb P_{k}(T)$
\begin{equation}\label{eq:PrSdec}
\mathbb P_{k}(T) = \Oplus_{\ell = 0}^{n} \Oplus_{f\in \Delta_{\ell}(T)} \mathbb P_k(S_{\ell}(f)).
\end{equation}
\end{theorem}
\begin{proof}
First we show that the sets  $\{S_{\ell}(f), f\in\Delta_{\ell}(T), \ell=0,\ldots,n\}$ are disjoint.
Take two vertices $\texttt{v}_1, \texttt{v}_2\in \Delta_0(T)$. For $\alpha\in D(\texttt{v}_1, r_0)$, we have $\alpha_{\texttt{v}_1} \geq k - r_0$. As $\texttt{v}_1\subseteq \texttt{v}_2^*$ and $k\geq 2r_0+1$, $|\alpha_{\texttt{v}_2^*}|\geq \alpha_{\texttt{v}_1}\geq k-r_0\geq r_0+1$, i.e., $\alpha \notin D(\texttt{v}_2, r_0)$. Hence $\{S_{0}(\texttt{v}), \texttt{v}\in\Delta_{0}(T)\}$ are disjoint and $\Oplus_{\texttt{v}\in\Delta_{0}(T)} S_{0}(\texttt{v})$ is a disjoint union.
By Lemma \ref{lm:appdisjoint} and \eqref{eq:appDeltaf=DeltaT}, we know $\{S_{\ell}(f), f\in\Delta_{\ell}(T), \ell=0,\ldots,n\}$ are disjoint.

Next we inductively prove 
\begin{equation*}
\Oplus_{i = 0}^{\ell}\Oplus_{f\in \Delta_{i}(T)} S_{i}(f)=\bigcup_{i=0}^{\ell}\bigcup_{f\in \Delta_{i}(T)}D(f, r_{i}) \quad \textrm{ for }\; \ell=0,\ldots, n-1.
\end{equation*}
Obviously \eqref{eq:USequalUD} holds for $\ell=0$. Assume \eqref{eq:USequalUD} holds for $\ell<j$. Then
\begin{align*}
&\quad\Oplus_{i = 0}^{j}\Oplus_{f\in \Delta_{i}(T)} S_{i}(f)= \Oplus_{f\in \Delta_{j}(T)} S_{j}(f)\;\oplus \;\bigcup_{i=0}^{j-1}\bigcup_{e\in \Delta_{i}(T)}D(e, r_{i}) \\
&= \Oplus_{f\in \Delta_{j}(T)}\left(D(f, r_{j}) \backslash \left [ \bigcup_{i=0}^{j-1}\bigcup_{e\in \Delta_{i}(T)}D(e, r_{i}) \right ]\right)\;\oplus \;\bigcup_{i=0}^{j-1}\bigcup_{e\in \Delta_{i}(T)}D(e, r_{i}) \\
&=\bigcup_{i=0}^{j}\bigcup_{f\in \Delta_{i}(T)}D(f, r_{i}).
\end{align*}
By induction, \eqref{eq:USequalUD} holds for $\ell=0,\ldots, n-1$. Then \eqref{eq:smoothdecnd} is true from the definition of $S_n(T)$ and \eqref{eq:USequalUD}. 
\end{proof}

We can write out the inequality constraints in $S_{\ell}(f)$. For $\ell = 1,\dots, n$,
\begin{equation}\label{eq:Slfineqlty}
S_{\ell}(f) =  \{\alpha\in\mathbb T_{k}^{n}: |\alpha_{f^*}|\leq r_{\ell}, |\alpha_{e}|\leq k - r_{i}-1, \forall e\in\Delta_i(f), i=0,\ldots,\ell-1\}.
\end{equation}
For $\alpha\in S_{\ell}(f)$, by Lemma~\ref{lm:disjoint} we also have $\alpha\not\in D(\tilde f, r_{\ell})$ for $\tilde f \in \Delta_{\ell} (T)\backslash\{f\}$, i.e.
\begin{equation}\label{eq:Slfineqlty2}
|\alpha_{\tilde f}|\leq k - r_{\ell}-1 \quad\forall~\tilde f \in \Delta_{\ell} (T)\backslash\{f\}.
\end{equation}
From the implementation point of view, the index set $S_{\ell}(f)$ can be found by a logic array and set the entry as true when the distance constraint holds. 

\subsection{Decomposition of degree of freedoms}
Recall that $L(f,s) = \{ \alpha \in \mathbb T^{n}_k, \dist(\alpha,f) = s\}$ consists of lattice nodes $s$ away from $f$. 
\begin{lemma}
Let $\ell=0,\ldots, n-1$ and $s\leq r_{\ell}$ be a non-negative integer. Given $f\in \Delta_{\ell}(T)$, let $n_f = \{n_f^1, n_f^2, \ldots, n_f^{n - \ell}\}$ be $n- \ell$ vectors spanning the normal plane of $f$. The polynomial space $\mathbb P_k(S_{\ell}(f)\cap L(f,s))$ is uniquely determined by DoFs
\begin{equation}\label{eq:normalDof}
\int_f  \frac{\partial^{\beta} u}{\partial n_f^{\beta}} \, \lambda_f^{\alpha_f} \dd s \quad \forall~\alpha\in S_{\ell}(f), |\alpha_f| = k - s, \beta \in \mathbb N^{1:n-\ell}, |\beta | = s.
\end{equation}
\end{lemma}
\begin{proof}
A basis of $\mathbb P_k(S_{\ell}(f)\cap L(f,s))$ is
$\{ \lambda^{\alpha} = \lambda_f^{\alpha_f}\lambda_{f^*}^{\alpha_{f^*}} , \alpha \in S_{\ell}(f), | \alpha_{f^*} | = s \}$ and thus the dimensions match (by mapping $\alpha_{f^*}$ to $\beta$).

We choose a basis of the normal plane $\{n_f^1, n_f^2, \ldots, n_f^{n - \ell}\}$ s.t. it is dual to the vectors $\{ \nabla \lambda_{f^*(1)}, \nabla \lambda_{f^*(2)}, \ldots, \}$, i.e., $\nabla \lambda_{f^*(i)}\cdot n_f^j = \delta_{i,j}$ for $i,j=1,\ldots, n - \ell$. Then we have the duality
\begin{equation}\label{eq:Dbeta}
\frac{\partial^{\beta} }{\partial n_f^{\beta}} (\lambda_{f^*}^{\alpha_{f^*}}) = \beta!\delta( \alpha_{f^*}, \beta), \quad \alpha_{f^*}, \beta\in \mathbb N^{1:n-\ell}, |\alpha_{f^*}| =|\beta| = s,
\end{equation}
which can be proved easily by induction on $s$. When $T$ is the reference simplex $\hat T$, $\lambda_i = x_i$ and $\nabla \lambda_i = - \bs{e}_i$, \eqref{eq:Dbeta} is the calculus result $D^{\beta}_{n_f}\bs x_{f^*}^{\alpha_{f^*}}=\beta!\delta( \alpha_{f^*}, \beta)$. 

Assume $u = \sum c_{\alpha_f, \alpha_{f^*}}  \lambda_f^{\alpha_f}\lambda_{f^*}^{\alpha_{f^*}} \in \mathbb P_k(S_{\ell}(f)\cap L(f,s))$. If the derivative is not fully applied to the component $\lambda_{f^*}^{\alpha_{f^*}}$, then there is a term $\lambda_{f^*}^{\gamma}$ with $|\gamma| > 0$ left and $\lambda_{i}^{\gamma}|_f = 0$ for $i\in f^*$. So for any $\beta \in \mathbb N^{1:n-\ell}$ and $|\beta | = s$,
\begin{equation*}
\frac{\partial^{\beta} u }{\partial n_f^{\beta}} |_f = \beta! \sum_{\alpha\in S_{\ell}(f),|\alpha_f| = k- s} c_{\alpha_f, \beta}  \lambda_f^{\alpha_f}.
\end{equation*}
The vanishing DoF \eqref{eq:normalDof} implies $\sum\limits_{ \alpha\in S_{\ell}(f),|\alpha_f| = k- s} c_{\alpha_f, \beta}  \lambda_f^{\alpha_f}|_f=0$.
Hence $c_{\alpha_f, \beta} = 0$ for all $|\alpha_f| = k- s, \alpha\in S_{\ell}(f)$. As $\beta$ is arbitrary, we conclude all coefficients $ c_{\alpha_f, \alpha_{f^*}} = 0$ and thus $u = 0$.
\end{proof}


For $u\in\mathbb P_k(S_{\ell}(f)\cap L(f,s))$ and $\beta\in\mathbb N^{1:n-\ell}$ with $|\beta|<s$, by Lemma~\ref{lm:derivative}, $\frac{\partial^{\beta} u}{\partial n_f^{\beta}}|_f=0$.
Applying the operator $\frac{\partial^{\beta} (\cdot)}{\partial n_f^{\beta}}|_f$ to the direct decomposition $\mathbb P_k(S_{\ell}(f)) = \Oplus_{s=0}^{r_{\ell}} \mathbb P_k(S_{\ell}(f)\cap L(f,s))$ will possess a block lower triangular structure and leads to the following unisolvence result.

\begin{lemma}\label{lem:appfaceunisolvence}
Let $\ell=0,\ldots, n-1$. The polynomial space $\mathbb P_k(S_{\ell}(f))$ is uniquely determined by DoFs
\begin{equation*}
\int_f  \frac{\partial^{\beta} u}{\partial n_f^{\beta}} \, \lambda_f^{\alpha_f} \dd s \quad \forall~\alpha\in S_{\ell}(f), |\alpha_f| = k - s, \beta \in \mathbb N^{1:n-\ell}, |\beta | = s, s=0,\ldots, r_{\ell}.
\end{equation*}
\end{lemma}

Together with decomposition \eqref{eq:PrSdec} of the polynomial space, we obtain the following result.
\begin{theorem}\label{th:applocalPrCm}
 Given integer $m\geq 0$, let non-negative integer array $\bs r=(r_0,r_1, \cdots, r_n)$ satisfy
$$
r_{n}=0,\;\; r_{n-1}=m,\;\; r_{\ell}\geq 2r_{\ell+1} \; \textrm{ for } \ell=n-2,\ldots, 0.
$$
Let $k\geq 2r_0+1 \geq 2^n m + 1$. Then the shape function $\mathbb P_k(T)$ is uniquely determined by the following DoFs
\begin{align}
\label{eq:C1nd0}
D^{\alpha} u (\texttt{v}) & \quad \alpha \in \mathbb N^{1:n}, |\alpha | \leq  r_0, \texttt{v}\in \Delta_0(T),\\
\label{eq:C1nd2}
\int_f \frac{\partial^{\beta} u}{\partial n_f^{\beta}}  \, \lambda_f^{\alpha_f} \dd s & \quad \alpha\in S_{\ell}(f), |\alpha_f| = k - s, \beta \in \mathbb N^{1:n-\ell}, |\beta | = s,\\
&\quad f\in \Delta_{\ell}(T), \ell =1,\ldots, n-1, s=0,\ldots, r_{\ell}, \notag \\
\label{eq:C1nd3}
\int_T u \lambda^{\alpha} \dx & \quad \alpha \in S_n(T).
\end{align}
\end{theorem}
\begin{proof}
Thanks to the decomposition \eqref{eq:PrSdec}, the dimensions match.
Take $u\in\mathbb P_k(T)$ satisfy all the DoFs \eqref{eq:C1nd0}-\eqref{eq:C1nd3} vanish. We are going to show $u=0$.

For $\alpha\in S_{\ell}(f)$ and $e\in\Delta_i(T)$ with $i\leq\ell$ and $e\neq f$,
by \eqref{eq:Slfineqlty} and \eqref{eq:Slfineqlty2} we have $|\alpha_{e^*}|\geq r_{i}+1$, hence $\frac{\partial^{\beta}\lambda^{\alpha}}{\partial n_e^{\beta}}|_e=0$ for $\beta\in\mathbb N^{1:n-i}$ with $|\beta|\leq r_i$.
Again this tells us that applying the operator $\frac{\partial^{\beta}(\cdot)}{\partial n_f^{\beta}}|_f$ to the direct decomposition $\mathbb P_{k}(T) = \Oplus_{\ell = 0}^{n} \Oplus_{f\in \Delta_{\ell}(T)} \mathbb P_k(S_{\ell}(f))$ will produce a block lower triangular structure. Then apply Lemma~\ref{lem:appfaceunisolvence}, we conclude $u\in\mathbb P_k(S_{n}(T))$, which together with the vanishing DoF \eqref{eq:C1nd3} gives $u=0$.
\end{proof}

\begin{remark}\rm
For $\alpha\in S_{\ell}(f)$, by \eqref{eq:Slfineqlty} we have $|\alpha_{e}|\leq k - r_{\ell-1}-1$ for all $e\in\Delta_{\ell-1}(f)$, then $\alpha_f\geq r_{\ell-1}+1-|\alpha_{f^*}|$, and
$$
\lambda^{\alpha}=\lambda_{f^*}^{\alpha_{f^*}}\lambda_f^{\alpha_f}=\lambda_{f^*}^{\alpha_{f^*}}\lambda_f^{r_{\ell-1}+1-|\alpha_{f^*}|}\lambda_f^{\alpha_f-(r_{\ell-1}+1)+|\alpha_{f^*}|}.
$$
Using $\alpha_f-(r_{\ell-1}+1)+|\alpha_{f^*}|$ as the new index, DoFs \eqref{eq:C1nd2}-\eqref{eq:C1nd3} can be replaced by
\begin{align*}
\int_f \frac{\partial^{\beta} u}{\partial n_f^{\beta}}  \, \lambda_f^{\alpha} \dd s & \quad \beta \in \mathbb N^{1:n-\ell}, |\beta | = s, s=0,\ldots, r_{\ell}, \;\;\alpha\in\mathbb T_{k-(\ell+1)(r_{\ell-1}+1)+\ell s}^{\ell}, \\
&\quad |\alpha_{e}|\leq k - r_{i}-1-(i+1)(r_{\ell-1}+1-s), \forall e\in\Delta_i(f), i=0,\ldots,\ell-2, \\
&\quad f\in \Delta_{\ell}(T), \ell =1,\ldots, n-1, \notag \\
\int_T u \lambda^{\alpha} \dx & \quad \alpha \in \mathbb T_{k-(n+1)(m+1)}^{n},\\
&\quad |\alpha_{e}|\leq k - r_{i}-1-(i+1)(m+1), \forall e\in\Delta_i(T), i=0,\ldots,n-2.
\end{align*}
Namely we can remove bubble functions in the test function space. 
\end{remark}

\subsection{Smooth finite elements in arbitrary dimension}
Given a triangulation $\mathcal T_h$, the finite element space is obtained by asking the DoFs depending on the sub-simplex only.

\begin{theorem}\label{th:appC1Rn}
 Given integer $m\geq 0$, let non-negative integer array $\bs r=(r_0,r_1, \cdots, r_n)$ satisfy
$$
r_{n}=0,\;\; r_{n-1}=m,\;\; r_{\ell}\geq 2r_{\ell+1} \; \textrm{ for } \ell=n-2,\ldots, 0.
$$
Let $k\geq 2r_0+1 \geq 2^n m + 1$. The following DoFs
\begin{align}
\label{eq:C1Rnd0}
D^{\alpha} u (\texttt{v}) & \quad \alpha \in \mathbb N^{1:n}, |\alpha | \leq  r_0, \texttt{v}\in \Delta_0(\mathcal T_h),\\
\label{eq:C1Rndf}
\int_f \frac{\partial^{\beta} u}{\partial n_f^{\beta}}  \, \lambda_f^{\alpha_f} \dd s & \quad \alpha\in S_{\ell}(f), |\alpha_f| = k - s, \beta \in \mathbb N^{1:n-\ell}, |\beta | = s, s=0,\ldots, r_{\ell},\\
&\quad f\in \Delta_{\ell}(\mathcal T_h), \ell =1,\ldots, n-1, \notag \\
\label{eq:C1RndT}
\int_T u \lambda^{\alpha} \dx & \quad \alpha \in S_n(T), T\in \mathcal T_h,
\end{align}
will define a finite element space
$$
V_h = \{ u \in C^{m}(\Omega) \mid \text{ DoFs }\eqref{eq:C1Rnd0}-\eqref{eq:C1Rndf} \text{ are single valued}, u|_{T}\in \mathbb P_k(T), \forall T\in \mathcal T_h \}.
$$
\end{theorem}
\begin{proof}
Restricted to one simplex $T$, by Theorem \ref{th:applocalPrCm}, DoFs \eqref{eq:C1Rnd0}-\eqref{eq:C1RndT}  will define a function $u$ s.t. $ u|_{T}\in \mathbb P_k(T)$. We only need to verify $u\in C^{m}(\Omega)$. 
It suffices to prove $\frac{\partial^{i} u}{\partial n_F^{i}}|_F \in \mathbb P_{k-i}(F)$, for all $i=0,\ldots, m$ and all $F\in \Delta_{n-1}(T)$, are uniquely determined by \eqref{eq:C1Rnd0}-\eqref{eq:C1Rndf} on $F$. 

Let $w=\frac{\partial^{i} u}{\partial n_F^{i}}|_F\in \mathbb P_{k-i}(F)$. Consider the modified index sequence $\bs r_F^{i} = (r_0 - i, r_1 - i, \ldots, r_{n-2} - i, 0)$ and degree $k^i = k - i$. Then $k^i, \bs r_F^{i}$ satisfies the condition in Theorem \ref{th:decT} and we obtain a direct decomposition of $\mathbb T^{n-1}_{k-i}(F) = \Oplus_{\ell = 0}^{n-1}\Oplus_{f\in \Delta_{\ell}(F)} S_{\ell}^F(f),$ where
\begin{align*}
S_0^F(\texttt{v}) &=  D(\texttt{v}, r_0-i)\cap \mathbb T^{n-1}_{k-i}(F), \\
S_{\ell}^F(f) &= (D(f, r_{\ell}-i)\cap \mathbb T^{n-1}_{k-i}(F)) \backslash \left [  \Oplus_{i = 0}^{\ell - 1}\Oplus_{e\in \Delta_{i}(F)} S_{i}^F(e) \right ], \; \ell = 1,\dots, n-2, \\
S_{n-1}^F(F) & = \mathbb T^{n-1}_{k-i}(F) \backslash \left [  \Oplus_{\ell = 0}^{n-2}\Oplus_{f\in \Delta_{\ell}(F)} S_{\ell}^F(f)\right ].
\end{align*}
The DoFs \eqref{eq:C1Rnd0}-\eqref{eq:C1Rndf} related to $w$ are
\begin{align*}
D_F^{\alpha} w (\texttt{v}) & \quad \alpha \in \mathbb N^{1:n-1}, |\alpha | \leq  r_0-i, \texttt{v}\in \Delta_0(F),\\
\int_f \frac{\partial^{\beta} w}{\partial n_{F,f}^{\beta}}  \, \lambda_f^{\alpha_f} \dd s & \quad \alpha\in S_{\ell}^F(f), |\alpha_f| = k-i - s, \beta \in \mathbb N^{1:n-1-\ell}, |\beta| = s, \\
&\quad f\in \Delta_{\ell}(F), \ell =1,\ldots, n-2, s=0,\ldots, r_{\ell}-i, \notag \\
\int_F w \lambda^{\alpha} \dx & \quad \alpha \in S_{n-1}^F(F),
\end{align*}
where $D_Fw$ is the tangential derivatives of $w$, $n_{F,f}$ is the normal vector of $f$ but tangential to $F$.
Clearly the modified sequence $\bs r^F_{i}$ still satisfies constraints required in Theorem \ref{th:applocalPrCm}. We can apply Theorem \ref{th:applocalPrCm} with the shape function space $\mathbb P_{k-i}(F)$ to conclude $w$ is uniquely determined on $F$. Thus the result $u\subset C^m(\Omega)$ follows.
\end{proof}

Counting the dimension of $V_h$ is hard and not necessary. The cardinality of $S_{\ell}(f)$ is difficult to determine due to the inequality constraints. In the implementation, compute the distance of lattice nodes to sub-simplexes and use a logic array to find out $S_{\ell}(f)$. 

\bibliographystyle{abbrv}
\bibliography{paper,refgeodecomp}

\begin{thebibliography}{10}

\bibitem{ArgyrisFriedScharpf1968}
J.~Argyris, I.~Fried, and D.~Scharpf.
\newblock {The TUBA family of plate elements for the matrix displacement
  method}.
\newblock {\em Aero. J. Roy. Aero. Soc.}, 72:701--709, 1968.

\bibitem{Arnold;Falk;Winther:2010Finite}
D.~Arnold, R.~Falk, and R.~Winther.
\newblock Finite element exterior calculus: from {H}odge theory to numerical
  stability.
\newblock {\em Bull. Amer. Math. Soc. (N.S.)}, 47(2):281--354, 2010.

\bibitem{Arnold:2018Finite}
D.~N. Arnold.
\newblock {\em Finite element exterior calculus}, volume~93 of {\em CBMS-NSF
  Regional Conference Series in Applied Mathematics}.
\newblock Society for Industrial and Applied Mathematics (SIAM), Philadelphia,
  PA, 2018.

\bibitem{ArnoldFalkWinther2006}
D.~N. Arnold, R.~S. Falk, and R.~Winther.
\newblock Finite element exterior calculus, homological techniques, and
  applications.
\newblock {\em Acta Numer.}, 15:1--155, 2006.

\bibitem{ArnoldFalkWinther2009}
D.~N. Arnold, R.~S. Falk, and R.~Winther.
\newblock Geometric decompositions and local bases for spaces of finite element
  differential forms.
\newblock {\em Computer Methods in Applied Mechanics and Engineering},
  198(21-26):1660--1672, 2009.

\bibitem{BrambleZlamal1970}
J.~H. Bramble and M.~Zl\'amal.
\newblock Triangular elements in the finite element method.
\newblock {\em Math. Comp.}, 24:809--820, 1970.

\bibitem{BrezziDouglasDuranFortin1987}
F.~Brezzi, J.~Douglas, Jr., R.~Dur{\'a}n, and M.~Fortin.
\newblock Mixed finite elements for second order elliptic problems in three
  variables.
\newblock {\em Numer. Math.}, 51(2):237--250, 1987.

\bibitem{BrezziDouglasMarini1986}
F.~Brezzi, J.~Douglas, Jr., and L.~D. Marini.
\newblock Recent results on mixed finite element methods for second order
  elliptic problems.
\newblock In {\em Vistas in applied mathematics}, Transl. Ser. Math. Engrg.,
  pages 25--43. Optimization Software, New York, 1986.

\bibitem{ChenHuang2018}
L.~Chen and X.~Huang.
\newblock Decoupling of mixed methods based on generalized {H}elmholtz
  decompositions.
\newblock {\em SIAM J. Numer. Anal.}, 56(5):2796--2825, 2018.

\bibitem{Chen;Huang:2021Geometric}
L.~Chen and X.~Huang.
\newblock Geometric decomposition of div-conforming finite element tensors.
\newblock {\em arXiv preprint arXiv:2112.14351}, 2021.

\bibitem{Chen;Huang:2022femcomplex2d}
L.~Chen and X.~Huang.
\newblock Finite element complexes in two dimensions.
\newblock {\em Preprint}, 2022.

\bibitem{Christiansen;Hu;Hu:2018finite}
S.~H. Christiansen, J.~Hu, and K.~Hu.
\newblock Nodal finite element de {R}ham complexes.
\newblock {\em Numer. Math.}, 139(2):411--446, 2018.

\bibitem{ChristiansenHu2018}
S.~H. Christiansen and K.~Hu.
\newblock Generalized finite element systems for smooth differential forms and
  {S}tokes' problem.
\newblock {\em Numer. Math.}, 140(2):327--371, 2018.

\bibitem{Chui;Lai:1990Multivariate}
C.~K. Chui and M.-J. Lai.
\newblock Multivariate vertex splines and finite elements.
\newblock {\em Journal of approximation theory}, 60(3):245--343, 1990.

\bibitem{FalkNeilan2013}
R.~S. Falk and M.~Neilan.
\newblock Stokes complexes and the construction of stable finite elements with
  pointwise mass conservation.
\newblock {\em SIAM J. Numer. Anal.}, 51(2):1308--1326, 2013.

\bibitem{FuGuzmanNeilan2020}
G.~Fu, J.~Guzm\'{a}n, and M.~Neilan.
\newblock Exact smooth piecewise polynomial sequences on {A}lfeld splits.
\newblock {\em Math. Comp.}, 89(323):1059--1091, 2020.

\bibitem{huConstructionConformingFinite2021}
J.~Hu, T.~Lin, and Q.~Wu.
\newblock A construction of ${C}^r$ conforming finite element spaces in any
  dimension.
\newblock {\em arXiv:2103.14924}, 2021.

\bibitem{HuZhangZhang2022}
K.~Hu, Q.~Zhang, and Z.~Zhang.
\newblock A family of finite element {S}tokes complexes in three dimensions.
\newblock {\em SIAM J. Numer. Anal.}, 60(1):222--243, 2022.

\bibitem{Lai;Schumaker:2007Spline}
M.-J. Lai and L.~L. Schumaker.
\newblock {\em Spline functions on triangulations}, volume 110.
\newblock Cambridge University Press, 2007.

\bibitem{Lai;Schumaker:2007Trivariate}
M.-J. Lai and L.~L. Schumaker.
\newblock Trivariate ${C}^r$ polynomial macroelements.
\newblock {\em Constructive Approximation}, 26(1):11--28, 2007.

\bibitem{Neilan2015}
M.~Neilan.
\newblock Discrete and conforming smooth de {R}ham complexes in three
  dimensions.
\newblock {\em Math. Comp.}, 84(295):2059--2081, 2015.

\bibitem{Stenberg2010}
R.~Stenberg.
\newblock A nonstandard mixed finite element family.
\newblock {\em Numer. Math.}, 115(1):131--139, 2010.

\bibitem{Zenisek1970}
A.~\v{Z}en\'\i \v{s}ek.
\newblock Interpolation polynomials on the triangle.
\newblock {\em Numer. Math.}, 15:283--296, 1970.

\bibitem{Zenisek1974a}
A.~\v{Z}en\'{\i}\v{s}ek.
\newblock Tetrahedral finite {$C^{(m)}$}-elements.
\newblock {\em Acta Univ. Carolinae---Math. et Phys.}, 15(1-2):189--193, 1974.

\bibitem{ZhangZhang2020}
Q.~Zhang and Z.~Zhang.
\newblock A family of curl-curl conforming finite elements on tetrahedral
  meshes.
\newblock {\em CSIAM Transactions on Applied Mathematics}, 1(4):639--663, 2020.

\bibitem{Zhang2009a}
S.~Zhang.
\newblock A family of 3{D} continuously differentiable finite elements on
  tetrahedral grids.
\newblock {\em Appl. Numer. Math.}, 59(1):219--233, 2009.

\bibitem{Zhang2016a}
S.~Zhang.
\newblock A family of differentiable finite elements on simplicial grids in
  four space dimensions.
\newblock {\em Mathematica Numerica Sinica}, 38(3):309--324, 2016.

\bibitem{ZhengHuXu2011}
B.~Zheng, Q.~Hu, and J.~Xu.
\newblock A nonconforming finite element method for fourth order curl equations
  in {$\Bbb{R}^{3}$}.
\newblock {\em Math. Comp.}, 80(276):1871--1886, 2011.

\end{thebibliography}
\end{document}